\RequirePackage{fix-cm}
\documentclass[envcountsect,smallcondensed]{svjour3}                     
\smartqed
\usepackage{etoolbox}
\AtEndEnvironment{proof}{\phantom{}\qed}

\usepackage{graphicx}
\usepackage[numbers,sort&compress]{natbib}
\usepackage[T1]{fontenc}
\usepackage[colorlinks=true, linkcolor=blue,urlcolor=blue, citecolor = blue, pdfusetitle]{hyperref}
\usepackage{verbatim}
\usepackage{enumitem}
\usepackage{float}
\usepackage{amsmath}
\usepackage{microtype}
\usepackage{svg}
\usepackage{multicol}
\usepackage{multirow}
\usepackage{makecell}
\usepackage{booktabs}
\usepackage{tabularx,ragged2e}
\usepackage{array}
\usepackage{physics}
\usepackage{mathtools}
\usepackage{amssymb}
\usepackage{caption}
\usepackage{subcaption}
\usepackage{lmodern}
\usepackage{marvosym}
\usepackage[nottoc,notlot,notlof]{tocbibind}
\usepackage[title,titletoc]{appendix}

\newcolumntype{H}{>{\setbox0=\hbox\bgroup}c<{\egroup}@{}}
\newcommand{\bunderline}[1]{\underline{#1\mkern-3mu}\mkern3mu }
\newcommand{\etal}{\emph{et al.}\ }
\newcommand{\df}{:=}
\newcommand{\T}{\textrm{T}}

\makeatletter
\DeclareRobustCommand{\onontimes}{%
	\mathbin{\mathpalette\on@ntimes\relax}%
}
\newcommand{\on@ntimes}[2]{%
	\vcenter{\hbox{%
			\sbox0{\m@th$#1\otimes$}%
			\setlength\unitlength{\wd0}%
			\begin{picture}(1,1)
			\linethickness{0.35pt}
			\put(.5,.5){\circle{.8}}
			\end{picture}%
	}}%
}
\makeatother
\newcommand{\vc}[1]{\underline{#1}}
\newcommand{\fn}[1]{\MakeUppercase{#1}}
\newcommand{\el}[1]{#1}
\newcommand{\fv}[1]{\underline{\fn{#1}}} 
\newcommand{\mat}[1]{\bunderline{\bunderline{#1}}}
\newcommand{\zeromat}{\vc{\vc{0}}}
\newcommand{\pdvln}[2]{\partial{#1} / \partial{#2}}
\newcommand{\set}[1]{\mathcal{\MakeUppercase{#1}}} 
\newcommand{\opr}[1]{\mathcal{\MakeUppercase{#1}}}
\newcommand{\R}{\mathbb{R}}
\newcommand{\xt}{(\vc{x},t)}
\newcommand{\xit}{(\hat{\vc{x}},t)}
\newcommand{\il}{\langle}
\newcommand{\ir}{\rangle}
\newcommand{\ug}[1]{\mathrm{#1}}

\let\oldnabla\nabla
\renewcommand{\nabla}{\vc{\oldnabla}}
\renewcommand{\textrm}[1]{\text{\textup{#1}}}

\spnewtheorem{assumption}{Assumption}[section]{\bfseries}{}
\numberwithin{equation}{section}
\begin{document}
	
\setcounter{tocdepth}{2}

\title{A unifying algebraic framework for discontinuous Galerkin and flux reconstruction methods based on the summation-by-parts property}
\titlerunning{Unifying algebraic framework for DG and FR methods based on the SBP property}      
\author{\normalfont Tristan Montoya \and David W.\ Zingg }

\institute{T. Montoya\at
		University of Toronto Institute for Aerospace Studies, Toronto, Canada\\
              \email{tristan.montoya@mail.utoronto.ca} 
           \and
           D. W. Zingg \at
          	University of Toronto Institute for Aerospace Studies, Toronto, Canada\\
          	\email{dwz@utias.utoronto.ca}
}

\newcommand{\ignore}[1]{}
\newcommand{\nobibentry}[1]{{\let\nocite\ignore\bibentry{#1}}}

\date{May 26, 2021}

\maketitle
\begin{abstract}
We propose a unifying framework for the matrix-based formulation and analysis of discontinuous Galerkin (DG) and flux reconstruction (FR) methods for conservation laws on general unstructured grids. Within such an algebraic framework, the multidimensional summation-by-parts (SBP) property is used to establish the discrete equivalence of strong and weak formulations, as well as the conservation and energy stability properties of a broad class of DG and FR schemes. Specifically, the analysis enables the extension of the equivalence between the strong and weak forms of the discontinuous Galerkin collocation spectral-element method demonstrated by Kopriva and Gassner (J Sci Comput 44:136--155, 2010) to more general nodal and modal DG formulations, as well as to the Vincent-Castonguay-Jameson-Huynh (VCJH) family of FR methods. Moreover, new algebraic proofs of conservation and energy stability for DG and VCJH schemes with respect to suitable quadrature rules and discrete norms are presented, in which the SBP property serves as a unifying mechanism for establishing such results. Numerical experiments are provided for the two-dimensional linear advection and Euler equations, highlighting the design choices afforded for methods within the proposed framework and corroborating the theoretical analysis. 
\end{abstract}
\section{Introduction}
Hyperbolic and convection-dominated systems of conservation laws are of considerable importance in the mathematical modelling of physical phenomena, with numerous scientific and engineering disciplines relying heavily on efficient and robust numerical methods for the solution of such partial differential equations. High-order methods (i.e.\ those of third order or higher) have shown significant promise for problems involving the propagation of waves over long distances (an early example being the work of Kreiss and Oliger \cite{kreiss_oliger_72}) and the fine-scale resolution of turbulent flow structures (as surveyed by Wang \etal \cite{wangreview13}), both of which incur prohibitive computational expense when conventional second-order spatial discretizations are used. The most popular high-order methods in current use by practitioners for such problems are arguably the well-established discontinuous Galerkin (DG) methods and the more recent flux reconstruction (FR) schemes. DG methods were originally introduced for the steady neutron transport equation by Reed and Hill \cite{reed_hill_dg_73}, and have since been applied successfully to a wide range of problems following extensive development by Cockburn, Shu, and collaborators (see, for example, \cite{cockburn_shu_general_framework_89,cockburnShuSystems_89,cockburn_shu_multidimensional_90,cockburn_shu_dg_multidimensional_systems_98}). The FR approach was first proposed by Huynh \cite{huynh_FR_07} for one-dimensional conservation laws, extended to triangular elements by Wang and Gao \cite{wang_gao_lcp_09} through the so-called lifting collocation penalty (LCP) formulation, and further developed by several other authors, including Vincent \etal \cite{vincent_esfr_10}, Castonguay \etal \cite{castonguay_vincent_jameson_triangular_fr_11}, and Williams and Jameson \cite{williams_esfr_adv_diff_tetrahedra_13}, who identified energy-stable families of FR schemes on one-dimensional, triangular, and tetrahedral elements, respectively. 

\par
Although DG methods are derived from a weak (i.e.\ variational) formulation whereas FR methods are derived from a strong (i.e.\ differential) formulation, the two approaches are closely related, as discussed in Huynh's seminal paper \cite{huynh_FR_07}, as well as in further investigations by Allaneau and Jameson \cite{allaneau_jameson_dg_fr_11}, De Grazia \etal \cite{degrazia_connections_dg_fr_14}, Mengaldo \etal \cite{mengaldo_degrazia_fr_dg_connections_curvilinear_16}, and Zwanenburg and Nadarajah \cite{zwanenburg_esfr_filtered_dg_16}. Moreover, certain DG and FR schemes have been recast in terms of summation-by-parts (SBP) operators, which are discrete differential operators equipped with compatible inner products such that integration by parts is mimicked algebraically (see, for example, the review papers of Del Rey Fern\'andez \etal \cite{delrey_sbp_sat_review_14} and Sv\"ard and Nordstr\"om \cite{svard_nordstrom_sbpreview_14}). Beginning with the work of Kreiss and Scherer \cite{kreiss_scherer_sbp_74}, this mimetic property, referred to as the SBP property, has been instrumental in the development and analysis of high-order finite-difference methods with discrete stability and conservation properties, and there have been significant advances in recent years towards extending the SBP methodology to encompass a wider range of new and existing methods, facilitated by the one-dimensional and multidimensional generalizations of Del Rey Fern\'andez \etal \cite{delrey_generalized_framework_14} and Hicken \etal \cite{hicken_mdsbp_16}, respectively.
\par
Due to its reliance on discretization-agnostic matrix properties, the SBP approach allows for existing theoretical results established for particular numerical methods to be reinterpreted within a more general and arguably simpler setting, and facilitates the cross-pollination of techniques between different schemes sharing common algebraic properties. This cross-pollination has proved fruitful over the past decade, with Gassner \cite{gassner_dgsem_sbp_13} and Ranocha \etal \cite{ranocha_sbp_cpr_16} employing the SBP property in order to develop nonlinearly stable DG and FR methods, respectively, for Burgers' equation in one space dimension using skew-symmetric split forms originally introduced for finite-difference discretizations. Motivated by these contributions, as well as subsequent extensions to entropy-stable and kinetic-energy-preserving DG-type schemes exploiting the SBP property (see, for example, Gassner \etal \cite{gassner_winters_kopriva_splitform_nodaldg_sbp_16} and Chan \cite{chan_discretely_entropy_conservative_dg_sbp_18}), the goal of this work is to develop a comprehensive framework enabling the unification of existing algebraic techniques for the analysis of semi-discrete DG and FR methods applied to time-dependent conservation laws and the extension of the existing theory based on the SBP property to more general DG and FR formulations. To this end, the analysis has resulted in the following novel contributions.

\begin{enumerate}
	\item The discrete equivalence of strong and weak formulations, demonstrated by Kopriva and Gassner \cite{kopriva_nodaldg_choosing_quadrature_weakform_10} in the context of DG methods employing collocated tensor-product Legendre-Gauss (LG) and Legendre-Gauss-Lobatto (LGL) quadrature rules (i.e.\ the LG and LGL variants discontinuous Galerkin collocation spectral-element method, which we refer to as the DGSEM-LG and DGSEM-LGL, respectively), is extended to more general quadrature-based and collocation-based DG formulations by way of the SBP property.
	\item The Vincent-Castonguay-Jameson-Huynh (VCJH) family of FR schemes (as described in \cite{vincent_esfr_10,castonguay_vincent_jameson_triangular_fr_11,williams_esfr_adv_diff_tetrahedra_13}), is recast in a generalized matrix form applicable to general element types, allowing for the use of modal bases and over-integration\footnote{In this paper, over-integration (also referred to as polynomial de-aliasing, and discussed further by Kirby and Karniadakis \cite{kirby_dealiasing_03} and Mengaldo \etal \cite{mengaldo_dg_fr_dealiasing_15}) refers to the use of a larger number of quadrature nodes on an element or facet than the dimension of the corresponding local approximation space.} of the volume and facet terms. Through the SBP property, these formulations are shown to be equivalent to filtered DG schemes in strong or weak form, in the latter case leading to methods which are algebraically equivalent to conventional strong-form FR methods while offering the potential for improved computational efficiency.
	\item Proofs of discrete conservation with respect to suitable quadrature rules are presented for DG and VCJH methods in strong and weak form, differing from the existing conservation proofs for the FR approach due to their reliance on the discrete divergence theorem (which holds as a consequence of the SBP property) rather than the continuity of the corrected flux, which is not explicitly reconstructed for multidimensional FR schemes on non-tensor-product elements.
	\item Energy stability (with respect to suitable discrete norms) is established as a consequence of the SBP property for DG and FR schemes applied to constant-coefficient linear advection problems on affine polytopal meshes with periodic or inflow-outflow boundary conditions, recovering the existing stability results for VCJH schemes (i.e.\ those in \cite[]{vincent_esfr_10,castonguay_vincent_jameson_triangular_fr_11,williams_esfr_adv_diff_tetrahedra_13}) as special cases of a more general theory.
\end{enumerate}
We now outline the content of the remainder of this paper. In \textsection \ref{sec:prelims}, we
describe the systems of conservation laws to which we seek numerical solutions, as well as the meshes and approximation spaces employed for the discretizations considered in this work; we then provide a brief review of the DG and FR approaches. In \textsection \ref{sec:formulation}, we develop a unifying matrix formulation for a broad class of DG and FR methods, which we employ in \textsection \ref{sec:analysis} for the analysis of such schemes based on the SBP property. In \textsection \ref{sec:numerical}, we present numerical experiments for the two-dimensional linear advection and Euler equations in support of the theory. Concluding remarks are provided in \textsection \ref{sec:conclusions}.

\section{Preliminaries}\label{sec:prelims}
\subsection{Notation}
Where possible, the symbols in this paper follow those used by Ranocha \etal \cite{ranocha_sbp_cpr_16} and Chan \cite{chan_discretely_entropy_conservative_dg_sbp_18}. As in the former, we use a single underline to denote vectors (treated as column matrices), and a double underline to denote matrices. Entries of matrices and vectors are indexed using subscripts, while elements of sequences are typically indexed using superscripts in parentheses. The symbols $\vc{0}$, $\vc{1}$, $\zeromat$, and $\mat{I}$ are reserved for vectors of zeros, vectors of ones, matrices of zeros, and identity matrices, respectively, with dimensions inferred from the contexts in which they appear.
\subsection{Problem formulation}
We consider a first-order system of time-dependent conservation laws governing the evolution of the $N_\textrm{Eq} \in \mathbb{N}$ solution variables $\fv{u}\xt \in \ug{\Upsilon} \subseteq \R^{N_{\textrm{Eq}}}$, where $\ug{\Upsilon}$ denotes the set of admissible solution states. Such a system of partial differential equations takes the form
\begin{equation}\label{eq:pde}
\begin{alignedat}{2}
\pdv{\fv{u}\xt}{t}+ \sum_{m=1}^d\pdv{\fv{f}^{(m)}(\fv{u}\xt)}{x_m}  &= \vc{0}, && \forall \,(\vc{x},t) \in \ug{\Omega} \times (0, T),\\
\fv{u}(\vc{x},0) &= \fv{u}^0(\vc{x}),   \qquad &&\forall \, \vc{x} \in \ug{\Omega},
\end{alignedat}
\end{equation}
subject to appropriate boundary conditions, if necessary, where $\ug{\Omega} \subset \R^d$ is a domain of dimension $d \in \mathbb{N}$ with a piecewise smooth boundary,  $T \in \R_+$ is the final time, and $\fv{f}^{(m)} : \ug{\Upsilon} \to \R^{N_\textrm{Eq}}$ are the flux components, which we define for $m \in \{1,\ldots, d\}$. For future reference, we also define the outward unit normal vector to the domain boundary as $\vc{n} : \partial\ug{\Omega} \to \mathbb{S}^{d-1}$,
where $\mathbb{S}^{d-1} \df \{\vc{x} \in \R^d : \lVert \vc{x} \rVert_2 = 1\}$ denotes the unit $(d-1)$-sphere and $\lVert \vc{x} \rVert_2 \df \sqrt{x_1^2 + \cdots + x_d^2}$ is the standard Euclidean norm.  \par
While the present framework applies to a very broad class of problems which may be formulated as in (\ref{eq:pde}), we highlight two representative examples, which serve as test cases for the numerical experiments in \textsection \ref{sec:numerical}. We first consider the linear advection equation, which governs the transport of a scalar quantity $\fn{u}\xt \in \R$ at a velocity $\vc{a} \in \R^d$, assumed here to be constant, and is given in conservation form by
\begin{equation}\label{eq:advection}
\pdv{\fn{u}\xt}{t} + \sum_{m=1}^d \pdv{(a_m\fn{u}\xt)}{x_m} = 0.
\end{equation}
We also consider the Euler equations, which govern the conservation of mass, momentum, and energy for a compressible, inviscid, adiabatic fluid, given as a system of $N_{\mathrm{Eq}} = d+2$ equations in the form of (\ref{eq:pde}) by
\begin{equation}\label{eq:euler}
\pdv{t}\underbrace{\mqty[\rho\xt \\ \rho\xt \fn{v}_1\xt \\ \vdots \\ \rho\xt\fn{v}_d\xt \\ \fn{e}\xt]}_{=:\, \fv{u}\xt} + \sum_{m=1}^d \pdv{x_m}\underbrace{\mqty[\rho\xt\fn{v}_m\xt \\ \rho\xt\fn{v}_1\xt\fn{v}_m\xt  + \fn{p}\xt\delta_{1m} \\ \vdots \\ \rho\xt\fn{v}_d\xt\fn{v}_m\xt + \fn{p}\xt\delta_{dm}  \\ \fn{v}_m\xt(\fn{e}\xt + \fn{p}\xt)]}_{=:\, \fv{f}^{(m)}(\fv{u}\xt)} = \vc{0}.
\end{equation}
In the above, $\rho\xt \in \R$ denotes the fluid density, $\fv{v}\xt \in \R^d$ denotes the flow velocity, $\fn{e}\xt \in \R$ denotes the total energy per unit volume, and $\fn{p}\xt \in \R$ denotes the pressure, which is related to the other variables through the equation of state, which is given by $\fn{p}\xt = (r-1)\qty(\fn{e}\xt - \frac{1}{2}\rho\xt\lVert \fv{v}\xt\rVert_2^2)$ for an ideal gas with constant specific heat and specific heat ratio $r$. The set of admissible solution states is therefore given by $\ug{\Upsilon} \df \{ \fv{u}\xt \in \R^{N_{\textrm{Eq}}+2} : \fn{p}\xt,\rho\xt > 0 \}$.
\begin{remark}
Throughout this paper, we tacitly assume that the values of numerical solutions to (\ref{eq:pde}) remain within $\ug{\Upsilon}$. This is difficult to ensure \emph{a priori} for nonlinear problems, and often requires the use of specially designed limiting procedures (see, for example, Zhang and Shu \cite{zhang_shu_max_principle_positivity_high_order_review_11}), which are not considered in this work.
\end{remark}
\subsection{Mesh and coordinate transformation}\label{sec:mesh}
As is standard for finite-element methods (see, for example, Ern and Guermond \cite[\textsection 1.3]{ern_guermond_theory_practice_fe}), we consider a mesh $\set{T}_h \df \{\ug{\Omega}^{(\kappa)}\}_{\kappa=1}^K$ consisting of $K \in \mathbb{N}$ compact elements of characteristic size $h \in \R_+$ with disjoint interiors, satisfying $\cup_{\kappa=1}^K \ug{\Omega}^{(\kappa)} = \bar{\ug{\Omega}}$. It is assumed that the boundary of each element is composed of a finite number of smooth facets,\footnote{For simplicity of notation, it is assumed that all elements are of the same type and therefore have an equal number of facets, which we denote as $N_f \in \mathbb{N}$.} which are denoted by $\ug{\Gamma}^{(\kappa,\gamma)} \subset \partial\ug{\Omega}^{(\kappa)}$ and indexed as $\gamma \in \{1,\ldots, N_f\}$, each with an outward unit normal vector given by $\vc{n}^{(\kappa,\gamma)} : \ug{\Gamma}^{(\kappa,\gamma)} \to \mathbb{S}^{d-1}$. We further assume that each element is the image of a polytopal reference element $\hat{\ug{\Omega}} \subset \R^d$ under a smooth, time-invariant mapping $\fv{X}^{(\kappa)} :\hat{\ug{\Omega}} \rightarrow \ug{\Omega}^{(\kappa)}$, where the Jacobian of the transformation is given by $\mat{\mathcal{J}}^{(\kappa)} : \hat{\ug{\Omega}} \to \R^{d\times d}$ such that $\mathcal{J}_{mn}^{(\kappa)}(\hat{\vc{x}}) \df \pdvln{\fn{X}_m^{(\kappa)}(\hat{\vc{x}})}{\hat{x}_n}$ with $\mathcal{J}^{(\kappa)}(\hat{\vc{x}}) \df \mathrm{det}(\mat{\mathcal{J}}^{(\kappa)}(\hat{\vc{x}})) > 0$, thereby ensuring that the mapping is bijective and orientation preserving. Each facet $\hat{\ug{\Gamma}}^{(\gamma)} \subset \partial\hat{\ug{\Omega}}$ of the reference element is then assumed to be flat, mapping onto the (potentially curved) physical facet $\ug{\Gamma}^{(\kappa,\gamma)}$ under $\fv{X}^{(\kappa)}$, 
with the corresponding outward unit normal vector $\hat{\vc{n}}^{(\gamma)} \in \mathbb{S}^{d-1}$ transforming as
\begin{equation}\label{eq:nanson}
\vc{n}^{(\kappa,\gamma)}(\fv{X}^{(\kappa)}(\hat{\vc{x}})) = \frac{\mathcal{J}^{(\kappa)}(\hat{\vc{x}})(\mat{\mathcal{J}}^{(\kappa)}(\hat{\vc{x}}))^{-\T} \hat{\vc{n}}^{(\gamma)}}{\lVert\mathcal{J}^{(\kappa)}(\hat{\vc{x}})(\mat{\mathcal{J}}^{(\kappa)}(\hat{\vc{x}}))^{-\T} \hat{\vc{n}}^{(\gamma)} \rVert_2},
\end{equation}
where we note that such a transformation is well known in the continuum mechanics literature (see, for example, Gurtin \etal \cite[\textsection 8.1]{gurtin_mechanics_continua}) and define $\mathcal{J}^{(\kappa,\gamma)}(\hat{\vc{x}}) \df \lVert\mathcal{J}^{(\kappa)}(\hat{\vc{x}})(\mat{\mathcal{J}}^{(\kappa)}(\hat{\vc{x}}))^{-\T}\hat{\vc{n}}^{(\gamma)} \rVert_2$ for future reference. With such definitions and assumptions in place, we recall from Vinokur \cite{vinokur_gasdynamics_curvilinear_74} that systems in the form of (\ref{eq:pde}) retain a conservative form when mapped to the reference element, given in the present context by
\begin{equation}\label{eq:pde_ref} 
\begin{multlined}
\pdv{\mathcal{J}^{(\kappa)}(\hat{\vc{x}})\fv{u}(\fv{X}^{(\kappa)}(\hat{\vc{x}}),t)}{t} \\ +  \sum_{m=1}^d\pdv{\hat{x}_m}\Bigg(\sum_{n=1}^d \mathcal{J}^{(\kappa)}(\hat{\vc{x}})[(\mat{\mathcal{J}}^{(\kappa)}(\hat{\vc{x}}))^{-1}]_{mn} \fv{f}^{(n)}(\fv{u}(\fv{X}^{(\kappa)}(\hat{\vc{x}}),t))\Bigg)  = \vc{0}.
\end{multlined}
\end{equation}
Since periodic boundaries may be treated identically to interior interfaces for systems formulated as in (\ref{eq:pde_ref}), we hereinafter use the symbol $\partial\ug{\Omega}$ to refer only to non-periodic portions of the boundary.
For the purposes of our analysis, we may then define the sets of element-facet index pairs $\set{K} \df \{1,\ldots, K\} \times \{1,\ldots, N_f \} $, $\set{K}_{\partial\ug{\Omega}} \df \{(\kappa,\gamma) \in \set{K} : \ug{\Gamma}^{(\kappa,\gamma)} \subset \partial\ug{\Omega} \}$, and $\set{K}_{\ug{\Omega}} \df \set{K} \setminus \set{K}_{\partial\ug{\Omega}}$, with the subsets $\set{K}_{\partial\ug{\Omega}}$ and $\set{K}_{\ug{\Omega}}$ corresponding to facets lying on the non-periodic portion of the domain boundary, and those forming interior or periodic interfaces, respectively.
\subsection{Polynomial approximation space}\label{sec:approx_space}
Considering the transformed system in (\ref{eq:pde_ref}), we seek a semi-discrete approximation $\fv{u}^{(h,\kappa)}\xit$ of $\fv{u}(\fv{X}^{(\kappa)}(\hat{\vc{x}}),t)$ such that for all $\kappa \in \{1,\ldots, K\}$, $e \in \{1,\ldots, N_{\textrm{Eq}}\}$, and $t \in (0,T)$, each component $\fn{u}_e^{(h,\kappa)}(\cdot,t)$ of the numerical solution lies within a finite-dimensional polynomial space of the form
\begin{equation}\label{eq:ref_space}
\begin{aligned}
\set{V} \df \Big\{\hat{\ug{\Omega}} \ni \hat{\vc{x}} \mapsto \fn{v}(\hat{\vc{x}}) \in \R : \ &\exists \{v^{(\vc{\alpha})}\}_{\vc{\alpha} \in \set{N}} \subset \R \quad \text{s. t.}\ \\ &\fn{v}(\hat{\vc{x}}) = \sum_{\vc{\alpha}\in \set{N}} v^{(\vc{\alpha})} \hat{x}_1^{\alpha_1} \cdots \hat{x}_d^{\alpha_d} \Big\},
\end{aligned}
\end{equation}
where $\set{N} \supseteq \{\vc{\alpha}_1 \in \mathbb{N}_0^d : \lvert \vc{\alpha} \rvert \leq p\}$ is a finite set of multi-indices\footnote{As is customary, we define the order of the multi-index $\vc{\alpha} \in \mathbb{N}_0^d$ as $\lvert \vc{\alpha} \rvert \df \alpha_1 + \cdots + \alpha_d$.} with cardinality $N_p \in \mathbb{N}$. The standard total-degree and tensor-product polynomial spaces $\set{P}_p(\hat{\ug{\Omega}})$ and $\set{Q}_p(\hat{\ug{\Omega}})$ are then recovered with $\set{N} \df \{\vc{\alpha} \in \mathbb{N}_0^d : \lvert \vc{\alpha} \rvert \leq p\}$ and $\set{N} \df \{\vc{\alpha} \in \mathbb{N}_0^d : \max_{m=1}^d \alpha_m \leq p\}$, respectively, and it is hereinafter assumed that, as is the case for $\set{P}_p(\hat{\ug{\Omega}})$ and $\set{Q}_p(\hat{\ug{\Omega}})$ with any $p \in \mathbb{N}_0$, the space $\set{V}$ is closed under partial differentiation with respect to each coordinate and contains the space of constant functions on the reference element, which we denote as $\set{P}_0(\hat{\ug{\Omega}})$. The global semi-discrete solution is therefore given by $\fv{u}^h\xt \df \bigoplus_{\kappa=1}^{K}\fv{u}^{(h,\kappa)}\big((\fv{X}^{(\kappa)})^{-1}(\vc{x}),t\big)$ such that $\fn{u}_e^h(\cdot,t) \in \set{V}_h$ for all $e \in \{1,\ldots, N_{\mathrm{Eq}}\}$ and $t \in (0,T)$, where we define the global approximation space
\begin{equation}\label{eq:solution_space}
\set{V}_h \df \Big\{\fn{v} \in L^2(\ug{\Omega}) :  \fn{v}\big\lvert_{\ug{\Omega}^{(\kappa)}}\circ \fv{X}^{(\kappa)} \in \set{V}, \ \forall \, \ug{\Omega}^{(\kappa)} \in \set{T}_h\Big\},
\end{equation}
with $L^2(\ug{\Omega})$ denoting the space of square-integrable functions on the domain $\ug{\Omega}$.
\begin{remark}
Rather than approximating $\fv{u}(\fv{X}^{(\kappa)}(\hat{\vc{x}}),t)$ such that $\fn{u}_e^{(h,\kappa)}(\cdot,t) \in \set{V}$, it is possible to approximate  $\mathcal{J}^{(\kappa)}(\hat{\vc{x}})\fv{u}(\fv{X}^{(\kappa)}(\hat{\vc{x}}),t)$ such that $\mathcal{J}^{(\kappa)}\fn{u}_e^{(h,\kappa)}(\cdot,t) \in \set{V}$. This modification of the approximation space has no effect for affine elements, but simplifies the treatment of curvilinear coordinates for time-dependent problems. While the analysis in \textsection \ref{sec:analysis} extends in a straightforward manner to such modified formulations, we note that, as discussed by Yu \etal \cite[\textsection 3]{yu_wang_liu_dg_sd_cpr_comparison_14}, methods employing such approximation spaces on curved elements may be less accurate than those employing (\ref{eq:solution_space}) as described in this paper.
\end{remark}
\subsection{Discontinuous Galerkin method}
Integrating each equation in (\ref{eq:pde_ref}) by parts against a test function belonging to the space $\set{V}$, we obtain a local weak formulation of the conservation law, where the resulting integrals may be expressed in terms of the standard $L^2$ inner products $\il\fn{u},\fn{v}\ir_{\hat{\ug{\Omega}}} \df \int_{\hat{\ug{\Omega}}} \fn{u}(\hat{\vc{x}})\fn{v}(\hat{\vc{x}}) \, \dd \hat{\vc{x}}$ and $\il\fn{u},\fn{v}\ir_{\hat{\ug{\Gamma}}^{(\gamma)}} \df \int_{\hat{\ug{\Gamma}}^{(\gamma)}} \fn{u}(\hat{\vc{x}})\fn{v}(\hat{\vc{x}}) \, \dd \hat{s}$. After applying a consistent and conservative numerical flux function $\fv{f}^* : \ug{\Upsilon} \times \ug{\Upsilon} \times \mathbb{S}^{d-1} \to \R^{N_{\textrm{Eq}}}$, typically corresponding to an approximate Riemann solver developed in the context of finite-volume methods (see, for example, Toro \cite{toro09}), we obtain the standard weak-form DG approximation on the reference element, which is given by
\begin{equation}\label{eq:dg}
\begin{multlined}
\bigg\il\fn{v},  \pdv{\mathcal{J}^{(\kappa)}\fn{u}_e^{(h,\kappa)}(\cdot,t)}{t} \bigg\ir_{\hat{\ug{\Omega}}} - \sum_{m=1}^d\bigg\il\pdv{\fn{v}}{\hat{x}_m},\fn{f}_e^{(h,\kappa,m)}(\cdot,t)\bigg\ir_{\hat{\ug{\Omega}}} \\ + \sum_{\gamma=1}^{N_f}\bigg\il\fn{v}, \fn{f}_e^{(*,\kappa,\gamma)}(\cdot,t)\bigg\ir_{\hat{\ug{\Gamma}}^{(\gamma)}} 
= 0, \quad \forall \, \fn{v} \in \set{V},
\end{multlined}
\end{equation}
for all $\kappa \in \{1,\ldots, K\}$, $e \in \{1,\ldots, N_{\textrm{Eq}}\}$, and $t \in (0,T)$. In the above, the volume-weighted contravariant flux components appearing in parentheses in (\ref{eq:pde_ref}) are evaluated in terms of the numerical solution as
\begin{equation}\label{eq:transformed_flux}
\fv{f}^{(h,\kappa,m)}\xit \df \sum_{n=1}^d\mathcal{J}^{(\kappa)}(\hat{\vc{x}}) [(\mat{\mathcal{J}}^{(\kappa)}(\hat{\vc{x}}))^{-1}]_{mn} \fv{f}^{(n)}(\fv{u}^{(h,\kappa)}\xit),
\end{equation}
and the appropriately scaled numerical flux components are given as
\begin{equation}\label{eq:num_flux}
\fv{f}^{(*,\kappa,\gamma)}\xit \df \mathcal{J}^{(\kappa,\gamma)}(\hat{\vc{x}}) \fv{f}^*\big(\fv{u}^{(h,\kappa)}\xit, \fv{u}^{(+,\kappa,\gamma)}\xit, \vc{n}^{{(\kappa,\gamma)}}(\fv{X}^{(\kappa)}(\hat{\vc{x}}))\big),
\end{equation}
 where $\fv{u}^{(+,\kappa,\gamma)}\xit$ denotes the external data corresponding to a weakly imposed boundary or interface condition. 

\subsection{Flux reconstruction method}\label{sec:review_fr}
While DG and FR methods both employ polynomial approximation spaces as described in \textsection \ref{sec:approx_space}, the FR approach involves the discretization of the strong form of (\ref{eq:pde_ref}) rather than the weak form, and traditionally employs a collocation-based approximation rather than a Galerkin projection. Considering a unisolvent nodal set $\set{S} \df \{ \hat{\vc{x}}^{(i)} \}_{i=1}^{N} \subset \hat{\ug{\Omega}}$ for the space $\set{V}$ such that for a given function $\fn{v} : \hat{\ug{\Omega}} \to \R$, there exists a unique interpolant $\opr{I}\fn{v} \in \set{V}$ satisfying $(\opr{I}\fn{v})(\hat{\vc{x}}^{(i)}) = \fn{v}(\hat{\vc{x}}^{(i)})$ for all $i \in \{1,\ldots, N\}$, a standard FR scheme may be expressed concisely as
\begin{equation}\label{eq:fr}
\begin{multlined}
\pdv{\opr{I}(\mathcal{J}^{(\kappa)}\fn{u}_e^{(h,\kappa)}(\cdot,t))}{t} \\ + \sum_{m=1}^d \pdv{\hat{x}_m} \bigg(\opr{I}\fn{f}_e^{(h,\kappa,m)}(\cdot,t)   + \sum_{\gamma=1}^{N_f}\sum_{j=1}^{N_\gamma}\fn{g}_m^{(\gamma,j)}\Delta \fn{f}_e^{(\kappa,\gamma)}(\hat{\vc{x}}^{(\gamma,j)},t)\bigg) = 0
\end{multlined}
\end{equation}
for all $\kappa \in \{1,\ldots, K\}$, $e \in \{1,\ldots, N_{\textrm{Eq}}\}$, and $t \in (0,T)$.  In the above, we associate each node in the set $\set{S}^{(\gamma)} \df  \{ \hat{\vc{x}}^{(\gamma, i)} \}_{i=1}^{N_\gamma} \subseteq \hat{\ug{\Gamma}}^{(\gamma)}$ defined for $\gamma \in \{1,\ldots, N_\gamma\}$ with a vector-valued correction function $\fv{g}^{(\gamma,j)} : \hat{\ug{\Omega}} \to \R^d$ satisfying $\nabla \cdot \fv{g}^{(\gamma,j)} \in \set{V}$ and $\fv{g}^{(\gamma,j)}(\hat{\vc{x}}^{(\eta,i)}) \cdot \hat{\vc{n}}^{(\eta)} = \delta_{\gamma\eta}\delta_{ij}$ for all $\eta \in \{1,\ldots, N_f\}$ and $i \in \{1,\ldots, N_\eta\}$, and we define the flux difference in the normal direction as
\begin{equation}\label{eq:normal_flux_difference}
\begin{multlined}
\Delta\fn{f}_e^{(\kappa,\gamma)}\xit \df  \fn{f}_e^{(*,\kappa,\gamma)}\xit- \sum_{n=1}^d\hat{n}^{(\gamma)}_n (\opr{I}\fn{f}_e^{(h,\kappa,n)}(\cdot,t))(\hat{\vc{x}}).
\end{multlined}
\end{equation}
The properties of the resulting scheme are then determined by the choice of correction functions, which we discuss in \textsection \ref{sec:correction}.
\section{Algebraic formulation}\label{sec:formulation}
\subsection{Discrete operators and the summation-by-parts property}
Before presenting our algebraic formulation of the DG and FR methods, we must first define certain operators associated with the approximation space $\set{V}$ and their associated matrix representations, the properties of which are discussed in this section.
\subsubsection*{Discrete inner products}
In order to construct approximations of the $L^2$ inner products $\il \fn{u},\fn{v} \ir_{\hat{\ug{\Omega}}}$ and $\il \fn{u},\fn{v} \ir_{\hat{\ug{\Gamma}}^{(\gamma)}}$ based on the values of $\fn{u}$ and $\fn{v}$ on nodal sets $\set{S}$ and $\set{S}^{(\gamma)}$ defined as in \textsection \ref{sec:review_fr}, we relax the assumption of unisolvency to allow for $N \geq N_p$ and define discrete inner products given by
\begin{equation}\label{eq:disc_in_prod_vol}
\il \fn{u},\fn{v} \ir_{\mat{W}} \df \sum_{i=1}^{N}\sum_{j=1}^{N}\fn{u}(\hat{\vc{x}}^{(i)})\el{W}_{ij}\fn{v}(\hat{\vc{x}}^{(j)})
\end{equation}
and
\begin{equation}\label{eq:disc_in_prod_facet}
\il \fn{u},\fn{v} \ir_{\mat{W}^{(\gamma)}}\df \sum_{i=1}^{N_{\gamma}}\sum_{j=1}^{N_\gamma}\fn{u}(\hat{\vc{x}}^{(\gamma,i)})\el{W}_{ij}^{(\gamma)}\fn{v}(\hat{\vc{x}}^{(\gamma,j)})
\end{equation}
in terms of the matrices $\mat{W} \in \R^{N\times N}$ and $\mat{W}^{(\gamma)} \in \R^{N_\gamma \times N_\gamma}$, respectively. These bilinear forms approximate each term of the multidimensional integration-by-parts (IBP) relation on the reference element as
\begin{equation}\label{eq:ibp}
\underbrace{\int_{\hat{\ug{\Omega}}}\fn{u}(\hat{\vc{x}})\pdv{\fn{v}(\hat{\vc{x}})}{\hat{x}_m}\, \dd \hat{\vc{x}}}_{\approx\, \il\fn{u},\,\pdvln{\fn{v}}{\hat{x}_m}\ir_{\mat{W}}} + \underbrace{\int_{\hat{\ug{\Omega}}}\pdv{\fn{u}(\hat{\vc{x}})}{\hat{x}_m}\fn{v}(\hat{\vc{x}})\, \dd \hat{\vc{x}}}_{\approx\, \il\pdvln{\fn{u}}{\hat{x}_m}, \,\fn{v}\ir_{\mat{W}}} = \sum_{\gamma=1}^{N_f}\underbrace{\int_{\hat{\ug{\Gamma}}^{(\gamma)}} \fn{u}(\hat{\vc{x}}) \fn{v}(\hat{\vc{x}})\hat{n}_m^{(\gamma)} \, \dd \hat{s}}_{\mathclap{\approx \, \hat{n}_m^{(\gamma)}\il\fn{u},\fn{v}\ir_{\mat{W}^{(\gamma)}}}},
\end{equation}
where we make the following assumption regarding such approximations.
\begin{assumption}\label{asm:sbp}
The IBP relation in (\ref{eq:ibp}) is satisfied exactly for all $m\in \{1,\ldots, d\}$ under the discrete inner products for arguments within the reference approximation space described in \textsection \ref{sec:approx_space}, as given by
\begin{equation}\label{eq:ibp_polynomials_discrete}
\begin{aligned}
\bigg\il\fn{u}, \pdv{\fn{v}}{\hat{x}_m}\bigg\ir_{\mat{W}}  + \bigg\il \pdv{\fn{u}}{\hat{x}_m},\fn{v}\bigg\ir_{\mat{W}} = \sum_{\gamma=1}^{N_f}\hat{n}_m^{(\gamma)}\big\il\fn{u}, \fn{v} \big\ir_{\mat{W}^{(\gamma)}}, \quad \forall \, \fn{u},\fn{v} \in \set{V}.
\end{aligned}
\end{equation}
\end{assumption}
Introducing an arbitrary nodal (i.e.\ Lagrange) or modal basis $\set{B} \df \{\phi^{(i)}\}_{i=1}^{N_p}$ for the space $\set{V}$, we may define the generalized Vandermonde matrices $\mat{V} \in \R^{N \times N_p}$ and $\mat{V}^{(\gamma)} \in \R^{N_\gamma \times N_p}$ with entries given by $\el{V}_{ij} = \phi^{(j)}(\hat{\vc{x}}^{(i)})$ and $\smash{\el{V}_{ij}^{(\gamma)}} = \phi^{(j)}(\hat{\vc{x}}^{(\gamma,i)})$, respectively. Referring to \textsection \ref{app:bases} as well as to the textbooks of Hesthaven and Warburton \cite[\textsection 3.1, \textsection 6.1, \textsection 10.1]{hesthaven08} and Karniadakis and Sherwin \cite[Ch.\ 3]{karniadakis_sherwin_spectral_hp_element} for details regarding the construction of such bases, we state the following additional assumption.
\begin{assumption}\label{asm:inner_prod}
The generalized Vandermonde matrix $\mat{V}$ is of rank $N_p$. Additionally, $\mat{W}$ is symmetric positive definite (SPD), as are  $\mat{W}^{(\gamma)}$ for all $\gamma \in \{1,\ldots, N_f\}$.
\end{assumption}
\begin{remark}
The condition on the rank of $\mat{V}$ generalizes the unisolvency property of $\set{S}$ assumed in \textsection \ref{sec:review_fr} to the case of $N \geq N_p$. 	Details regarding the construction of discrete inner products satisfying Assumptions  \ref{asm:sbp} and \ref{asm:inner_prod} for quadrature-based and collocation-based schemes are provided in \textsection \ref{app:quadrature} and \textsection \ref{app:collocation}, respectively.
\end{remark}
The mass matrix with entries consisting of the discrete inner products $\il \cdot, \cdot \ir_{\mat{W}}$ of all pairs of basis functions in $\set{B}$ is then characterized by the following lemma.
\begin{lemma}\label{lem:spd}
Under Assumption \ref{asm:inner_prod}, the mass matrix $\mat{M} \df \mat{V}^\T\mat{W}\mat{V}$ is SPD.
\end{lemma}
\begin{proof}
The symmetry and positivity of the corresponding quadratic form are clear from the structure of $\mat{M}$ and the fact that $\mat{W}$ is SPD, while the definiteness property follows from the fact that the only vector in the nullspace of $\mat{V}$ is the zero vector due to it being of full column rank.
\end{proof}
\begin{remark}\label{rmk:bijective}
Using the expansions $\fn{u}(\hat{\vc{x}}) = \sum_{i=1}^{N_p}\el{u}_i\phi^{(i)}(\hat{\vc{x}})$ and $\fn{v}(\hat{\vc{x}})  = \sum_{i=1}^{N_p}\el{v}_i\phi^{(i)}(\hat{\vc{x}})$, the discrete inner product of any two functions $\fn{u},\fn{v} \in \set{V}$ may be expressed as $(\fn{u},\fn{v}\ir_{\mat{W}} = \vc{u}^\T\mat{M}\vc{v}$, where $\mat{M}$ is SPD by Lemma \ref{lem:spd}, and since the mapping between any function in $\set{V}$ and its expansion coefficients is an isomorphism, the discrete inner product satisfies $\il \fv{v},\fv{v} \ir_{\mat{W}} > 0$ for all nonzero $\fn{v} \in \set{V}$. The discrete inner product is therefore a true inner product for the space $\set{V}$, although definiteness is violated for more general arguments outside of the space $\set{V}$.
\end{remark}
We also define the mass matrix $\mat{M}^{(\kappa)} \df \mat{V}^\T\mat{W}\mat{J}^{(\kappa)}\mat{V}$ corresponding to the physical element $\ug{\Omega}^{(\kappa)} \in \set{T}_h$ for future reference, where $\mat{J}^{(\kappa)} \df \textrm{diag}([\mathcal{J}^{(\kappa)}(\hat{\vc{x}}^{(1)}), \ldots,\mathcal{J}^{(\kappa)}(\hat{\vc{x}}^{(N)})]^\T)$.

\subsubsection*{Projection operators}
Associated with the discrete inner product $\il \cdot, \cdot \ir_{\mat{W}}$ is a corresponding orthogonal projection operator $\Pi$ onto the space $\set{V}$, which we introduce with the following lemma.
\begin{lemma}\label{lem:proj}
	Under Assumption \ref{asm:inner_prod}, the discrete projection $\Pi\fn{u} \in \set{V}$ of a given function $\fn{u} : \hat{\ug{\Omega}} \to \R$ satisfying $\il\fn{v}, \Pi\fn{u} - \fn{u} \ir_{\mat{W}} = 0$ for all $\fn{v} \in \set{V}$ can be represented in terms of the basis $\set{B}$ and the matrix $\mat{P} \df \mat{M}^{-1}\mat{V}^\T\mat{W}$ as $(\Pi\fn{u})(\hat{\vc{x}}) =  \sum_{i=1}^{N_p}\sum_{j=1}^{N}\el{P}_{ij}\fn{u}(\hat{\vc{x}}^{(j)})\phi^{(i)}(\hat{\vc{x}})$,
	and corresponds to an interpolation of $\fn{u}$ on the nodes in $\set{S}$ when $N = N_p$, satisfying $(\Pi\fn{u})(\hat{\vc{x}}^{(k)}) =\fn{u}(\hat{\vc{x}}^{(k)})$ for all $k \in \{1,\ldots, N\}$.
\end{lemma}

\begin{proof}
	As with Lemma \ref{lem:sbp}, the result follows from expanding $\fn{v}$ and $\Pi\fn{u}$ in terms of $\set{B}$, as shown, for example, in Karniadakis and Sherwin \cite[Eq.\ (4.1.33)]{karniadakis_sherwin_spectral_hp_element}, and noting that the invertibility of $\mat{M}$ follows from its positive-definiteness, which was established in Lemma \ref{lem:spd}. If $\set{S}$ is unisolvent for $\set{V}$, as is implied by Assumption \ref{asm:inner_prod} when $N = N_p$, we use the invertibility of $\mat{V}$ to obtain $\mat{P} = \mat{V}^{-1}\mat{W}^{-1}\mat{V}^{-T}\mat{V}^\T\mat{W} = \mat{V}^{-1}$, and hence evaluating the projection at each node we obtain $(\Pi\fn{u})(\hat{\vc{x}}^{(k)}) = \sum_{i=1}^{N_p}\sum_{j=1}^{N_p} [\mat{V}^{-1}]_{ij} \el{V}_{ki} \fn{u}(\hat{\vc{x}}^{(j)}) = \fn{u}(\hat{\vc{x}}^{(k)})$ for all $k \in \{1,\ldots, N\}$, since $\phi_i(\hat{\vc{x}}^{(k)}) = \el{V}_{ki}$ and $\sum_{i=1}^{N_p}[\mat{V}^{-1}]_{ij}\el{V}_{ki} = \delta_{jk}$, demonstrating that $\Pi\fn{u} \in \set{V}$ interpolates  $\fn{u}$ on $\set{S}$. 
\end{proof}
While Lemma \ref{lem:proj} describes an approximation of the $L^2$ (i.e.\ Galerkin) projection on the reference element, we may also define a volume-weighted projection corresponding to the approximation of the $L^2$ projection on the physical element $\ug{\Omega}^{(\kappa)} \in \set{T}_h$ by replacing $\mat{P}$ with $\mat{P}^{(\kappa)} \df (\mat{M}^{(\kappa)})^{-1}\mat{V}^\T\mat{W}\mat{J}^{(\kappa)}$, for which we require the following assumption.
\begin{assumption}\label{asm:mass_invertible}
	The physical mass matrix $\mat{M}^{(\kappa)}$ is invertible for all $\kappa \in \{1,\ldots, K\}$.
\end{assumption}
Since the assumptions of \textsection \ref{sec:mesh} imply that the matrix $\mat{J}^{(\kappa)}$ is SPD, $\mat{W}\mat{J}^{(\kappa)}$ and $\mat{M}^{(\kappa)}$ are also SPD (and therefore invertible) under Assumption \ref{asm:inner_prod} when $\mat{W}$ is diagonal, as is the case for the quadrature-based approximations in \textsection \ref{app:quadrature}. For the collocation-based approach in \textsection \ref{app:collocation}, $\mat{W}$ is not necessarily diagonal, but since $\mat{M}^{(\kappa)}$ is equal to $\mat{W}\mat{J}^{(\kappa)}$, which is the product of two SPD matrices, Assumption \ref{asm:mass_invertible} remains satisfied. 
\begin{remark}
If the mapping is affine (i.e.\ $\mathcal{J}^{(\kappa)}$ is constant) or if $N = N_p$, it follows that $\mat{P}^{(\kappa)} = \mat{P}$, and the weighting therefore has no influence on the projection.
\end{remark}
\subsubsection*{Summation-by-parts property}
As the space $\set{V}$ is assumed in \textsection \ref{sec:approx_space} to be closed under partial differentiation, the operator $\partial/\partial\hat{x}_m :\set{V} \to \set{V}$ may be represented algebraically for any $m \in \{1,\ldots, d\}$ in terms of the basis $\set{B}$ through the matrix $\mat{D}^{(m)} \in \R^{N_p \times N_p}$ satisfying
\begin{equation}\label{eq:diff_mat}
 \pdv{\fn{v}(\hat{\vc{x}})}{\hat{x}_m} = \sum_{i=1}^{N_p}\sum_{j=1}^{N_p} \el{D}_{ij}^{(m)}v_j \phi^{(i)}(\hat{\vc{x}})
\end{equation}
for any function $\fn{v} \in \set{V}$ given by $\fn{v}(\hat{\vc{x}})  = \sum_{i=1}^{N_p}\el{v}_i\phi^{(i)}(\hat{\vc{x}})$. The relation in (\ref{eq:ibp_polynomials_discrete}) may then be represented algebraically, as established with the following lemma.
\begin{lemma}\label{lem:sbp}
Under Assumption \ref{asm:sbp}, the matrix $\mat{D}^{(m)}$ defined implicitly by (\ref{eq:diff_mat}) satisfies the SBP property
\begin{equation}\label{eq:sbp}
\mat{M}\mat{D}^{(m)} + (\mat{D}^{(m)})^\T\mat{M} = \sum_{\gamma=1}^{N_f}\hat{n}_m^{(\gamma)} (\mat{V}^{(\gamma)})^\T\mat{W}^{(\gamma)}\mat{V}^{(\gamma)}.
\end{equation}
\end{lemma}
\begin{proof}
As noted by Chan \cite[Eq.\ (26)]{chan_discretely_entropy_conservative_dg_sbp_18}, the result follows directly from expanding $\fn{v}$ and $\fn{w}$ and their partial derivatives in terms of the basis $\set{B}$ and expressing the discrete inner products in matrix form as given in (\ref{eq:disc_in_prod_vol}) and (\ref{eq:disc_in_prod_facet}).
\end{proof}
\begin{remark}
	As the operators in (\ref{eq:sbp}) act on expansion coefficients corresponding to a given basis $\set{B}$, which do not necessarily correspond to nodal values, the relation is sometimes referred to as a ``modal'' SBP property (see, for example, Chen and Shu \cite{chen_shu_dgsbp_review_19}). As demonstrated in \cite[Lemma 1]{chan_discretely_entropy_conservative_dg_sbp_18}, a nodal SBP property is recovered by pre-multiplying (\ref{eq:sbp}) by $\mat{P}^\T$ and post-multiplying by $\mat{P}$, where the difference operator $\mat{V}\mat{D}^{(m)}\mat{P}$ approximating $\partial / \partial \hat{x}_m$ on the nodes in $\set{S}$ meets the conditions of Hicken \etal \cite[Definition 2.1]{hicken_mdsbp_16} to be a (nodal) multidimensional SBP operator. When $\set{B}$ is taken to be a nodal basis defined on $\tilde{\set{S}} \df \{\tilde{\hat{\vc{x}}}^{(i)}\}_{i=1}^{N_p} \subset \hat{\ug{\Omega}}$ as in \textsection \ref{app:nodal_basis}, the matrix $\mat{D}^{(m)}$ is itself a nodal SBP operator, recovering a ``staggered'' SBP formulation as described by Parsani \etal \cite{parsani_entropy_stable_staggered_grid_collocation_16} or Del Rey Fern\'andez \etal \cite{delrey_mdsbp_staggered_19} when $\tilde{\set{S}} \neq \set{S}$. 
\end{remark}
\subsection{Algebraic formulation of the discontinuous Galerkin method}\label{sec:dg_algebraic}
Given a basis $\set{B}$ for the space $\set{V}$, the components of the transformed numerical solution for a DG or FR method may be expressed as $\fn{u}_e^{(h,\kappa)}\xit  = \sum_{i=1}^{N_p}\el{u}_i^{(h,\kappa,e)}(t)\phi^{(i)}(\hat{\vc{x}})$. 
In the particular case of a DG method, replacing the $L^2$ inner products in (\ref{eq:dg}) with the corresponding discrete inner products defined in (\ref{eq:disc_in_prod_vol}) and (\ref{eq:disc_in_prod_facet}) results in a coupled system of ordinary differential equations given by
\begin{equation}\label{eq:dg_algebraic}
\begin{multlined}
\mat{M}^{(\kappa)}\dv{\vc{u}^{(h,\kappa,e)}(t)}{t}  - \sum_{m=1}^d(\mat{D}^{(m)})^\T\mat{V}^\T\mat{W}\vc{f}^{(h,\kappa,m,e)}(t)  \\ +\sum_{\gamma=1}^{N_f}(\mat{V}^{(\gamma)})^\T\mat{W}^{(\gamma)} \vc{f}^{(*,\kappa,\gamma,e)}(t)  = \vc{0}
\end{multlined}
\end{equation} 
for all $\kappa \in \{1,\ldots, K\}$, $e \in \{1,\ldots,N_{\textrm{Eq}}\}$, and $t \in (0,T)$, where the nodal values of the flux components are given as $\vc{f}^{(h,\kappa,m,e)}(t) \df [\fn{f}_e^{(h,\kappa,m)}(\hat{\vc{x}}^{(1)},t), \ldots,\fn{f}_e^{(h,\kappa,m)}(\hat{\vc{x}}^{(N)},t)]^\T$ and $\vc{f}^{(*,\kappa,\gamma,e)}(t) \df [\fn{f}_e^{(*,\kappa,\gamma)}(\hat{\vc{x}}^{(\gamma,1)},t),\ldots, \fn{f}_e^{(*,\kappa,\gamma)}(\hat{\vc{x}}^{(\gamma,N_\gamma)},t)]^\T$. Evaluating the initial data as $\vc{u}^{(0,\kappa,e)} \df [\fn{u}_e^0(\fv{X}^{(\kappa)}(\hat{\vc{x}}^{(1)})),\ldots, \fn{u}_e^0(\fv{X}^{(\kappa)}(\hat{\vc{x}}^{(N)}))]^\T$, the system in (\ref{eq:dg_algebraic}) is initialized as $\vc{u}^{(h,\kappa,e)}(0) = \mat{P}^{(\kappa)}\vc{u}^{(0,\kappa,e)}$ for all $\kappa \in \{1,\ldots,K\}$ and $e \in \{1,\ldots, N_{\textrm{Eq}}\}$ such that each component of $\fv{u}^h(\vc{x},0)$ approximates the global $L^2$ projection of the initial data $\fv{u}^0(\vc{x})$ onto the space $\set{V}_h$ given in (\ref{eq:solution_space}).

\begin{remark}
	In order to obtain $\vc{f}^{(h,\kappa,m,e)}(t)$ and $\vc{f}^{(*,\kappa,\gamma,e)}(t)$, the fluxes in (\ref{eq:transformed_flux}) and (\ref{eq:num_flux}) may be evaluated pointwise following the pre-multiplication of $\vc{u}^{(h,\kappa,e)}$ by $\mat{V}$ and $\mat{V}^{(\gamma)}$, respectively, in order to obtain the corresponding nodal values of the numerical solution. If $\set{B}$ is taken to be a nodal basis defined on $\tilde{\set{S}}$ as in \textsection \ref{app:nodal_basis}, the pre-multiplication by $\mat{V}$ may be avoided if collocation is exploited with $\set{S} = \tilde{\set{S}}$, while the pre-multiplication by $\mat{V}^{(\gamma)}$ may be avoided if $\set{S}^{(\gamma)} \subset \tilde{\set{S}}$.
\end{remark}

\subsection{Algebraic formulation of the flux reconstruction method}\label{sec:fr_algebraic}
Separating the divergence of the projected flux in (\ref{eq:fr}) from that of the correction term, defining the correction field $\fn{h}^{(\gamma,j)} \in \set{V}$ as the divergence of the corresponding correction function $\fv{g}^{(\gamma,j)}$ as in Castonguay \etal \cite{castonguay_vincent_jameson_triangular_fr_11} or Williams and Jameson \cite{williams_esfr_adv_diff_tetrahedra_13}, and replacing the interpolation operator $\opr{I}$ introduced in \textsection \ref{sec:review_fr} with the more general projection operator $\Pi$ described in Lemma \ref{lem:proj}, including within the normal flux difference defined in (\ref{eq:normal_flux_difference}), the FR scheme may be expressed as
\begin{equation}\label{eq:fr_lift}
\begin{multlined}
\pdv{\Pi(\mathcal{J}^{(\kappa)}\fn{u}_e^{(h,\kappa)}(\cdot,t))}{t} + \sum_{m=1}^d \pdv{\Pi\fn{f}_e^{(h,\kappa,m)}(\cdot,t)}{\hat{x}_m} \\ + \sum_{\gamma=1}^{N_f}\sum_{j=1}^{N_\gamma}\fn{h}^{(\gamma,j)}\Delta\fn{f}_e^{(\kappa,\gamma)}(\hat{\vc{x}}^{(\gamma,j)},t) = 0
\end{multlined}
\end{equation}
for all $\kappa \in \{1,\ldots, K\}$, $e \in \{1,\ldots, N_{\textrm{Eq}}\}$, and $t \in (0,T)$. An algebraic formulation of the FR method may then be obtained analogously to that described for the DG method in \textsection \ref{sec:dg_algebraic} by expressing all functions in $\set{V}$ appearing in (\ref{eq:fr_lift}) in terms of the basis $\set{B}$, from which we obtain
\begin{equation}\label{eq:fr_algebraic}
\begin{multlined}
\mat{P}\mat{J}^{(\kappa)}\mat{V}\dv{\vc{u}^{(h,\kappa,e)}(t)}{t}+ \sum_{m=1}^d \mat{D}^{(m)}\mat{P} \vc{f}^{(h,\kappa,m,e)}(t)  \\ +  \sum_{\gamma=1}^{N_f}\mat{L}^{(\gamma)} \qty(\vc{f}^{(*,\kappa,\gamma,e)}(t)  - \sum_{n=1}^d\hat{n}_n^{(\gamma)} \mat{V}^{(\gamma)}\mat{P}\vc{f}^{(h,\kappa,n,e)}(t)) = \vc{0}
\end{multlined}
\end{equation}
for all $\kappa \in \{1,\ldots, K\}$, $e \in \{1,\ldots,N_{\textrm{Eq}}\}$, and $t \in (0,T)$, with the initial condition applied as described for the DG method in \textsection \ref{sec:dg_algebraic}. The correction fields appearing in (\ref{eq:fr_lift}) are encoded within the lifting matrices $\mat{L}^{(\gamma)} \in \R^{N_p \times N_\gamma}$, which are defined such that
\begin{equation}\label{eq:corr_field}
\fn{h}^{(\gamma,j)}(\hat{\vc{x}}) = \sum_{i=1}^{N_p}\el{L}_{ij}^{(\gamma)} \phi^{(i)}(\hat{\vc{x}})
\end{equation}
for all $\gamma \in \{1,\ldots, N_f\}$ and $j \in \{1,\ldots, N_\gamma\}$. These operators allow for more general correction procedures than those in \cite[Eq.\ (22)]{chan_discretely_entropy_conservative_dg_sbp_18} and are analogous to the correction matrices introduced in \cite{ranocha_sbp_cpr_16} for one-dimensional collocated formulations.

\begin{remark}
The matrix $\mat{P}\mat{J}^{(\kappa)}\mat{V}$ pre-multiplying the time derivative in (\ref{eq:fr_algebraic}) may be expressed equivalently as $\mat{M}^{-1}\mat{M}^{(\kappa)}$ and is therefore invertible under Assumption \ref{asm:mass_invertible}. 
\end{remark}
\subsection{Vincent-Castonguay-Jameson-Huynh correction functions}\label{sec:correction}
To complete our algebraic formulation of the FR method in (\ref{eq:fr_algebraic}), the entries of the matrices $\mat{L}^{(\gamma)}$, or, equivalently, the expansion coefficients of the correction fields in terms of the basis $\set{B}$, must be specified. The following theorem demonstrates how such a matrix may be obtained, beginning with the fundamental assumptions associated with the VCJH family of correction functions.
\begin{theorem}\label{thm:vcjh}
Suppose that there exist correction functions $ \fv{g}^{(\gamma,j)}: \hat{\ug{\Omega}} \to \R^d$ defined for $\gamma \in \{1,\ldots,N_f\}$ and $j \in \{1,\ldots, N_\gamma\}$ such that $\fn{h}^{(\gamma,j)} \df \nabla \cdot \fv{g}^{(\gamma,j)} \in \set{V}$ and $
 \fv{g}^{(\gamma,j)}(\hat{\vc{x}}^{(\eta,i)}) \cdot \hat{\vc{n}}^{(\eta)} = \delta_{\gamma\eta}\delta_{ij}$ for all $\eta \in \{1,\ldots, N_f\}$ and $i \in \{1,\ldots, N_\eta\}$, satisfying
\begin{equation}\label{eq:vcjh}
\int_{\hat{\ug{\Omega}}} \bigg(\fv{g}^{(\gamma,j)}(\hat{\vc{x}}) \cdot \nabla\fn{v}(\hat{\vc{x}}) - \sum_{\vc{\alpha} \in \set{M}} \frac{\fn{C}({\vc{\alpha}})}{\lvert\hat{\ug{\Omega}}\rvert} \frac{\partial^{\abs{\vc{\alpha}}}\fn{V}(\hat{\vc{x}})}{\partial \hat{x}_1^{\alpha_1} \cdots \partial \hat{x}_d^{\alpha_d}}\frac{\partial^{\abs{\vc{\alpha}}}\fn{h}^{(\gamma,j)}(\hat{\vc{x}})}{\partial \hat{x}_1^{\alpha_1} \cdots \partial \hat{x}_d^{\alpha_d}}\bigg) \, \dd \hat{\vc{x}} = 0
\end{equation}
for all $ \fn{v} \in \set{V}$, where $\fn{C} : \set{M} \to \R$ defines the coefficients for each term in the sum over the finite set of multi-indices $\set{M} \subset \set{N}$ (with $\set{N}$ defined in \textsection \ref{sec:approx_space}), and $\lvert \hat{\ug{\Omega}} \rvert$ is the volume (i.e. Lebesgue measure) of the reference element. Given an arbitrary basis $\set{B}$ for $\set{V}$ and discrete inner products of sufficient accuracy such that $\il\fn{u},\fn{v}\ir_{\mat{W}} =\il\fn{u},\fn{v}\ir_{\hat{\ug{\Omega}}}  $ and $\il\fn{u},\fn{v}\ir_{\mat{W}^{(\gamma)}} = \il\fn{u},\fn{v}\ir_{\hat{\ug{\Gamma}}^{(\gamma)}}$ for all $\fn{u},\fn{v} \in \set{V}$, the correction field corresponding to $\gamma \in \{1,\ldots, N_f\}$ and $j \in \{1,\ldots, N_\gamma\}$ may then be expanded as in (\ref{eq:corr_field}) with coefficients given by $\mat{L}^{(\gamma)} \df (\mat{M}+\mat{K})^{-1}(\mat{V}^{(\gamma)})^\T\mat{W}^{(\gamma)}$,
where we define $\mat{K} \df \sum_{\vc{\alpha} \in \set{M}} \frac{\fn{C}({\vc{\alpha}})}{\lvert\hat{\ug{\Omega}}\rvert}(\mat{D}^{(\vc{\alpha})})^\T\mat{M}\mat{D}^{(\vc{\alpha})}$ with $\mat{D}^{(\vc{\alpha})} \df (\mat{D}^{(1)})^{\alpha_1} \cdots (\mat{D}^{(d)})^{\alpha_d}$, under the assumption that the coefficients $C(\vc{\alpha})$ have been chosen such that $\mat{M}+\mat{K}$ is SPD.
\end{theorem}
\begin{proof}Applying IBP to the first term of the integrand in (\ref{eq:vcjh}) and noting that the resulting integrals are $L^2$ inner products of functions in $\set{V}$ and their traces on each facet, we can use the exactness of the discrete inner products and the definition of the correction field $\fn{h}^{(\gamma,j)}$ as the divergence of the corresponding correction function $\fv{g}^{(\gamma,j)}$ to obtain
\begin{equation}\label{eq:vcjh_ibp}
\begin{aligned}
\Big\il \fn{v},\fn{h}^{(\gamma,j)} \Big\ir_{\mat{W}} &+ \sum_{\vc{\alpha} \in \set{M}} \frac{\fn{C}({\vc{\alpha}})}{\lvert\hat{\ug{\Omega}}\rvert}\bigg\il \frac{\partial^{\abs{\vc{\alpha}}}\fn{V}}{\partial \hat{x}_1^{\alpha_1} \cdots \partial \hat{x}_d^{\alpha_d}},\frac{\partial^{\abs{\vc{\alpha}}}\fn{h}^{(\gamma,j)}}{\partial \hat{x}_1^{\alpha_1} \cdots \partial \hat{x}_d^{\alpha_d}} \bigg\ir_{\mat{W}} \\ &= \sum_{\eta=1}^{N_f}\Big\il \fn{v}, \fv{g}^{(\gamma,j)} \cdot \hat{\vc{n}}^{(\eta)} \Big\ir_{\mat{W}^{(\eta)}}, \quad \forall \, \fn{v} \in \set{V}.
\end{aligned}
\end{equation}
Using the fact that all terms in (\ref{eq:vcjh_ibp}) are linear with respect to the test function $\fn{v}$ and the fact that the basis $\set{B}$ spans the space $\set{V}$, we can test with each basis function and apply the expansion to $\fn{h}^{(\gamma,j)} \in \set{V}$ in order to obtain
\begin{equation}\label{eq:vcjh_expanded}
\begin{multlined}
\sum_{i=1}^{N_p}\Bigg(\Big\il\phi^{(k)},\phi^{(i)} \Big\ir_{\mat{W}}\vspace{-0.5em} \\ +\sum_{\vc{\alpha} \in \set{M}} \frac{\fn{C}({\vc{\alpha}})}{\lvert\hat{\ug{\Omega}}\rvert}\bigg\il \frac{\partial^{\abs{\vc{\alpha}}} \phi^{(k)}}{\partial \hat{x}_1^{\alpha_1} \cdots \partial \hat{x}_d^{\alpha_d}},\frac{\partial^{\abs{\vc{\alpha}}} \phi^{(i)} }{\partial \hat{x}_1^{\alpha_1} \cdots \partial \hat{x}_d^{\alpha_d}} \bigg\ir_{\mat{W}}\Bigg) \el{L}_{ij}^{(\gamma)} \\ = \sum_{\eta=1}^{N_f}\Big\il \phi^{(k)},  \fv{g}^{(\gamma,j)} \cdot \hat{\vc{n}}^{(\eta)} \Big\ir_{\mat{W}^{(\eta)}}
\end{multlined}
\end{equation}
for all $k \in \{1,\ldots, N_p\}$. Recognizing the entries of $\mat{M}$ and $\mat{K}$ on the left-hand side and expanding the discrete inner product on the right-hand side, (\ref{eq:vcjh_expanded}) may be expressed as
\begin{equation}
\sum_{i=1}^{N_p}\qty(\el{M}_{ki} + \el{K}_{ki})\el{L}_{ij}^{(\gamma)} = \sum_{\eta=1}^{N_f}\sum_{i=1}^{N_\gamma}\sum_{l=1}^{N_\gamma}\el{V}_{ik}^{(\eta)}\el{W}_{il}^{(\eta)}\fv{g}^{(\gamma,j)}(\hat{\vc{x}}^{(\eta,l)}) \cdot \hat{\vc{n}}^{(\eta)}.
\end{equation}
Considering the above system of equations and using the property $
\fv{g}^{(\gamma,j)}(\hat{\vc{x}}^{(\eta,i)}) \cdot \hat{\vc{n}}^{(\eta)} = \delta_{\gamma\eta}\delta_{ij}$, we see that the $j^{\textrm{th}}$ column of the matrix $\mat{L}^{(\gamma)}$, containing the expansion coefficients for the corresponding correction field $\fn{h}^{(\gamma,j)} \in \set{V}$ in terms of the basis $\set{B}$, may be obtained by solving the linear system of equations given by $\sum_{i=1}^{N_p}(\el{M}_{ki} + \el{K}_{ki})\mat{L}_{ij}^{(\gamma)} = \sum_{i=1}^{N_\gamma}\el{V}_{ik}^{(\gamma)}\el{W}_{ij}^{(\gamma)}$ for all $k \in \{1,\ldots, N_p\}$, where $\mat{M}+\mat{K}$ is invertible under the assumption of positive-definiteness. As such a system is obtained from (\ref{eq:vcjh}) for all $j \in \{1,\ldots, N_\gamma\}$ and $\gamma \in \{1,\ldots, N_f\}$, the VCJH correction fields may be fully specified by defining $\mat{L}^{(\gamma)} \df (\mat{M}+\mat{K})^{-1}(\mat{V}^{(\gamma)})^\T\mat{W}^{(\gamma)}$.
\end{proof}
\begin{remark}\label{rmk:correction}
While the conditions on the correction functions in Theorem \ref{thm:vcjh} imply the existence of correction fields given by (\ref{eq:corr_field}) in terms of a particular form of $\mat{L}^{(\gamma)}$, the converse of such a statement does not hold in the general case. We therefore cannot necessarily associate a given correction field $\fn{h}^{(\gamma,j)}$ obtained from $\mat{L}^{(\gamma)}$ as in (\ref{eq:corr_field}) with a unique correction function $\fv{g}^{(\gamma,j)}$ satisfying the conditions of Theorem \ref{thm:vcjh}. In fact, correction functions satisfying such conditions have not, to the authors' knowledge, been explicitly constructed  for simplicial elements in two or more space dimensions.\footnote{Recalling the formulation in (\ref{eq:fr_lift}), as well as those appearing elsewhere in the literature (e.g.\ \cite[Eq.\ (4.35)]{castonguay_vincent_jameson_triangular_fr_11} and \cite[Eq.\ (29)]{williams_esfr_adv_diff_tetrahedra_13}) it is the correction fields $\fn{h}^{(\gamma,j)}$, and not the correction functions $\fv{g}^{(\gamma,j)}$, that appear explicitly in the implementation of a multidimensional FR scheme, and thus such a construction is not necessary.} The analysis of FR schemes based on properties of the lifting matrix, which we present in \textsection \ref{sec:analysis}, is therefore more general than that explicitly relying on properties of the correction functions, and, moreover, ensures that all mathematical objects involved in the proofs are well defined.
\end{remark}
To facilitate the analysis in \textsection \ref{sec:analysis}, the choices of $\set{M}$ and $\fn{C}$ are constrained by the following assumption.
\begin{assumption}\label{asm:k_mat}
The set of multi-indices $\set{M}$ is chosen such that $\mat{K}\mat{D}^{(m)} = \zeromat$ for all $m \in \{1,\ldots, d\}$, and the coefficients $C(\vc{\alpha})$ are chosen such that $\mat{M}+\mat{K}$ is SPD.
\end{assumption}
Conditions for Assumption \ref{asm:k_mat} to be satisfied depend on the particular choice of approximation space $\set{V}$ and are discussed further in \textsection \ref{app:k_mat}. If we begin with a set of correction functions satisfying (\ref{eq:vcjh}), the first condition in Assumption \ref{asm:k_mat} amounts to choosing $\set{M}$ such that the second term of the integrand is constant on $\hat{\ug{\Omega}}$. In such a case, we recover the fundamental assumption for energy stability for the VCJH schemes described by Vincent \etal \cite[Eqs.\ (3.31), (3.32)]{vincent_esfr_10} and Castonguay \etal \cite[Eq.\ (5.37)]{castonguay_vincent_jameson_triangular_fr_11}, which is given in the general case by
 \begin{equation}\label{eq:vcjh_constant}
 \int_{\hat{\ug{\Omega}}} \fv{g}^{(\gamma,j)}(\hat{\vc{x}}) \cdot \nabla\fn{v}(\hat{\vc{x}}) \, \dd\vc{x} - \sum_{\vc{\alpha} \in \set{M}} \fn{C}({\vc{\alpha}}) \frac{\partial^{\abs{\vc{\alpha}}}\fn{V}}{\partial \hat{x}_1^{\alpha_1} \cdots \partial \hat{x}_d^{\alpha_d}}\frac{\partial^{\abs{\vc{\alpha}}}\fn{h}^{(\gamma,j)}}{\partial \hat{x}_1^{\alpha_1} \cdots \partial \hat{x}_d^{\alpha_d}} = 0
 \end{equation}
for all $\fn{v} \in \set{V}$, where we have abused notation to treat the second term on the left-hand side as a constant.
 
\subsection{Connection between FR methods and strong-form filtered DG schemes}
Considering an FR method for which the correction fields are given as in Theorem \ref{thm:vcjh} with $\mat{K} = \zeromat$, as is obtained when (\ref{eq:vcjh}) is satisfied with $C(\vc{\alpha}) = 0$ for all $\vc{\alpha} \in \set{M}$, we may pre-multiply the resulting scheme in the form of (\ref{eq:fr_algebraic}) by $\mat{M}$ in order to obtain a strong-form DG method given by
\begin{equation}\label{eq:dg_strong}
\begin{multlined}
\mat{M}^{(\kappa)}\dv{\vc{u}^{(h,\kappa,e)}(t)}{t}+ \sum_{m=1}^d \mat{M}\mat{D}^{(m)}\mat{P} \vc{f}^{(h,\kappa,m,e)}(t)  \\ +  \sum_{\gamma=1}^{N_f}(\mat{V}^{(\gamma)})^\T\mat{W}^{(\gamma)} \qty(\vc{f}^{(*,\kappa,\gamma,e)}(t)  - \sum_{n=1}^d\hat{n}_n^{(\gamma)} \mat{V}^{(\gamma)}\mat{P}\vc{f}^{(h,\kappa,n,e)}(t)) = \vc{0}
\end{multlined}
\end{equation}
for all $\kappa \in \{1,\ldots, K\}$, $e \in \{1,\ldots, N_{\textrm{Eq}}\}$, and $t \in (0,T)$. In the above, we note that the projection is applied prior to differentiating each transformed flux component as well as prior to evaluating the normal component of the transformed flux on each facet of the reference element. 
\begin{remark}
It is also possible to compute the volume terms of a strong-form DG or FR method by differentiating the flux components in (\ref{eq:transformed_flux}) exactly using the chain rule and projecting afterwards, and to compute the facet terms by evaluating the normal component of the flux on each facet directly in terms of the numerical solution on $\set{S}^{(\gamma)}$ without first projecting.\footnote{As noted by Kopriva and Gassner \cite[\textsection 4]{kopriva_nodaldg_choosing_quadrature_weakform_10} in the context of the DGSEM-LGL, this modification to the facet terms has no effect when $N = N_p$ and $\set{S}^{(\gamma)} \subset \set{S}$ for all $\gamma \in \{1,\ldots, N_f\}$.} However, the theory in \textsection \ref{sec:analysis} does not apply to such alternative formulations, which, as discussed by Wang and Gao \cite[\textsection 3.3]{wang_gao_lcp_09} and Kopriva and Gassner \cite[\textsection 4]{kopriva_nodaldg_choosing_quadrature_weakform_10}, generally differ from those in (\ref{eq:fr_algebraic}) and (\ref{eq:dg_strong}) when applied to nonlinear or variable-coefficient problems and may therefore result in nonconservative strong-form discretizations which are not equivalent to their weak-form counterparts.
\end{remark}
For more general choices of $\mat{K}$ satisfying Assumption \ref{asm:k_mat}, it was recognized in the one-dimensional case by Allaneau and Jameson \cite{allaneau_jameson_dg_fr_11} that a strong-form linearly filtered DG method is recovered in the sense that the FR residual is obtained through pre-multiplication of the local DG residual resulting from (\ref{eq:dg_strong}) by a constant filter matrix $\mat{F}^{(\kappa)}$. Through transformation to a modal basis, it can be shown that such a filter acts only on polynomial modes of the highest degree contained within the space $\set{V}$ (see, for example, Allaneau and Jameson \cite[\textsection 3.2]{allaneau_jameson_dg_fr_11} and Williams and Jameson \cite[Appendix B]{williams_esfr_adv_diff_tetrahedra_13}). The following lemma illustrates such an equivalence in the general multidimensional case with arbitrary choices of basis and projection.
\begin{lemma}\label{lem:filtered_dg}
Provided that Assumptions \ref{asm:inner_prod}, \ref{asm:mass_invertible}, and \ref{asm:k_mat} hold, the FR method in (\ref{eq:fr_algebraic}) with $\mat{L}^{(\gamma)} \df (\mat{M}+\mat{K})^{-1}(\mat{V}^{(\gamma)})^\T\mat{W}^{(\gamma)}$ is equivalent to the strong-form DG method in (\ref{eq:dg_strong}) with the mass matrix $\mat{M}^{(\kappa)}$ replaced by $\mat{M}^{(\kappa)}(\mat{F}^{(\kappa)})^{-1}$, where the transformed filter matrix is given by $\mat{F}^{(\kappa)} \df (\mat{P}\mat{J}^{(\kappa)}\mat{V})^{-1}\mat{F}(\mat{P}\mat{J}^{(\kappa)}\mat{V})$ with $\mat{F} \df (\mat{I} + \mat{M}^{-1}\mat{K})^{-1}$.
\end{lemma}
\begin{proof}
Pre-multiplying (\ref{eq:fr_algebraic}) by $\mat{M}+\mat{K}$ and simplifying the first term by using $\mat{M}+\mat{K} = \mat{M}\mat{F}^{-1}$ and $\mat{M}^{(\kappa)} = \mat{M}(\mat{P}\mat{J}^{(\kappa)}\mat{V})$ to obtain $(\mat{M}+\mat{K})(\mat{P}\mat{J}^{(\kappa)}\mat{V}) = \mat{M}^{(\kappa)}(\mat{F}^{(\kappa)})^{-1}$, the formulation may be expressed equivalently as
\begin{equation}\label{eq:fr_premultiply}
\begin{multlined}
\mat{M}^{(\kappa)}(\mat{F}^{(\kappa)})^{-1}\dv{\vc{u}^{(h,\kappa,e)}(t)}{t} + \sum_{m=1}^d\mat{M}\mat{D}^{(m)}\mat{P} \vc{f}^{(h,\kappa,m,e)}(t)\\ + \sum_{m=1}^d\mat{K}\mat{D}^{(m)}\mat{P} \vc{f}^{(h,\kappa,m,e)}(t)  \\ +  \sum_{\gamma=1}^{N_f}(\mat{V}^{(\gamma)})^\T\mat{W}^{(\gamma)} \qty(\vc{f}^{(*,\kappa,\gamma,e)}(t)  - \sum_{n=1}^d\hat{n}_n^{(\gamma)} \mat{V}^{(\gamma)}\mat{P}\vc{f}^{(h,\kappa,n,e)}(t)) = \vc{0}.
\end{multlined}
\end{equation}
Noting that the third term  of (\ref{eq:fr_premultiply}) vanishes when $\mat{K}\mat{D}^{(m)} = \zeromat$ for all $m \in \{1,\ldots,d\}$, as is the case for methods satisfying Assumption \ref{asm:k_mat}, we therefore recover a filtered strong-form DG method differing from the unfiltered scheme in (\ref{eq:dg_strong}) only by the mass matrix.
\end{proof}
\begin{remark}
The above result is distinguished from that in Zwanenburg and Nadarajah \cite{zwanenburg_esfr_filtered_dg_16} due to the application of the filter to the entire semi-discrete residual, as opposed to only the facet terms.
\end{remark}
\section{Theoretical analysis}\label{sec:analysis}
\subsection{Equivalence between strong and weak forms}\label{sec:equiv}
Just as IBP is essential to obtaining the weak form of a partial differential equation from the strong form, the SBP property in (\ref{eq:sbp}) may be used analogously to transform a strong-form discretization into a weak-form discretization, and vice versa (as discussed in \cite[\textsection 8.1]{delrey_sbp_sat_review_14}). Although not relying explicitly on the matrix form of the SBP property in (\ref{eq:sbp}), Kopriva and Gassner proved in \cite{kopriva_nodaldg_choosing_quadrature_weakform_10} that for the DGSEM-LG and DGSEM-LGL on curvilinear hexahedral elements, the strong form in (\ref{eq:dg_strong}) is equivalent to the weak form in (\ref{eq:dg_algebraic}). Such an equivalence is extended to more general DG formulations with the following theorem.
\begin{theorem}\label{thm:equiv_dg}
Under Assumptions \ref{asm:sbp}, \ref{asm:inner_prod}, and \ref{asm:mass_invertible}, the strong-form DG method in (\ref{eq:dg_strong}) is equivalent to the weak-form DG method in (\ref{eq:dg_algebraic}).
\end{theorem}
\begin{proof}
Applying the SBP property in (\ref{eq:sbp}) to the matrix product $\mat{M}\mat{D}^{(m)}$ appearing in the second term of (\ref{eq:dg_strong}) for all $m \in \{1,\ldots, d\}$, we obtain
\begin{equation}\label{eq:dg_ibp}
\begin{multlined}
\mat{M}^{(\kappa)}\dv{\vc{u}^{(h,\kappa,e)}(t)}{t} \\+ \sum_{m=1}^d \qty(\sum_{\gamma=1}^{N_f}\hat{n}_m^{(\gamma)}(\mat{V}^{(\gamma)})^\T\mat{W}^{(\gamma)}\mat{V}^{(\gamma)} -(\mat{D}^{(m)})^\T\mat{M})\mat{P} \vc{f}^{(h,\kappa,m,e)}(t)  \\ +  \sum_{\gamma=1}^{N_f}(\mat{V}^{(\gamma)})^\T\mat{W}^{(\gamma)} \qty(\vc{f}^{(*,\kappa,\gamma,e)}(t)  - \sum_{n=1}^d\hat{n}_n^{(\gamma)} \mat{V}^{(\gamma)}\mat{P}\vc{f}^{(h,\kappa,n,e)}(t)) = \vc{0}.
\end{multlined}
\end{equation}
 Recognizing that the terms containing the normal component of the projected flux at each facet node cancel from the second and third terms of (\ref{eq:dg_ibp}) and noting that $\mat{M}\mat{P} = \mat{V}^\T\mat{W}$, we therefore recover the weak form in (\ref{eq:dg_algebraic}).
\end{proof}
Exploiting the connection between FR methods and filtered DG schemes established in Lemma \ref{lem:filtered_dg}, Theorem \ref{thm:equiv_dg} also implies that an algebraically equivalent weak form may be recovered from the FR method in (\ref{eq:fr_algebraic}), as is demonstrated with the following theorem.
\begin{theorem}\label{thm:fr_equiv}
Under Assumptions \ref{asm:sbp}, \ref{asm:inner_prod}, \ref{asm:mass_invertible}, and \ref{asm:k_mat}, the FR method in (\ref{eq:fr_algebraic}) with  $\mat{L}^{(\gamma)} \df (\mat{M}+\mat{K})^{-1}(\mat{V}^{(\gamma)})^\T\mat{W}^{(\gamma)}$is equivalent to
\begin{equation}\label{eq:fr_weak}
\begin{multlined}
\mat{M}^{(\kappa)}(\mat{F}^{(\kappa)})^{-1}\dv{\vc{u}^{(h,\kappa,e)}(t)}{t}  - \sum_{m=1}^d(\mat{D}^{(m)})^\T\mat{V}^\T\mat{W}\vc{f}^{(h,\kappa,m,e)}(t)  \\ +\sum_{\gamma=1}^{N_f}(\mat{V}^{(\gamma)})^\T\mat{W}^{(\gamma)} \vc{f}^{(*,\kappa,\gamma,e)}(t)  = \vc{0}
\end{multlined}
\end{equation}
for all $\kappa \in \{1,\ldots, K\}$, $e \in \{1,\ldots, N_{\textrm{Eq}}\}$, and $t \in (0,T)$, corresponding to a filtered DG scheme in weak form, with the filter matrix $\mat{F}^{(\kappa)}$ defined as in Lemma \ref{lem:filtered_dg}.
\end{theorem}
\begin{proof}
From Lemma \ref{lem:filtered_dg}, the FR method in (\ref{eq:fr_algebraic}) may be expressed as a strong-form filtered DG scheme given by
\begin{equation}
\begin{multlined}
\mat{M}^{(\kappa)}(\mat{F}^{(\kappa)})^{-1}\dv{\vc{u}^{(h,\kappa,e)}(t)}{t} + \sum_{m=1}^d \mat{M}\mat{D}^{(m)}\mat{P} \vc{f}^{(h,\kappa,m,e)}(t)  \\ +  \sum_{\gamma=1}^{N_f}(\mat{V}^{(\gamma)})^\T\mat{W}^{(\gamma)} \qty(\vc{f}^{(*,\kappa,\gamma,e)}(t)  - \sum_{n=1}^d\hat{n}_n^{(\gamma)} \mat{V}^{(\gamma)}\mat{P}\vc{f}^{(h,\kappa,n,e)}(t)) = \vc{0}.
\end{multlined}
\end{equation}
Applying the SBP property to $\mat{M}\mat{D}^{(m)}$ for all $m \in \{1,\ldots,d\}$ results in a formulation identical to that in (\ref{eq:dg_ibp}) with $\mat{M}^{(\kappa)}$ replaced by $\mat{M}^{(\kappa)}(\mat{F}^{(\kappa)})^{-1}$, and hence we may proceed similarly to recover the weak form in (\ref{eq:fr_weak}).
\end{proof}
\begin{remark}
Although the weak-form filtered and unfiltered DG methods in (\ref{eq:fr_weak}) and (\ref{eq:dg_algebraic}) are mathematically equivalent to the strong-form FR and DG schemes in (\ref{eq:fr_algebraic}) and (\ref{eq:dg_strong}), respectively, the weak-form discretizations have the potential for reduced computational expense relative to the corresponding strong forms due to the elimination of terms of the form $\sum_{m=1}^d\hat{n}_m^{(\gamma)}\mat{V}^{(\gamma)}\mat{P}\vc{f}^{(h,\kappa,m,e)}$ from the facet contributions to the semi-discrete residual. Assuming that the projected flux components $\mat{P}\vc{f}^{(h,\kappa,m,e)}$ are evaluated prior to computing the volume terms and that $\mat{V}^{(\gamma)}$ is brought outside of the sum, the use of the weak form leads to an elimination of $K\cdot N_\textrm{Eq} \cdot N_f$ dense matrix-vector multiplications of size $N_\gamma \times N_p$ unless $\set{B}$ is taken to be a Lagrange basis with solution nodes $\tilde{\set{S}}$ containing $\set{S}^{(\gamma)}$, in which case the interpolation/extrapolation from $\tilde{\set{S}}$ to $\set{S}^{(\gamma)}$ through pre-multiplication by $\mat{V}^{(\gamma)}$ is not required.\footnote{This is discussed by Kopriva and Gassner \cite[\textsection 7]{kopriva_nodaldg_choosing_quadrature_weakform_10} for the DGSEM-LG and DGSEM-LGL, where, unlike the DGSEM-LG, the DGSEM-LGL satisfies $\set{S}^{(\gamma)} \subset \tilde{\set{S}}$ for all $\gamma \in \{1,\ldots, N_f\}$.} 
\end{remark}
\begin{remark}
Within the present framework, the weak formulation of the VCJH family of FR methods proposed by Zwanenburg and Nadarajah \cite[Eq.\ (3.7)]{zwanenburg_esfr_filtered_dg_16} may be expressed as
\begin{equation}\label{eq:zwanenburg}
\begin{multlined}
\mat{M}^{(\kappa)}\dv{\vc{u}^{(h,\kappa,e)}(t)}{t}  - \sum_{m=1}^d(\mat{D}^{(m)})^\T\mat{V}^\T\mat{W}\vc{f}^{(h,\kappa,m,e)}(t)  \\ +\sum_{\gamma=1}^{N_f}\mat{F}^\T(\mat{V}^{(\gamma)})^\T\mat{W}^{(\gamma)} \vc{f}^{(*,\kappa,\gamma,e)}(t)  \\ + \sum_{\gamma=1}^{N_f}\sum_{m=1}^d \hat{n}_m^{(\gamma)}(\mat{I}-\mat{F}^\T)(\mat{V}^{(\gamma)})^\T\mat{W}^{(\gamma)}\mat{V}^{(\gamma)}\mat{P}\vc{f}^{(h,\kappa,m,e)}(t)  = \vc{0},
\end{multlined}
\end{equation} 
which is algebraically equivalent to the strong formulation in (\ref{eq:fr_algebraic}) under the assumptions of Theorem \ref{thm:fr_equiv}, but has a higher computational cost than the weak-form DG method in (\ref{eq:dg_algebraic}) when $\mat{K} \neq \zeromat$. This does not, however, imply that VCJH schemes with $\mat{K} \neq \zeromat$ are inherently more computationally expensive than DG methods, since by instead using the weak formulation in (\ref{eq:fr_weak}), we are, in fact, able to attain equivalent computational expense to (\ref{eq:dg_algebraic})  with general choices of $\mat{K}$ satisfying Assumption \ref{asm:k_mat}.
\end{remark}
\subsection{Conservation}\label{sec:conservation}
The proofs of conservation in this section and the proof of energy stability in \textsection \ref{sec:stability} make use of the following assumption, which ensures that the mesh is conforming, with facet nodes aligning in physical space such that the numerical flux may be evaluated pointwise without the need for additional interpolation operators or mortar elements (see, for example, Laughton \etal \cite{laughton_comparison_of_interpolation_21} for an overview and comparison of such techniques).
\begin{assumption}\label{asm:conformal}
	The mesh satisfies all assumptions in \textsection \ref{sec:mesh}, and for each pair of element indices $\kappa,\nu \in \{1,\ldots, K\}$ with $\kappa \neq \nu$ such that $\partial\ug{\Omega}^{(\kappa)} \cap \partial\ug{\Omega}^{(\nu)} \neq \emptyset$, there exist facet indices $\gamma,\eta \in \{1,\ldots, N_f\}$ such that either $\ug{\Gamma}^{(\kappa,\gamma)} = \ug{\Gamma}^{(\nu,\eta)}$, or the facets are connected via periodic boundary conditions, in which case $\ug{\Gamma}^{(\nu,\eta)}$ is the image of $\ug{\Gamma}^{(\kappa,\gamma)}$ under a pure translation, with an opposite orientation in the sense that $\vc{n}^{(\nu,\eta)} = -\vc{n}^{(\kappa,\gamma)}$. Moreover, for such corresponding facets, the nodes align up to some permutation $\tau^{(\kappa,\gamma)} : \{1,\ldots,N_\gamma\} \to \{1,\ldots,N_\gamma\}$ such that in addition to the requirement that $N_\eta = N_\gamma$, the relation
	\begin{equation}\label{eq:nodes_match}
	\fv{X}^{(\kappa)}(\hat{\vc{x}}^{(\gamma,i)}) = \fv{X}^{(\nu)}(\hat{\vc{x}}^{(\eta,\tau^{(\kappa,\gamma)}(i))})
	\end{equation}
	is satisfied for all $i \in \{1,\ldots,N_\gamma\}$, with an additional translation in the periodic case. Given the corresponding permutation matrix $\mat{T}^{(\kappa,\gamma)} \df [\vc{e}^{(\tau^{(\kappa,\gamma)}(1))}, \ldots,\vc{e}^{(\tau^{(\kappa,\gamma)}(N_\gamma))} ]^\T$, where $\vc{e}^{(i)} \in \R^{N_\gamma}$ is the $i^{\textrm{th}}$ standard basis vector, the facet operators are then assumed to satisfy
	\begin{equation}\label{eq:weights_match}
	\mat{W}^{(\eta)}\mat{J}^{(\nu,\eta)} = (\mat{T}^{(\kappa,\gamma)})^\T\mat{W}^{(\gamma)}\mat{J}^{(\kappa,\gamma)} \mat{T}^{(\kappa,\gamma)} ,
	\end{equation}
	where we have defined $\mat{J}^{(\kappa,\gamma)} \df \textrm{diag}([\mathcal{J}^{(\kappa,\gamma)}(\hat{\vc{x}}^{(\gamma,1)}), \ldots, \mathcal{J}^{(\kappa,\gamma)}(\hat{\vc{x}}^{(\gamma,N_\gamma)})]^\T)$.
\end{assumption}
The local and global conservation properties for weak-form unfiltered and filtered DG methods then follow from taking a constant test function, which is demonstrated algebraically with the following theorem.
\begin{theorem}\label{thm:conservation}
Under Assumptions \ref{asm:inner_prod}, \ref{asm:mass_invertible}, and \ref{asm:k_mat}, the weak-form unfiltered and filtered DG methods in (\ref{eq:dg_algebraic}) and (\ref{eq:fr_weak}), respectively, the latter with a filter matrix given as in Lemma \ref{lem:filtered_dg}, are locally conservative, satisfying
\begin{equation}\label{eq:local_cons}
\dv{t}\qty(\vc{1}^\T\mat{W}\mat{J}^{(\kappa)}\mat{V}\vc{u}^{(h,\kappa,e)}(t)) + \sum_{\gamma=1}^{N_f}\vc{1}^\T\mat{W}^{(\gamma)}\vc{f}^{(*,\kappa,\gamma,e)}(t) = 0
\end{equation}
for all $\kappa \in \{1,\ldots, K\}$, $e \in \{1,\ldots, N_{\textrm{Eq}}\}$, and $t \in (0,T)$, when $\set{P}_0(\hat{\ug{\Omega}}) \subseteq \set{V}$. Moreover, for a numerical flux satisfying the conservation property $\fv{f}^*(\fv{u},\fv{v},\vc{n}) =-\fv{f}^*(\fv{v},\fv{u},-\vc{n})$ for all $\fv{u},\fv{v} \in \ug{\Upsilon}$ and $\vc{n}\in \mathbb{S}^{d-1}$, the methods are globally conservative, satisfying
\begin{equation}\label{eq:global_cons}
\dv{t}\qty(\sum_{\kappa=1}^K\vc{1}^\T\mat{W}\mat{J}^{(\kappa)}\mat{V}\vc{u}^{(h,\kappa,e)}(t))\ + \ \sum_{\mathclap{(\kappa,\gamma) \in \set{K}_{\partial\ug{\Omega}} }} \ \vc{1}^\T\mat{W}^{(\gamma)}\vc{f}^{(*,\kappa,\gamma,e)}(t) = 0
\end{equation}
for all $e \in \{1,\ldots, N_{\textrm{Eq}}\}$ and $t \in (0,T)$, under Assumption \ref{asm:conformal}.
\end{theorem}
\begin{proof}
Considering the filtered DG formulation in (\ref{eq:fr_weak}), which recovers the standard DG method in (\ref{eq:dg_algebraic}) when $\mat{K} = \zeromat$, we first consider a single element and pre-multiply by $(\mat{P}\vc{1})^\T$. Recalling from Lemma \ref{lem:filtered_dg} and the definition of $\mat{M}$ that $\mat{M}^{(\kappa)}(\mat{F}^{(\kappa)})^{-1} = (\mat{M}+\mat{K})(\mat{P}\mat{J}^{(\kappa)}\mat{V}) = \mat{V}^\T\mat{W}\mat{J}^{(\kappa)}\mat{V} + \mat{K}\mat{P}\mat{J}^{(\kappa)}\mat{V}$, we therefore obtain
\begin{equation}\label{eq:fr_weak_1}
\begin{multlined}
(\mat{P}\vc{1})^\T\mat{V}^\T\mat{W}\mat{J}^{(\kappa)}\mat{V}\dv{\vc{u}^{(h,\kappa,e)}(t)}{t} + (\mat{P}\vc{1})^\T\mat{K}\mat{W}\mat{V}\dv{\vc{u}^{(h,\kappa,e)}(t)}{t} \\ - \sum_{m=1}^d(\mat{P}\vc{1})^\T(\mat{D}^{(m)})^\T\mat{V}^\T\mat{W}\vc{f}^{(h,\kappa,m,e)}(t)  \\ +\sum_{\gamma=1}^{N_f}(\mat{P}\vc{1})^\T(\mat{V}^{(\gamma)})^\T\mat{W}^{(\gamma)} \vc{f}^{(*,\kappa,\gamma,e)}(t)  = \vc{0}.
\end{multlined}
\end{equation}
 Since $\set{P}_0(\hat{\ug{\Omega}}) \subseteq \set{V}$, the vector $\mat{P}\vc{1} \in \R^{N_p}$ contains the expansion coefficients for a constant function equal to unity on the reference element, which we denote as $1 \in \set{P}_0(\hat{\ug{\Omega}})$, in terms of the basis $\set{B}$. This simplifies the first and last terms on the left-hand side of (\ref{eq:fr_weak_1}), since $\mat{V}(\mat{P}\vc{1}) = \vc{1}$ and $\mat{V}^{(\gamma)}(\mat{P}\vc{1}) = \vc{1}$ for all $\gamma \in \{1,\ldots, N_\gamma\}$, and eliminates the second and third terms, since $\mat{K}(\mat{P}\vc{1}) = \vc{0}$ and $\mat{D}^{(m)}(\mat{P}\vc{1}) = \vc{0}$ for all $m \in \{1,\ldots, d\}$. We therefore recover the statement of local conservation in (\ref{eq:local_cons}). For global conservation, we sum (\ref{eq:local_cons}) over all $\kappa \in \{1,\ldots, K\}$ and split the net contribution $F_{\textrm{net}}^{(\kappa,\gamma,e)}(t) \df \vc{1}^\T\mat{W}^{(\gamma)}\vc{f}^{(*,\kappa,\gamma,e)}(t) + \vc{1}^\T\mat{W}^{(\eta)}\vc{f}^{(*,\nu,\eta,e)}(t)$ from each shared (i.e.\ interior or periodic) interface equally between the index pairs $(\kappa,\gamma), (\nu,\eta) \in \set{K}_\ug{\Omega}$ corresponding to facets $\ug{\Gamma}^{(\kappa,\gamma)}$ and $\ug{\Gamma}^{(\nu,\eta)}$ which are either coincident or connected via periodic boundary conditions, resulting in the global balance
\begin{equation}\label{eq:global_cons_sum}
\begin{multlined}
\dv{t}\qty(\sum_{\kappa=1}^K\vc{1}^\T\mat{W}\mat{J}^{(\kappa)}\mat{V}\vc{u}^{(h,\kappa,e)}(t)) + \sum_{(\kappa,\gamma) \in \set{K}_\ug{\Omega}} \frac{1}{2} F_{\textrm{net}}^{(\kappa,\gamma,e)}(t)  \\ +  \sum_{\mathclap{(\kappa,\gamma) \in \set{K}_{\partial\ug{\Omega}}}} \ \vc{1}^\T\mat{W}^{(\gamma)}\vc{f}^{(*,\kappa,\gamma,e)}(t) = \vc{0}.
\end{multlined}
\end{equation}
As (\ref{eq:weights_match}) holds under Assumption \ref{asm:conformal} and any permutation matrix $\mat{T}^{(\kappa,\gamma)}$ satisfies $\mat{T}^{(\kappa,\gamma)}\vc{1} = \vc{1}$, the contribution from each interface may be simplified as
\begin{equation}\label{eq:f_net}
\begin{aligned}
F_{\textrm{net}}^{(\kappa,\gamma,e)}(t) &=\vc{1}^\T\mat{W}^{(\gamma)} \vc{f}^{(*,\kappa,\gamma,e)}(t) \\ & \qquad\qquad +  \vc{1}^\T(\mat{T}^{(\kappa,\gamma)})^\T\mat{W}^{(\gamma)} \mat{J}^{(\kappa,\gamma)}\mat{T}^{(\kappa,\gamma)}(\mat{J}^{(\nu,\eta)})^{-1}\vc{f}^{(*,\nu,\eta,e)}(t) \\
 &= \vc{1}^\T\mat{W}^{(\gamma)}\big(\vc{f}^{(*,\kappa,\gamma,e)}(t) + \mat{J}^{(\kappa,\gamma)}\mat{T}^{(\kappa,\gamma)}(\mat{J}^{(\nu,\eta)})^{-1}\vc{f}^{(*,\nu,\eta,e)}(t)\big).
 \end{aligned}
\end{equation}
Using (\ref{eq:nodes_match}) and the fact that the outward unit normal vectors for an interior or periodic interface corresponding to $(\kappa,\gamma),(\nu,\eta) \in \set{K}_{\ug{\Omega}}$ satisfy $\vc{n}^{(\nu,\eta)} = -\vc{n}^{(\kappa,\gamma)}$, the terms in parentheses on the last line of (\ref{eq:f_net}) are given by
\begin{equation}
\vc{f}^{(*,\kappa,\gamma,e)}(t) = \mat{J}^{(\kappa,\gamma)}\mqty[\fn{f}_e^*\big(\fv{u}^{(h,\kappa)}(\hat{\vc{x}}^{(\gamma,1)},t), \fv{u}^{(h,\nu)}(\hat{\vc{x}}^{(\eta,\tau^{(\kappa,\gamma)}(1))},t),\\ \qquad\qquad\qquad\quad \vc{n}^{(\kappa,\gamma)}(\fv{X}^{(\kappa)}(\hat{\vc{x}}^{(\gamma,1)}))\big) \\ \vdots \\ \fn{f}_e^*\big(\fv{u}^{(h,\kappa)}(\hat{\vc{x}}^{(\gamma,N_\gamma)},t), \fv{u}^{(h,\nu)}(\hat{\vc{x}}^{(\eta,\tau^{(\kappa,\gamma)}(N_\gamma))},t),\\ \qquad\qquad\qquad\quad \vc{n}^{(\kappa,\gamma)}(\fv{X}^{(\kappa)}(\hat{\vc{x}}^{(\gamma,N_\gamma)}))\big)]
\end{equation}
and
\begin{equation}
\begin{aligned}
 \mat{J}^{(\kappa,\gamma)}\mat{T}^{(\kappa,\gamma)}&(\mat{J}^{(\nu,\eta)})^{-1}\vc{f}^{(*,\nu,\eta,e)}(t) \\  &=\mat{J}^{(\kappa,\gamma)}\mqty[\fn{f}_e^*\big(\fv{u}^{(h,\nu)}(\hat{\vc{x}}^{(\eta,\tau^{(\kappa,\gamma)}(1))},t),\fv{u}^{(h,\kappa)}(\hat{\vc{x}}^{(\gamma,1)},t),\\ \qquad\qquad\qquad\quad -\vc{n}^{(\kappa,\gamma)}(\fv{X}^{(\kappa)}(\hat{\vc{x}}^{(\gamma,1)}))\big) \\ \vdots \\ \fn{f}_e^*\big(\fv{u}^{(h,\nu)}(\hat{\vc{x}}^{(\eta,\tau^{(\kappa,\gamma)}(N_\gamma))},t),\fv{u}^{(h,\kappa)}(\hat{\vc{x}}^{(\gamma,N_\gamma)},t),\\ \qquad\qquad\qquad\quad -\vc{n}^{(\kappa,\gamma)}(\fv{X}^{(\kappa)}(\hat{\vc{x}}^{(\gamma,N_\gamma)}))\big)],
\end{aligned}
\end{equation}
which sum to zero as a result of the conservation property of the numerical flux. Each term of the second sum in (\ref{eq:global_cons_sum}) then vanishes, resulting in the statement of global conservation in (\ref{eq:global_cons}).
\end{proof} 
The relation in (\ref{eq:local_cons}) represents a quadrature-based approximation of the integral balance on the physical element $\ug{\Omega}^{(\kappa)} \in \set{T}_h$, which may be obtained from the differential form of the conservation law in (\ref{eq:pde}) using the divergence theorem. On the reference element, the divergence theorem may be expressed under the discrete inner products for vector fields with components belonging to the approximation space as
\begin{equation}\label{eq:disc_div_thm}
\big\il 1, \nabla \cdot \fv{v} \big\ir_{\mat{W}} = \sum_{\gamma=1}^{N_f} \big\il 1, \fv{v} \cdot \hat{\vc{n}}^{(\gamma)} \big\ir_{\mat{W}^{(\gamma)}}, \quad \forall \, \fv{v} \in \set{V}^d,
\end{equation}
which is satisfied with Assumption \ref{asm:sbp} and $\set{P}_0(\hat{\ug{\Omega}}) \subseteq \set{V}$ as sufficient but not necessary conditions.\footnote{Considering, for example, the approach outlined for the total-degree polynomial space in \textsection \ref{app:quadrature}, volume quadrature rules of degree $p-1$ or greater and facet quadrature rules of degree $p$ or greater do not meet the requirements for Assumption \ref{asm:sbp}, but nonetheless satisfy (\ref{eq:disc_div_thm}).} We now show that this property allows for the extension of Theorem \ref{thm:conservation} to strong formulations.
\begin{theorem}\label{thm:conservation_strong}
Theorem \ref{thm:conservation} is valid for the strong-form FR scheme in (\ref{eq:fr_algebraic}) with $\mat{L}^{(\gamma)} \df (\mat{M}+\mat{K})^{-1}(\mat{V}^{(\gamma)})^\T\mat{W}^{(\gamma)}$ and the strong-form DG scheme in (\ref{eq:dg_strong}), provided that the discrete divergence theorem in (\ref{eq:disc_div_thm}) is satisfied.
\end{theorem}
\begin{proof}
Expressing (\ref{eq:disc_div_thm}) in matrix form by expanding in terms of the basis $\set{B}$ to obtain $\sum_{m=1}^d(\mat{P}\vc{1})^\T\mat{M}\mat{D}^{(m)} = \sum_{\gamma=1}^{N_f}\sum_{m=1}^d\hat{n}_m^{(\gamma)} (\mat{P}\vc{1})^\T(\mat{V}^{(\gamma)})^\T \mat{W}^{(\gamma)}\mat{V}^{(\gamma)}$,
we may pre-multiply the FR method in (\ref{eq:fr_algebraic}) by $(\mat{P}\vc{1})^\T(\mat{M}+\mat{K})$, which for $\mat{K} = \zeromat$ is equivalent to pre-multiplying the strong-form DG method in (\ref{eq:dg_strong}) by $(\mat{P}\vc{1})^\T$, and use the above relation to arrive at (\ref{eq:fr_weak_1}). The remainder of the proof is therefore identical to that of Theorem \ref{thm:conservation}, and the result follows.
\end{proof}

\begin{remark}\label{rmk:cons_functional}
For periodic boundary conditions, Theorems \ref{thm:conservation} and \ref{thm:conservation_strong} imply that for any $e \in \{1,\ldots, K\}$, the functional $\opr{C}(\fn{u}_e^h(\cdot,t)) \df \sum_{\kappa=1}^K\vc{1}^\T\mat{W}\mat{J}^{(\kappa)}\mat{V}\vc{u}^{(h,\kappa,e)}(t)$, approximating the integral $\int_{\ug{\Omega}} \fn{u}_e^h\xt\, \dd \vc{x}$ is invariant over all $t \in (0,T)$. The invariance of such a quantity therefore provides a useful test for global conservation, which we employ for the numerical experiments in \textsection \ref{sec:numerical}.
\end{remark}
\begin{remark}
Following the discussion in Remark \ref{rmk:correction}, we note that the proofs of conservation in the FR literature (for example, those of Huynh \cite{huynh_FR_07} and Castonguay \etal \cite[\textsection 4.2]{castonguay_vincent_jameson_triangular_fr_11}) rely explicitly on properties of the correction functions $\fv{g}^{(\gamma,j)}$, notably that $
\fv{g}^{(\gamma,j)}(\hat{\vc{x}}^{(\eta,i)}) \cdot \hat{\vc{n}}^{(\eta)} = \delta_{\gamma\eta}\delta_{ij}$, in order to show that the normal trace of the corrected flux in physical space is continuous at all interfaces. Our proof, in contrast, is instead based only on well-defined algebraic properties of the correction fields encoded in the lifting matrices.
\end{remark}

\subsection{Energy stability}\label{sec:stability}
While the theory in \textsection \ref{sec:equiv} and \textsection \ref{sec:conservation} applies to DG and FR methods for systems in the form of (\ref{eq:pde}) on general curvilinear elements, the analysis in this section is restricted to the particular case of the constant-coefficient linear advection equation in (\ref{eq:advection}) on meshes consisting of affinely mapped polytopes. We therefore require the following assumption.
\begin{assumption}\label{asm:affine}
In addition to the assumptions of \textsection \ref{sec:mesh}, the mapping $\fv{X}^{(\kappa)}$ is  affine for all $\kappa \in \{1,\ldots, K\}$, and hence we may abuse notation to define $\mat{\mathcal{J}}^{(\kappa)} \in \R^{d\times d}$, $\mathcal{J}^{(\kappa)} \in \R_+$, $\mathcal{J}^{(\kappa,\gamma)} \in \R_+$, and $\vc{n}^{(\kappa,\gamma)} \in \mathbb{S}^{d-1}$ as taking the values of the corresponding constant functions.
\end{assumption}
Under Assumption \ref{asm:affine}, the matrix $\mat{P}\mat{J}^{(\kappa)}\mat{V}$ pre-multiplying the time derivative in (\ref{eq:fr_algebraic}) corresponds simply to a scaling by a constant factor of $\mathcal{J}^{(\kappa)}$. This ensures that the modified mass matrix $\mat{M}^{(\kappa)}(\mat{F}^{(\kappa)})^{-1}$ corresponding to a VCJH method or the equivalent filtered DG scheme is SPD under Assumption \ref{asm:k_mat}, allowing for a norm to be defined on the global solution space $\set{V}_h$, as shown with the following lemma.
\begin{lemma}\label{lem:norm}
The functional defined by $\lVert\fn{v}\rVert_{\set{V}_h} \df \sqrt{\sum_{\kappa=1}^K(\vc{v}^{(\kappa)})^\T\mat{M}^{(\kappa)}(\mat{F}^{(\kappa)})^{-1}\vc{v}^{(\kappa)}}$ for any $\fn{v} \in \set{V}_h$ given by $\fn{v}(\vc{x}) \df \bigoplus_{\kappa=1}^K \el{v}_i^{(\kappa)}\phi^{(i)}\big((\fv{X}^{(\kappa)})^{-1}(\vc{x})\big)$, with the space $\set{V}_h$ defined as in (\ref{eq:solution_space}), is a norm under Assumptions \ref{asm:k_mat} and \ref{asm:affine}.
\end{lemma}
\begin{proof}
As $\set{B}_h \df \{ \phi^{(i)} \circ (\fv{X}^{(\kappa)})^{-1} : (\kappa,i) \in \{1,\ldots, K \} \times \{1,\ldots, N_p\}\}$ is a basis for the space $\set{V}_h$, expanding $\fn{v} \in \set{V}_h$ as above defines an isomorphism between the vector spaces $\set{V}_h$ and $\R^{K \cdot N_p}$. Using such an isomorphism along with the fact that the matrix $\mat{M}^{(\kappa)}(\mat{F}^{(\kappa)})^{-1} = \mathcal{J}^{(\kappa)}(\mat{M}+\mat{K})$ is SPD for every $\kappa \in \{1,\ldots, K\}$ under the present assumptions allows us to conclude, as outlined in Remark \ref{rmk:bijective} for $\il\cdot,\cdot\ir_{\mat{W}}$, that the bilinear form defined by $\il\fn{u},\fn{v}\ir_{\set{V}_h} \df \sum_{\kappa=1}^K(\vc{u}^{(\kappa)})^\T\mat{M}^{(\kappa)}(\mat{F}^{(\kappa)})^{-1}\vc{v}^{(\kappa)}$ for any $\fn{u},\fn{v} \in \set{V}_h$ given by $\fn{u}(\vc{x}) \df \bigoplus_{\kappa=1}^K \el{u}_i^{(\kappa)}\phi^{(i)}\big((\fv{X}^{(\kappa)})^{-1}(\vc{x})\big)$ and $\fn{v}(\vc{x}) \df \bigoplus_{\kappa=1}^K \el{v}_i^{(\kappa)}\phi^{(i)}\big((\fv{X}^{(\kappa)})^{-1}(\vc{x})\big)$ is an inner product for the space $\set{V}_h$. As every inner product induces a corresponding norm on the same space, the desired result then follows from expresseing $\lVert \fn{v} \rVert_{\set{V}_h}$ as $\sqrt{\il \fn{v},\fn{v}\ir_{\set{V}_h}}$.
\end{proof}

\begin{remark}
When the conditions of Theorem \ref{thm:vcjh} are met, in particular, when $\il\fn{u},\fn{v}\ir_{\mat{W}} = \il\fn{u},\fn{v}\ir_{\hat{\ug{\Omega}}}$ for all $\fn{u},\fn{v} \in \set{V}$, we accordingly obtain 
\begin{equation}\label{eq:sobolev_norm}
\begin{multlined}
\lVert\fn{v}\rVert_{\set{V}_h}^2  = \\ \sum_{\kappa=1}^K \int_{\hat{\ug{\Omega}}}\Bigg( \big(\fn{v}(\fv{X}^{(\kappa)}(\hat{\vc{x}}))\big)^2 + \sum_{\vc{\alpha} \in \set{M}} \frac{\fn{C}({\vc{\alpha}})}{\lvert\hat{\ug{\Omega}}\rvert}\bigg(\frac{\partial^{\abs{\vc{\alpha}}} \fn{v}(\fv{X}^{(\kappa)}(\hat{\vc{x}}))}{\partial \hat{x}_1^{\alpha_1} \cdots \partial \hat{x}_d^{\alpha_d}}\bigg)^2\Bigg) \mathcal{J}^{(\kappa)}(\hat{\vc{x}}) \, \dd \hat{\vc{x}},
\end{multlined}
\end{equation}
which is the broken Sobolev-type norm introduced by Jameson \cite{jameson_sd_stability_10} in the context of one-dimensional spectral-difference methods and adopted in several stability proofs for FR schemes (e.g.\ those in Vincent \etal \cite{vincent_esfr_10} and Castonguay \etal \cite{castonguay_vincent_jameson_triangular_fr_11}). The stability theory for VCJH schemes, which is recovered within the present framework as a consequence of Theorem \ref{thm:vcjh}, is therefore obtained as a special case of that developed in this section.
\end{remark}
In order for advection problems on bounded domains to be well posed in the continuous case, boundary conditions must be prescribed appropriately on portions of $\partial\ug{\Omega}$ such that only incoming waves are specified. In the discrete case, we follow a similar approach and impose boundary and interface conditions weakly in a manner consistent with the physics of the problem through an appropriate choice of numerical flux. This is analogous to the use of simultaneous approximation terms, which are penalty terms added to strong-form discretizations at boundaries and interfaces in order to obtain stable approximations of initial-boundary-value problems with SBP operators,\footnote{In the present framework, the third term of (\ref{eq:fr_algebraic}) may be viewed as a SAT, and serves the same purpose in the stability analysis.} and were introduced in the finite-difference context by Carpenter \etal \cite{carpenter_time_stable_bcs_94}, extending the penalty approach developed for spectral methods by Funaro and Gottlieb \cite{funaro_gottlieb_boundary_conditions_pseudospectral_88}.
\begin{assumption}\label{asm:advection}
For the linear advection equation in (\ref{eq:advection}), the (non-periodic) boundary $\partial\ug{\Omega}$ may be partitioned into disjoint subsets $\partial\ug{\Omega}^- \df \{\vc{x} \in \partial\ug{\Omega} : \vc{a}\cdot\vc{n}(\vc{x}) \leq 0 \}$ and $\partial\ug{\Omega}^+ \df \partial\ug{\Omega} \setminus \partial\ug{\Omega}^-$, the former on which we prescribe the boundary condition $\fn{u}\xt = \fn{b}\xt$. The index set $\set{K}_{\partial\ug{\Omega}}$ is then further partitioned into the subsets $\set{K}_{\partial\ug{\Omega}^{-}} \df \{(\kappa,\gamma) \in \set{K}_{\partial\ug{\Omega}} : \ug{\Gamma}^{(\kappa,\gamma)} \subseteq \partial\ug{\Omega}^-\}$ and $\set{K}_{\partial\ug{\Omega}^+} \df \set{K}_{\partial\ug{\Omega}} \setminus \set{K}_{\partial\ug{\Omega}^{-}}$, while the numerical flux is assumed to take the form
\begin{equation}\label{eq:num_flux_adv}
\fn{f}^*(\fn{u},\fn{v},\vc{n}) \df \frac{1}{2}\sum_{m=1}^d n_m (\fn{f}^{(m)}(\fn{u}) + \fn{f}^{(m)}(\fn{v})) - \frac{\lambda}{2}\lvert\vc{a}\cdot \vc{n}\rvert (\fn{v}-\fn{u}),
\end{equation}
where the parameter $\lambda \in \R_{\geq 0}$ takes the same value on either side of an interior interface, and is taken to be unity for any facet lying on $\partial\ug{\Omega}$. 
\end{assumption}
The numerical flux in (\ref{eq:num_flux_adv}) recovers an upwind flux for $\lambda = 1$ and a central flux for $\lambda = 0$, corresponding to upwind and symmetric SATs, respectively, in SBP terminology. We are now equipped to present the following generalized energy stability result for DG and FR schemes in strong and weak form applied to the linear advection equation, in which (as discussed, for example, in Gustafsson \etal \cite[Ch.\ 11]{gustafsson13}) we show that the discrete norm $\lVert \fn{u}^h(\cdot,t) \rVert_{\set{V}_h}$ may only increase in time due to contributions arising from the inflow boundary. 
\begin{theorem}\label{thm:stability}
Under Assumptions \ref{asm:sbp}, \ref{asm:inner_prod}, \ref{asm:k_mat}, \ref{asm:conformal}, \ref{asm:affine}, and \ref{asm:advection}, the DG method in (\ref{eq:dg_algebraic}) and the FR method in (\ref{eq:fr_algebraic}) with $\mat{L}^{(\gamma)} \df (\mat{M}+\mat{K})^{-1}(\mat{V}^{(\gamma)})^\T\mat{W}^{(\gamma)}$ applied to the linear advection equation satisfy the discrete energy balance
\begin{equation}\label{eq:energy}
\begin{multlined}
\dv{t}\qty(\sum_{\kappa=1}^K(\vc{u}^{(h,\kappa)}(t))^\T\mat{M}^{(\kappa)}(\mat{F}^{(\kappa)})^{-1}\vc{u}^{(h,\kappa)}(t))  \\ \leq  \quad \sum_{\mathclap{(\kappa,\gamma) \in \set{K}_{\partial\ug{\Omega}^-}}} \big\lvert \vc{a} \cdot \vc{n}^{(\kappa,\gamma)} \big\rvert(\vc{b}^{(\kappa,\gamma)}(t))^\T\mat{W}^{(\gamma)}\mat{J}^{(\kappa,\gamma)}\vc{b}^{(\kappa,\gamma)}(t)
\end{multlined}
\end{equation}
for all $t \in (0,T)$, where $\vc{b}^{(\kappa,\gamma)}(t) \df [\fn{b}(\fv{X}^{(\kappa)}(\hat{\vc{x}}^{(\gamma,1)},t)),\ldots,\fn{b}(\fv{X}^{(\kappa)}(\hat{\vc{x}}^{(\gamma,N_\gamma)},t))]^\T$ for $(\kappa,\gamma) \in \set{K}_{\partial\ug{\Omega}^-}$, implying that the methods are energy stable in the norm $\lVert \cdot \rVert_{\set{V}_h}$.
\end{theorem}
\begin{proof}
Energy stability will be demonstrated for the strong-form FR method in (\ref{eq:fr_algebraic}), which is equivalent to the filtered weak-form DG scheme in (\ref{eq:fr_weak}) under the present assumptions as a result of Theorem \ref{thm:fr_equiv}, recovering the standard DG method in (\ref{eq:dg_algebraic}) with $\mat{K} = \zeromat$. Beginning with a single element and pre-multiplying (\ref{eq:fr_algebraic}) by $(\vc{u}^{(h,\kappa)}(t))^\T(\mat{M}+\mat{K})$, we may use the definition of the filter matrix and apply the chain rule in time  to obtain
\begin{equation}\label{eq:local_terms}
\begin{multlined}
\dv{t}\bigg(\frac{1}{2}(\vc{u}^{(h,\kappa)}(t))^\T\mat{M}^{(\kappa)}(\mat{F}^{(\kappa)})^{-1}\vc{u}^{(h,\kappa)}(t)\bigg)= \\ -\underbrace{\sum_{m=1}^d (\vc{u}^{(h,\kappa)}(t))^\T(\mat{M}+\mat{K})\mat{D}^{(m)}\mat{P} \vc{f}^{(h,\kappa,m)}(t)}_{=:\, \fn{E}_{\textrm{I}}^{(\kappa)}(t)}  \\ + \sum_{\gamma=1}^{N_f}\underbrace{\sum_{m=1}^d\hat{n}_m^{(\gamma)}(\vc{u}^{(h,\kappa)}(t))^\T(\mat{V}^{(\gamma)})^\T\mat{W}^{(\gamma)}  \mat{V}^{(\gamma)}\mat{P}\vc{f}^{(h,\kappa,m)}(t)}_{=: \, \fn{E}_{\textrm{II}}^{(\kappa,\gamma)}(t)}\\ -  \sum_{\gamma=1}^{N_f}\underbrace{(\vc{u}^{(h,\kappa)}(t))^\T(\mat{V}^{(\gamma)})^\T\mat{W}^{(\gamma)}\vc{f}^{(*,\kappa,\gamma)}(t)}_{=:\, \fn{E}_{\textrm{III}}^{(\kappa,\gamma)}(t)},
\end{multlined}
\end{equation}
where the contributions $\fn{E}_{\textrm{I}}^{(\kappa)}(t)$, $\fn{E}_{\textrm{II}}^{(\kappa,\gamma)}(t)$, and $\fn{E}_{\textrm{III}}^{(\kappa,\gamma)}(t)$ to the local energy balance are defined as above for all $\kappa \in \{1,\ldots, K\}$ and $\gamma \in \{1,\ldots, N_f\}$. 
Since $\mat{K}\mat{D}^{(m)} = \zeromat$ for all $m \in \{1,\ldots d\}$ by Assumption \ref{asm:k_mat}, the SBP property in (\ref{eq:sbp}) implies that the relation
\begin{equation}\label{eq:sym_skewsym}
\begin{aligned}
(\mat{M}+\mat{K})\mat{D}^{(m)} &= \mat{M}\mat{D}^{(m)}\\
&= \frac{\mat{M}\mat{D}^{(m)} - (\mat{D}^{(m)})^\T\mat{M}}{2} + \frac{1}{2}\sum_{\gamma=1}^{N_f}\hat{n}_m^{(\gamma)}(\mat{V}^{(\gamma)})^\T\mat{W}^{(\gamma)}  \mat{V}^{(\gamma)}
\end{aligned}
\end{equation}
is satisfied for all $m \in \{1,\ldots, d\}$. Noting the skew-symmetry of the first term on the second line of (\ref{eq:sym_skewsym}), we then obtain
\begin{equation}\label{eq:telescoping}
\begin{multlined}
(\vc{u}^{(h,\kappa)}(t))^\T(\mat{M}+\mat{K})\mat{D}^{(m)}\vc{u}^{(h,\kappa)}(t) \\ = \frac{1}{2}\sum_{\gamma=1}^{N_f} \hat{n}_m^{(\gamma)}(\vc{u}^{(h,\kappa)}(t))^\T(\mat{V}^{(\gamma)})^\T\mat{W}^{(\gamma)}. \mat{V}^{(\gamma)}\vc{u}^{(h,\kappa)}(t).
\end{multlined}
\end{equation}
Since the mapping is assumed to be affine and the flux is assumed to be linear with constant coefficients, the first two terms of the energy balance in (\ref{eq:local_terms}) may be expressed as 
\begin{equation}\label{eq:e1kappa}
E_{\textrm{I}}^{(\kappa)}(t) = \sum_{m=1}^d\sum_{n=1}^d a_n\mathcal{J}^{(\kappa)}[(\mat{\mathcal{J}}^{(\kappa)})^{-1}]_{mn}(\vc{u}^{(h,\kappa)}(t))^\T(\mat{M}+\mat{K})\mat{D}^{(m)}\vc{u}^{(h,\kappa)}(t) 
\end{equation}
and
\begin{equation}\label{eq:e2kappa}
\begin{aligned}
E_{\textrm{II}}^{(\kappa,\gamma)}(t)  &=\\  \sum_{m=1}^d\sum_{n=1}^d & a_n\mathcal{J}^{(\kappa)}[(\mat{\mathcal{J}}^{(\kappa)})^{-1}]_{mn} \hat{n}_m^{(\gamma)}(\vc{u}^{(h,\kappa)}(t))^\T(\mat{V}^{(\gamma)})^\T\mat{W}^{(\gamma)}  \mat{V}^{(\gamma)}\vc{u}^{(h,\kappa)}(t) \\
&=  \vc{a} \cdot \vc{n}^{(\kappa,\gamma)} (\vc{u}^{(h,\kappa)}(t))^\T(\mat{V}^{(\gamma)})^\T\mat{W}^{(\gamma)}\mat{J}^{(\kappa,\gamma)}  \mat{V}^{(\gamma)}\vc{u}^{(h,\kappa)}(t)
\end{aligned}
\end{equation}
respectively, where the second equality in (\ref{eq:e2kappa}) results from  (\ref{eq:nanson}). We may therefore apply (\ref{eq:telescoping}) to (\ref{eq:e1kappa}) in order to obtain $E_{\textrm{I}}^{(\kappa)}(t) = \sum_{\gamma=1}^{N_f}E_{\textrm{II}}^{(\kappa,\gamma)}(t)/2$. Eliminating the factor of $1/2$ from both sides, the local energy balance in (\ref{eq:local_terms}) is then given by
\begin{equation}\label{eq:local_balance}
\dv{t}\qty((\vc{u}^{(h,\kappa)}(t))^\T\mat{M}^{(\kappa)}(\mat{F}^{(\kappa)})^{-1}\vc{u}^{(h,\kappa)}(t)) = \sum_{\gamma=1}^{N_f}\Big( E_{\textrm{II}}^{(\kappa,\gamma)}(t) - 2 E_{\textrm{III}}^{(\kappa,\gamma)}(t)\Big),
\end{equation}
where we note that the right-hand side contains only facet contributions. As with the proof of global conservation, summing (\ref{eq:local_balance}) over all elements and splitting the net contribution $E_{\textrm{net}}^{(\kappa,\gamma)}(t) \df E_{\textrm{II}}^{(\kappa,\gamma)}(t) + E_{\textrm{II}}^{(\nu,\eta)}(t) - 2(E_{\textrm{III}}^{(\kappa,\gamma)}(t) + E_{\textrm{III}}^{(\nu,\eta)}(t) )$ from each shared interface equally between the index pairs $(\kappa,\gamma), (\nu,\eta) \in \set{K}_\ug{\Omega}$ corresponding to the coincident or periodically connected facets $\ug{\Gamma}^{(\kappa,\gamma)}$ and $\ug{\Gamma}^{(\nu,\eta)}$, the global energy balance is given by
\begin{equation}\label{eq:global_balance}
\begin{multlined}
\dv{t}\qty(\sum_{\kappa=1}^K(\vc{u}^{(h,\kappa)}(t))^\T\mat{M}^{(\kappa)}(\mat{F}^{(\kappa)})^{-1}\vc{u}^{(h,\kappa)}(t)) \\ = \quad \sum_{\mathclap{(\kappa,\gamma) \in \set{K}_\ug{\Omega}}} \ \ \frac{1}{2}E_{\textrm{net}}^{(\kappa,\gamma)}(t)\ + \ \sum_{\mathclap{(\kappa,\gamma) \in  \set{K}_{\partial\ug{\Omega}}}}\Big(E_{\textrm{II}}^{(\kappa,\gamma)}(t) - 2E_{\textrm{III}}^{(\kappa,\gamma)}(t)\Big).
\end{multlined}
\end{equation}
As we will now demonstrate, the contributions to the right-hand side of (\ref{eq:global_balance}) arising from interior or periodic interfaces and those arising from facets on the inflow and outflow portions of the boundary may be analyzed separately in order to establish a bound for the quantity on the left-hand side.
\subsubsection*{Interior or periodic interfaces}

Similarly to the proof of Theorem \ref{thm:conservation}, it follows from Assumption \ref{asm:conformal} that for any interior or periodic interface corresponding to $(\kappa,\gamma) \in \set{K}_{\ug{\Omega}}$, the contributions arising from a numerical flux function taking the form of (\ref{eq:num_flux_adv}) may be expressed as
\begin{equation}\label{eq:e3kappa}
\begin{multlined}
E_{\textrm{III}}^{(\kappa,\gamma)}(t) = \frac{1}{2}E_{\textrm{II}}^{(\kappa,\gamma)}(t)\\ + \frac{1}{2} \vc{a} \cdot\vc{n}^{(\kappa,\gamma)}(\vc{u}^{(h,\kappa)}(t))^\T(\mat{V}^{(\gamma)})^\T\mat{W}^{(\gamma)}\mat{J}^{(\kappa,\gamma)}\mat{T}^{(\kappa,\gamma)}\mat{V}^{(\eta)}\vc{u}^{(h,\nu)}(t) \\
- \frac{\lambda}{2} \big\lvert\vc{a} \cdot\vc{n}^{(\kappa,\gamma)}\big\rvert(\vc{u}^{(h,\kappa)}(t))^\T(\mat{V}^{(\gamma)})^\T\mat{W}^{(\gamma)}\mat{J}^{(\kappa,\gamma)}\mat{T}^{(\kappa,\gamma)}\mat{V}^{(\eta)}\vc{u}^{(h,\nu)}(t)\\ + 
\frac{\lambda}{2} \big\lvert\vc{a} \cdot\vc{n}^{(\kappa,\gamma)}\big\lvert(\vc{u}^{(h,\kappa)}(t))^\T(\mat{V}^{(\gamma)})^\T\mat{W}^{(\gamma)}\mat{J}^{(\kappa,\gamma)}  \mat{V}^{(\gamma)}\vc{u}^{(h,\kappa)}(t)
\end{multlined}
\end{equation}
and
\begin{equation}\label{eq:e3nu}
\begin{multlined}
E_{\textrm{III}}^{(\nu,\eta)}(t) = \frac{1}{2}E_{\textrm{II}}^{(\nu,\eta)}(t) \\ -\frac{1}{2}\vc{a} \cdot\vc{n}^{(\kappa,\gamma)}(\vc{u}^{(h,\nu)}(t))^\T(\mat{V}^{(\eta)})^\T(\mat{T}^{(\kappa,\gamma)})^\T\mat{W}^{(\gamma)}\mat{J}^{(\kappa,\gamma)}  \mat{V}^{(\gamma)}\vc{u}^{(h,\kappa)}(t)  \\
- \frac{\lambda}{2} \big\lvert\vc{a} \cdot\vc{n}^{(\kappa,\gamma)}\big\rvert(\vc{u}^{(h,\nu)}(t))^\T(\mat{V}^{(\eta)})^\T(\mat{T}^{(\kappa,\gamma)})^\T\mat{W}^{(\gamma)}\mat{J}^{(\kappa,\gamma)}  \mat{V}^{(\gamma)}\vc{u}^{(h,\kappa)}(t) \\ + \frac{\lambda}{2}\big\lvert\vc{a} \cdot\vc{n}^{(\kappa,\gamma)}\big\lvert(\vc{u}^{(h,\nu)}(t))^\T(\mat{V}^{(\eta)})^\T(\mat{T}^{(\kappa,\gamma)})^\T\mat{W}^{(\gamma)}\mat{J}^{(\kappa,\gamma)}\mat{T}^{(\kappa,\gamma)}  \mat{V}^{(\eta)}\vc{u}^{(h,\nu)}(t).
\end{multlined}
\end{equation}
Noting that the first terms of (\ref{eq:e3kappa}) and (\ref{eq:e3nu}) cancel out the contributions of $E_{\textrm{II}}^{(\kappa,\gamma)}(t)$ and $E_{\textrm{II}}^{(\nu,\eta)}(t)$ to $E_{\textrm{net}}^{(\kappa,\gamma)}(t)$, respectively, and that the matrix $\mat{W}^{(\gamma)}\mat{J}^{(\kappa,\gamma)}$ is SPD under Assumption \ref{asm:affine}, the net contribution to the energy balance in (\ref{eq:global_balance}) corresponding to any $(\kappa,\gamma) \in \set{K}_{\ug{\Omega}}$ is then given by
\begin{equation}
\begin{aligned}
E_{\textrm{net}}^{(\kappa,\gamma)}(t) &=  -\lambda \big\lvert\vc{a} \cdot\vc{n}^{(\kappa,\gamma)}\big\rvert\Big((\vc{u}^{(h,\kappa)}(t))^\T(\mat{V}^{(\gamma)})^\T\mat{W}^{(\gamma)}\mat{J}^{(\kappa,\gamma)}  \mat{V}^{(\gamma)}\vc{u}^{(h,\kappa)}(t) \\ &\qquad - 2(\vc{u}^{(h,\kappa)}(t))^\T(\mat{V}^{(\gamma)})^\T\mat{W}^{(\gamma)}\mat{J}^{(\kappa,\gamma)}\mat{T}^{(\kappa,\gamma)}\mat{V}^{(\eta)}\vc{u}^{(h,\nu)}(t) \\ &\qquad+ (\vc{u}^{(h,\nu)}(t))^\T(\mat{V}^{(\eta)})^\T(\mat{T}^{(\kappa,\gamma)})^\T\mat{W}^{(\gamma)}\mat{J}^{(\kappa,\gamma)}\mat{T}^{(\kappa,\gamma)}  \mat{V}^{(\eta)}\vc{u}^{(h,\nu)}(t)\Big)\\
&=-\lambda \big\lvert\vc{a} \cdot\vc{n}^{(\kappa,\gamma)}\big\rvert(\vc{d}^{(\kappa,\gamma)}(t))^\T\mat{W}^{(\gamma)}\mat{J}^{(\kappa,\gamma)}\vc{d}^{(\kappa,\gamma)}(t)
\end{aligned}
\end{equation}
where we define $\vc{d}^{(\kappa,\gamma)}(t) \df \mat{V}^{(\gamma)}\vc{u}^{(h,\kappa)}(t) - \mat{T}^{(\kappa,\gamma)}\mat{V}^{(\eta)}\vc{u}^{(h,\nu)}(t)$. It therefore follows that $E_{\textrm{net}}^{(\kappa,\gamma)}(t) \leq 0$ if $\lambda \geq 0$, where for $\lambda = 0$, the interior or periodic interfaces do not contribute to the energy balance, whereas for $\lambda > 0$, the upwind bias of the numerical flux in (\ref{eq:num_flux_adv}) dissipates energy at a rate depending quadratically on the size of the inter-element jump. We have therefore demonstrated that the interior/periodic interface coupling procedure is energy stable, and now turn our attention to the contributions to (\ref{eq:global_balance}) resulting from the inflow and outflow portions of the boundary.
\subsubsection*{Inflow boundary}
For a given facet  $\ug{\Gamma}^{(\kappa,\gamma)} \subset \partial\ug{\Omega}^{-}$ corresponding to $(\kappa,\gamma) \in \set{K}_{\partial\ug{\Omega}^-}$, the contribution $ E_{\textrm{II}}^{(\kappa,\gamma)}(t) - 2 E_{\textrm{III}}^{(\kappa,\gamma)}(t)$ to the energy balance in (\ref{eq:global_balance}) results from using (\ref{eq:e2kappa}) and (\ref{eq:e3kappa}) with $\mat{T}^{(\kappa,\gamma)}\mat{V}^{(\eta)}\vc{u}^{(h,\nu)}(t)$ replaced by $\vc{b}^{(\kappa,\gamma)}$, taking $\lambda = 1$ as per Assumption \ref{asm:advection}, and noting that $\vc{a} \cdot \vc{n}^{(\kappa,\gamma)} = - \lvert\vc{a} \cdot \vc{n}^{(\kappa,\gamma)}\rvert$. We therefore obtain
\begin{equation}
\begin{aligned}
 E_{\textrm{II}}^{(\kappa,\gamma)}(t) - 2 E_{\textrm{III}}^{(\kappa,\gamma)}(t) &= \big\lvert\vc{a} \cdot\vc{n}^{(\kappa,\gamma)}\big\rvert\Big(2(\vc{u}^{(h,\kappa)}(t))^\T(\mat{V}^{(\gamma)})^\T\mat{W}^{(\gamma)}\mat{J}^{(\kappa,\gamma)}\vc{b}^{(\kappa,\gamma)}(t) \\ & \qquad\qquad- (\vc{u}^{(h,\kappa)}(t))^\T(\mat{V}^{(\gamma)})^\T\mat{W}^{(\gamma)}  \mat{V}^{(\gamma)}\vc{u}^{(h,\kappa)}(t) \Big) \\
 & =  \big\lvert\vc{a} \cdot\vc{n}^{(\kappa,\gamma)}\big\rvert\Big((\vc{b}^{(\kappa,\gamma)}(t))^\T\mat{W}^{(\gamma)}\mat{J}^{(\kappa,\gamma)}\vc{b}^{(\kappa,\gamma)}(t) \\ & \qquad\qquad-(\vc{d}^{(\kappa,\gamma)}(t))^\T\mat{W}^{(\gamma)}\mat{J}^{(\kappa,\gamma)}\vc{d}^{(\kappa,\gamma)}(t)\Big) \\
 &\leq\big\lvert\vc{a} \cdot\vc{n}^{(\kappa,\gamma)}\big\rvert(\vc{b}^{(\kappa,\gamma)}(t))^\T\mat{W}^{(\gamma)}\mat{J}^{(\kappa,\gamma)}\vc{b}^{(\kappa,\gamma)}(t),
 \end{aligned}
\end{equation}
where the second equality follows from completing the square and defining $\vc{d}^{(\kappa,\gamma)}(t) \df \mat{V}^{(\gamma)}\vc{u}^{(h,\kappa)}(t) - \vc{b}^{(\kappa,\gamma)}(t)$ and the final inequality follows from the positive-definiteness of $\mat{W}^{(\gamma)}\mat{J}^{(\kappa,\gamma)}$.
\subsubsection*{Outflow boundary}
Using the fact that $\ug{\Gamma}^{(\kappa,\gamma)} \subset \partial\ug{\Omega}^+$ implies that $\vc{a} \cdot \vc{n}^{(\kappa,\gamma)} = \lvert \vc{a} \cdot \vc{n}^{(\kappa,\gamma)}\rvert$, it follows from taking $\lambda = 1$ in (\ref{eq:e3kappa}) that $E_{\textrm{III}}^{(\kappa,\gamma)}(t) = E_{\textrm{II}}^{(\kappa,\gamma)}(t)$. The resulting contribution to the second sum on the right-hand side of (\ref{eq:global_balance}) for any $(\kappa,\gamma) \in \set{K}_{\partial\ug{\Omega}^+}$ is therefore simply $-E_{\textrm{II}}^{(\kappa,\gamma)}(t)$, which is non-positive as a result of (\ref{eq:e2kappa}) and the positive-definiteness of $\mat{W}^{(\gamma)}\mat{J}^{(\kappa,\gamma)}$, and thus does not contribute to growth in the solution energy. The final energy balance is therefore given by
\begin{equation}\label{eq:global_balance_final}
\begin{multlined}
\dv{t}\qty(\sum_{\kappa=1}^K(\vc{u}^{(h,\kappa)}(t))^\T\mat{M}^{(\kappa)}(\mat{F}^{(\kappa)})^{-1}\vc{u}^{(h,\kappa)}(t)) \\ = \ \ \sum_{\mathclap{(\kappa,\gamma) \in  \set{K}_{\partial\ug{\Omega}^-}}} \big\lvert\vc{a} \cdot\vc{n}^{(\kappa,\gamma)}\big\rvert(\vc{b}^{(\kappa,\gamma)}(t))^\T\mat{W}^{(\gamma)}\mat{J}^{(\kappa,\gamma)}\vc{b}^{(\kappa,\gamma)}(t) \\- \ \ \sum_{\mathclap{(\kappa,\gamma) \in  \set{K}_{\partial\ug{\Omega}^+}}} \ \big\lvert\vc{a} \cdot \vc{n}^{(\kappa,\gamma)}\big\rvert (\vc{u}^{(h,\kappa)}(t))^\T(\mat{V}^{(\gamma)})^\T\mat{W}^{(\gamma)}\mat{J}^{(\kappa,\gamma)}  \mat{V}^{(\gamma)}\vc{u}^{(h,\kappa)}(t)
\\- \ \ \sum_{\mathclap{(\kappa,\gamma) \in  \set{K}_{\partial\ug{\Omega}^-}}}\big\lvert\vc{a} \cdot\vc{n}^{(\kappa,\gamma)}\big\lvert(\vc{d}^{(\kappa,\gamma)}(t))^\T\mat{W}^{(\gamma)}\mat{J}^{(\kappa,\gamma)}\vc{d}^{(\kappa,\gamma)}(t)\\
-\ \ \sum_{\mathclap{(\kappa,\gamma) \in \set{K}_\ug{\Omega}}} \ \ \frac{\lambda}{2}\big\lvert\vc{a} \cdot\vc{n}^{(\kappa,\gamma)}\big\rvert(\vc{d}^{(\kappa,\gamma)}(t))^\T\mat{W}^{(\gamma)}\mat{J}^{(\kappa,\gamma)}\vc{d}^{(\kappa,\gamma)}(t),
\end{multlined}
\end{equation}
where the second term on the right-hand side corresponds to the rate of energy exiting the domain, while the final two terms correspond to dissipation resulting from the facet jumps. We therefore obtain the inequality in (\ref{eq:energy}), bounding the time rate of change in the norm $\lVert\fn{u}^h(\cdot,t)\rVert_{\set{V}_h}$.
\end{proof}

\begin{remark}\label{rmk:stability}
For homogeneous or periodic boundary conditions (i.e.\ $\fn{b}\xt = 0$ for all $\vc{x} \in \partial\ug{\Omega}^-$ or $\set{K}_{\partial\ug{\Omega}} = \emptyset$, respectively), the energy balance in (\ref{eq:energy}) implies that $\dd \lVert\fn{u}^h(\cdot,t)\rVert_{\set{V}_h}^2/ \dd t \leq 0$, resulting in energy conservation for periodic problems when a central flux (i.e.\ $\lambda = 0$) is chosen for all interfaces.
\end{remark}

\section{Numerical experiments}\label{sec:numerical}
\subsection{Design choices and implementation}\label{sec:design_choices}
The methods described in \textsection \ref{sec:formulation} were implemented within a freely available\footnote{\url{https://github.com/tristanmontoya/GHOST/tree/algebraic_framework_v2}} Python code offering substantial flexibility in the design choices afforded for DG and FR methods, which are detailed in this section. 

\subsubsection*{Approximation space}
The two-dimensional numerical experiments in this work involve formulations on the reference triangle given by $\hat{\ug{\Omega}} \df \{\hat{\vc{x}} \in \R^2 : \hat{x}_1 \geq -1, \ \hat{x}_2 \geq -1, \ \hat{x}_1 + \hat{x}_2 \leq 1 \}$, with $\set{V}$ taken to be the total-degree polynomial space $\set{P}_p(\hat{\ug{\Omega}})$ given in \textsection \ref{sec:approx_space}, with $p \in \{2,3,4\}$. 

\subsubsection*{Discrete inner products}\label{sec:num_disc_in_prod}
Three options for discrete inner products are employed for the numerical experiments in this section, which are referred to as \emph{Quadrature I}, \emph{Quadrature II}, and \emph{Collocation} schemes. The first two make use of the quadrature-based approach described in \textsection \ref{app:quadrature}, while the third employs the collocation-based approach in \textsection \ref{app:collocation}. For the quadrature-based schemes, all volume terms are computed using positive quadrature rules of degree $2p$ due to Xiao and Gimbutas \cite{xiao_gimbutas_quadrature_10}. To compute the facet terms, we use LG quadrature rules with $p+1$ nodes on each facet (i.e.\ degree $2p+1$) for the Quadrature I schemes and LGL quadrature rules with $p+1$ nodes on each facet (i.e.\ degree $2p-1$) for the Quadrature II schemes. Although both approaches result in Assumption \ref{asm:inner_prod} being satisfied (where, as discussed in Remark \ref{rmk:spd}, the rank of $\mat{V}$ was verified numerically), the insufficient accuracy of the facet integration for the Quadrature II schemes precludes the use of Lemma \ref{lem:total_sbp}. As a result, Assumption \ref{asm:sbp} is satisfied for the Quadrature I schemes, but not for the Quadrature II schemes, leading to the loss of the SBP property in the latter case. The Collocation schemes interpolate the volume terms on the ``warp \& blend'' nodes introduced by Warburton \cite{warburton_interpolation_nodes_simplex_06}, which form unisolvent sets for the corresponding polynomial spaces $\set{P}_p(\hat{\ug{\Omega}})$, and include $p+1$ LGL nodes on each facet, which are used to interpolate the facet terms. Therefore, as discussed in \textsection \ref{app:collocation}, the Collocation schemes satisfy Assumptions \ref{asm:sbp} and \ref{asm:inner_prod} by construction. The nodal sets corresponding to each choice of discrete inner product for polynomial degrees 2, 3, and 4 are shown in Figures \ref{fig:elem_p2}, \ref{fig:elem_p3}, \ref{fig:elem_p4}, respectively, where the symbols $\onontimes$ and {\tiny$\blacksquare$} represent nodes in $\set{S}$ and $\set{S}^{(\gamma)}$, respectively. 

\begin{figure}[t!]
	\centering
	\begin{subfigure}{0.33\linewidth}
		\centering
		\includegraphics[width=0.45\textwidth]{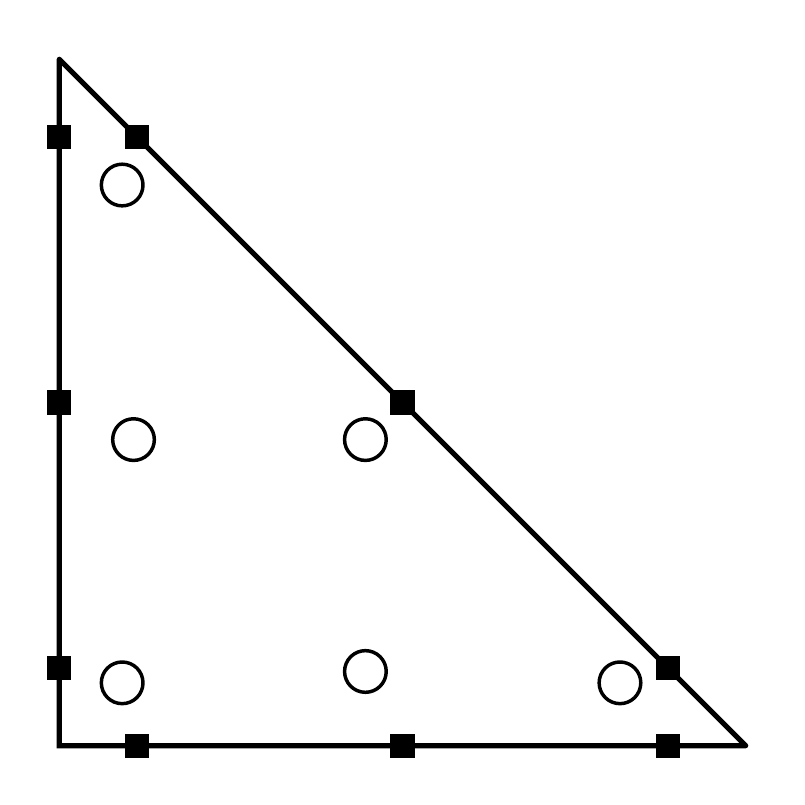}
		\caption{Quadrature I}
		\end{subfigure}~
	\begin{subfigure}{0.33\linewidth}
		\centering
		\includegraphics[width=0.45\textwidth]{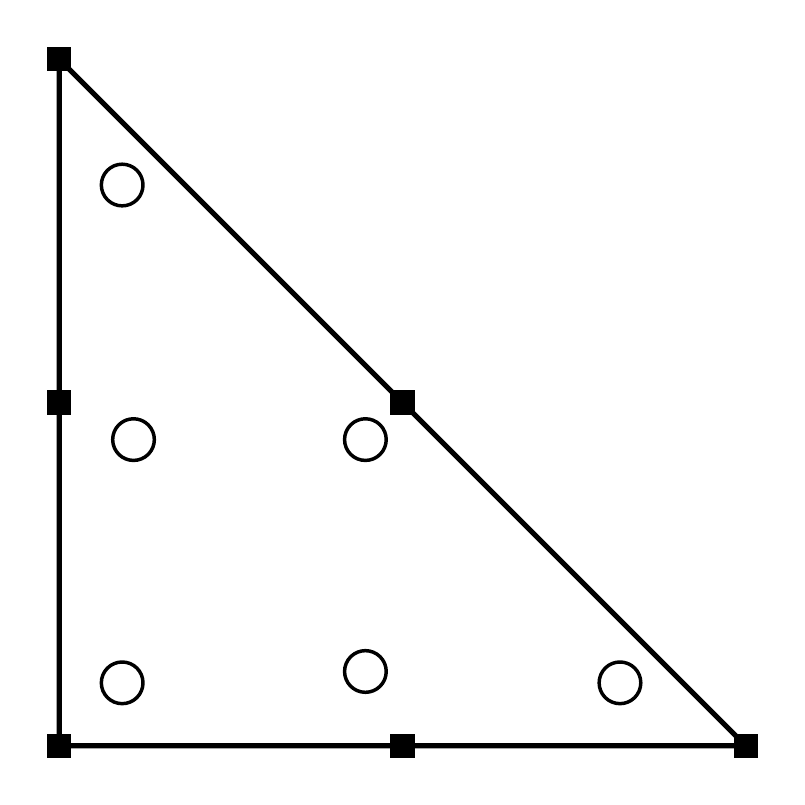}
		\caption{Quadrature II}
	\end{subfigure}~
	\begin{subfigure}{0.33\linewidth}
	\centering
	\includegraphics[width=0.45\textwidth]{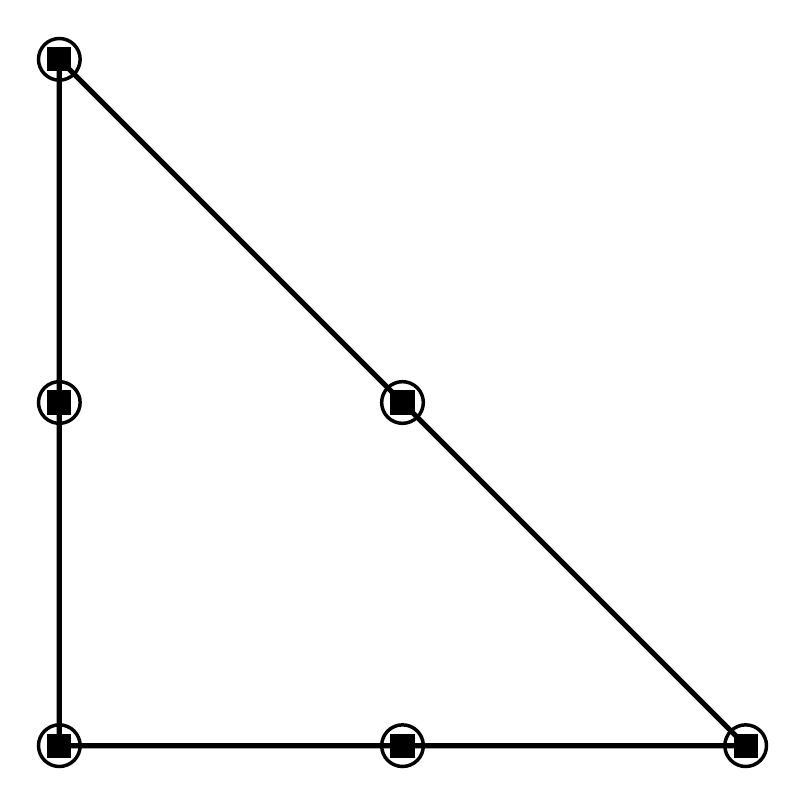}
	\caption{Collocation}
\end{subfigure}
\caption{Nodal sets for $p=2$}\label{fig:elem_p2}\vspace{0.5em}
	\centering
	\begin{subfigure}{0.33\linewidth}
		\centering
		\includegraphics[width=0.45\textwidth]{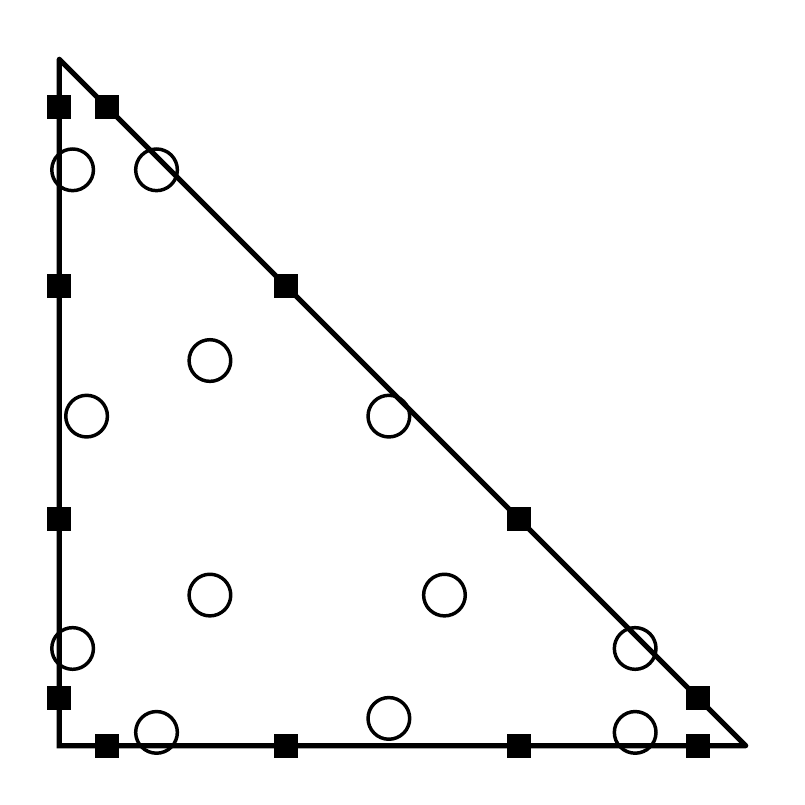}
		\caption{Quadrature I}
	\end{subfigure}~
	\begin{subfigure}{0.33\linewidth}
		\centering
		\includegraphics[width=0.45\textwidth]{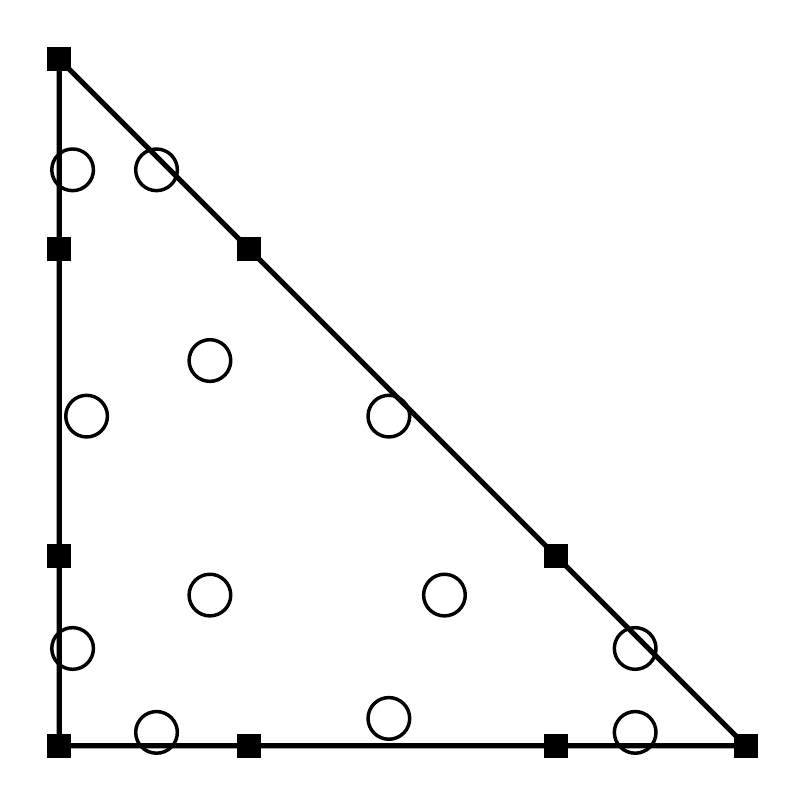}
		\caption{Quadrature II}
	\end{subfigure}~
	\begin{subfigure}{0.33\linewidth}
		\centering
		\includegraphics[width=0.45\textwidth]{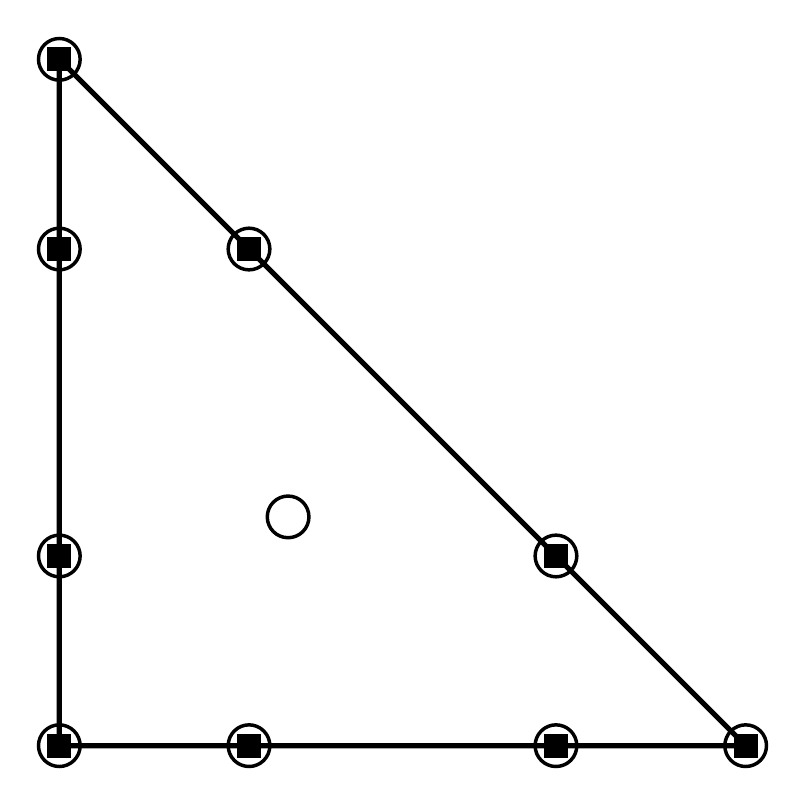}
		\caption{Collocation}
	\end{subfigure}
	\caption{Nodal sets for $p=3$}\label{fig:elem_p3}\vspace{0.5em}
	\centering
	\begin{subfigure}{0.33\linewidth}
		\centering
		\includegraphics[width=0.45\textwidth]{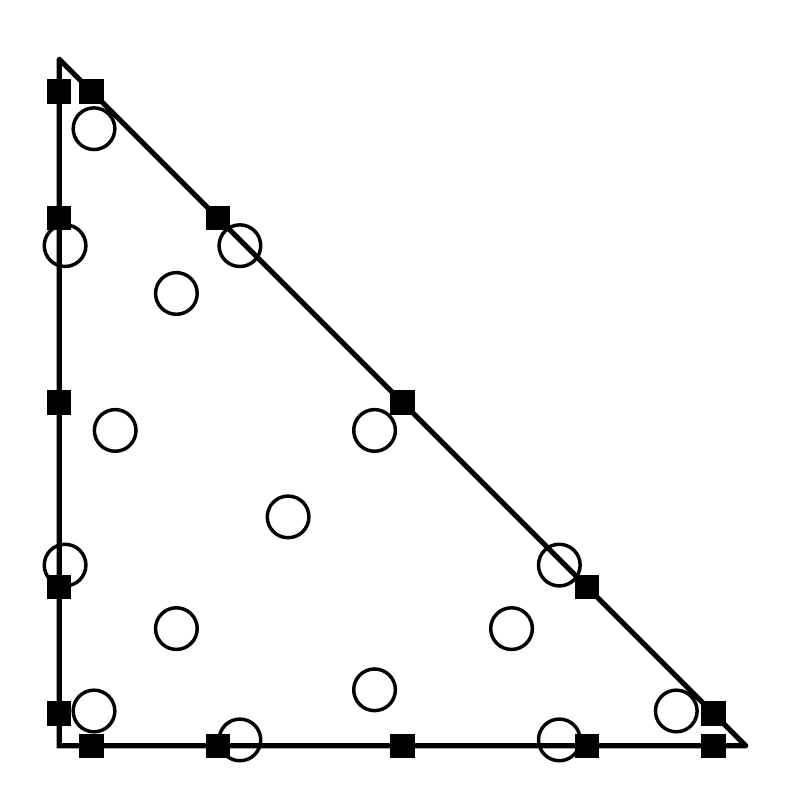}
		\caption{Quadrature I}
	\end{subfigure}~
	\begin{subfigure}{0.33\linewidth}
		\centering
		\includegraphics[width=0.45\textwidth]{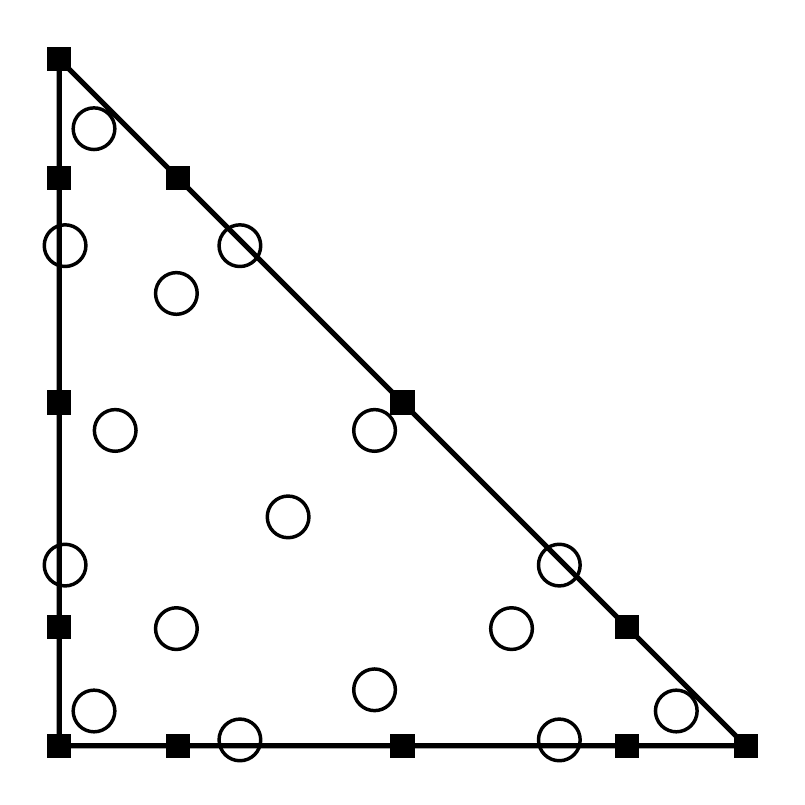}
		\caption{Quadrature II}
	\end{subfigure}~
	\begin{subfigure}{0.33\linewidth}
		\centering
		\includegraphics[width=0.45\textwidth]{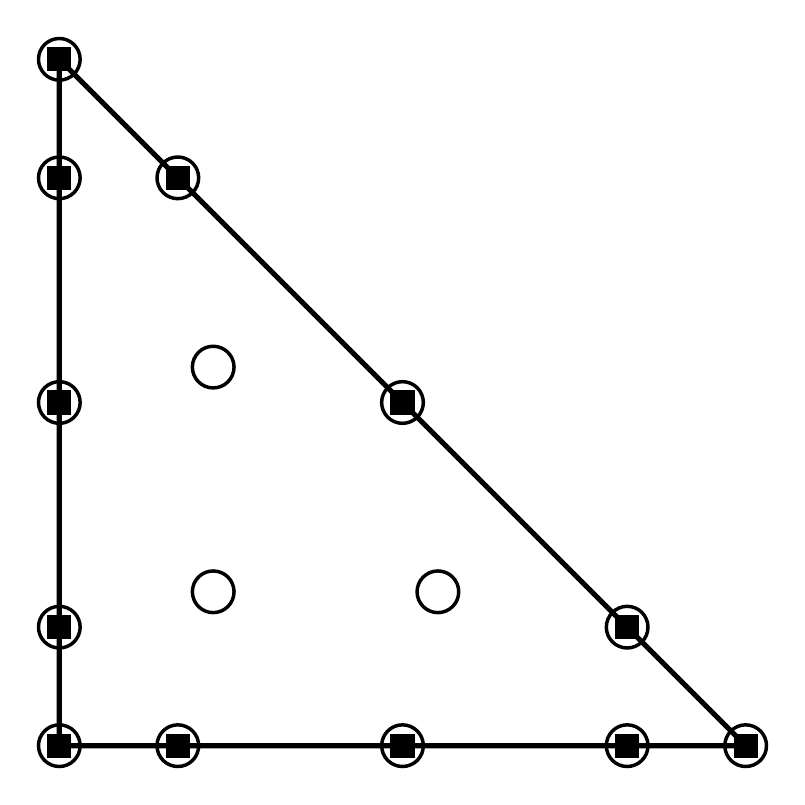}
		\caption{Collocation}
	\end{subfigure}
	\caption{Nodal sets for $p=4$}\label{fig:elem_p4}
\end{figure}

\subsubsection*{Basis}
For the quadrature-based schemes, orthonormal bases with respect to the $L^2$ inner product are used, constructed through normalization of the modal basis functions defined in \textsection \ref{app:modal_basis}. For the collocation-based schemes, nodal bases defined as in \textsection \ref{app:nodal_basis} are used with $\tilde{\set{S}} = \set{S}$, in this case corresponding to the ``warp \& blend'' nodes in \cite{warburton_interpolation_nodes_simplex_06}.

\subsubsection*{Spatial residual formulation}
Either the strong formulation in (\ref{eq:fr_algebraic}) or the weak formulation in (\ref{eq:fr_weak}) may be used, in the case of $\mat{K}=\zeromat$ recovering the standard unfiltered strong-form and weak-form DG schemes in (\ref{eq:dg_strong}) and (\ref{eq:dg_algebraic}), respectively.

\subsubsection*{Lifting and filter matrices}
The lifting operators employed for the FR schemes take the form $\mat{L}^{(\gamma)} \df (\mat{M}+\mat{K})^{-1}(\mat{V}^{(\gamma)})^\T\mat{W}^{(\gamma)}$ as in Theorem \ref{thm:vcjh}, with $\mat{K} = \mat{K}^{\textrm{2D}}$ as defined as in \textsection \ref{app:k_mat}, and the corresponding filter operator given by $\mat{F} \df (\mat{I} + \mat{M}^{-1}\mat{K})^{-1}$ as in Lemma \ref{lem:filtered_dg}. The coefficient $c \in \R$ scaling the matrix $\mat{K}$ is chosen either as $c = c_{\textrm{DG}} = 0$, recovering a DG scheme in strong or weak form, or as $c = c_+$, corresponding to the VCJH scheme found by Castonguay \etal \cite{castonguay_vincent_jameson_triangular_fr_11} to allow for the largest stable time step for a given polynomial degree when solving the linear advection equation. The latter choice is merely used as a reference value in order to demonstrate that the results in \textsection \ref{sec:analysis} hold for $c > 0$, and we emphasize that the optimality of such a choice in terms of time step size is highly dependent on the time-marching method, the mesh, and the particular problem being solved. The numerical experiments in Castonguay \etal \cite{castonguay_vincent_jameson_triangular_fr_11} employ values of $c_+$ given by $4.3 \times 10^{-2}$, $6.0 \times 10^{-4}$, and $5.6 \times 10^{-6}$ for polynomial degrees 2, 3, and 4, respectively, and we make use of the same values in order to recover discretizations analogous to those described in their work.
\subsubsection*{Numerical flux}
The numerical flux for the linear advection equation is given as in (\ref{eq:num_flux_adv}), where we report results for both $\lambda=1$ and $\lambda = 0$, corresponding to upwind and central fluxes, respectively. In the case of the Euler equations, we employ Roe's approximate Riemann solver \cite{roe81}.

\subsubsection*{Mesh and coordinate transformation}
As the test problems considered in this work are posed on the square domain $\ug{\Omega} \df (0,L)^2$, we begin with a regular Cartesian grid with $M \in \mathbb{N}$ equal intervals in each direction, dividing each square into two triangles of equal size to obtain a total of $K = 2M^2$ elements. Using an affine transformation to map the interpolation nodes for a Lagrange basis $\set{B}_{\textrm{map}} \df \{\phi_{\textrm{map}}^{(i)}\}_{i=1}^{N_{\textrm{map}}}$ of degree $p_{\textrm{map}} \in \mathbb{N}$ (chosen here as the ``warp \& blend'' nodes from \cite{warburton_interpolation_nodes_simplex_06}) on the reference element onto each physical element $\ug{\Omega}^{(\kappa)} \in \set{T}_h$ to obtain $\set{S}_{\textrm{map}}^{(\kappa)} \df \{\vc{x}^{(\kappa,i)}\}_{i=1}^{N_{\textrm{map}}} \subset \ug{\Omega}^{(\kappa)}$, we then interpolate the warping function considered by Del Rey Fern\'andez \etal \cite{delrey_tensorproduct_sbp_17} given by
\begin{equation}\label{eq:warp}
\fv{\Theta}(\vc{x}) \df \mqty[x_1 + \frac{1}{5} L \sin\qty(\pi x_1/L)\sin\qty(\pi x_2/L) \\ x_2 + \frac{1}{5} L\exp(1-x_2/L)\sin\qty(\pi x_1/L)\sin\qty(\pi x_2/L)]
\end{equation} 
on such a Lagrange basis, resulting in $\fv{X}^{(\kappa)}(\hat{\vc{x}}) \df \sum_{i=1}^{N_{\textrm{map}}} \fv{\Theta}(\vc{x}^{(\kappa,i)}) \phi_{\textrm{map}}^{(i)}(\hat{\vc{x}})$. Examples of such meshes obtained through warping of a split Cartesian mesh with $M = 5$ are shown in Figure \ref{fig:meshes}, where the symbol $\bullet$ is used to represent the image of each interpolation node for the basis $\set{B}_{\textrm{map}}$ under the mapping.
 \begin{figure}[t!]
 	\centering
 	\begin{subfigure}{0.33\linewidth}
 		\centering
 		\includegraphics[width=0.9\textwidth]{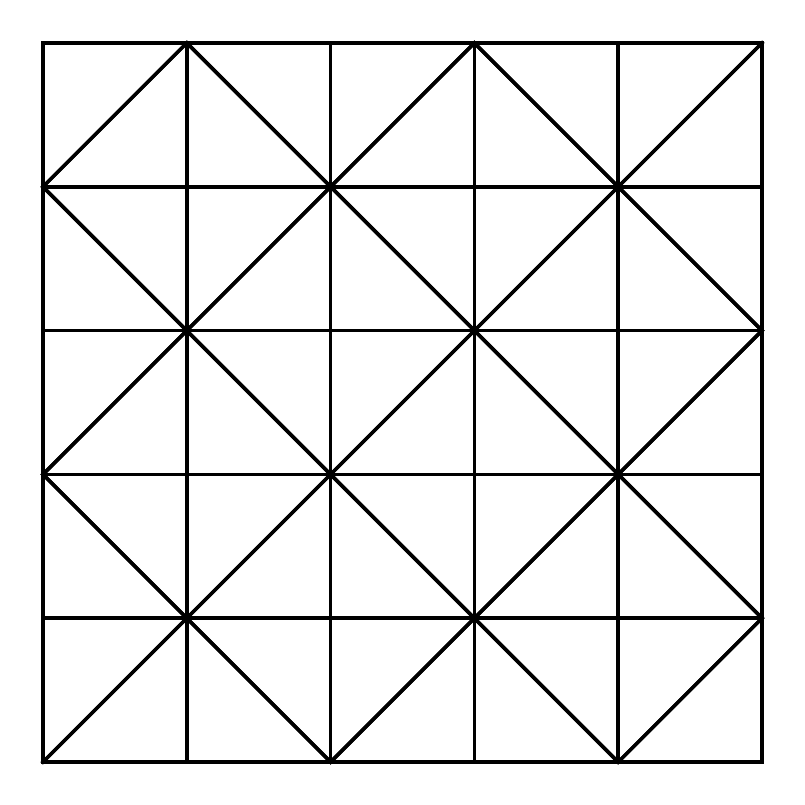}
 		\caption{Original mesh}
 	\end{subfigure}~
 	\begin{subfigure}{0.33\linewidth}
 		\centering
 		\includegraphics[width=0.9\textwidth]{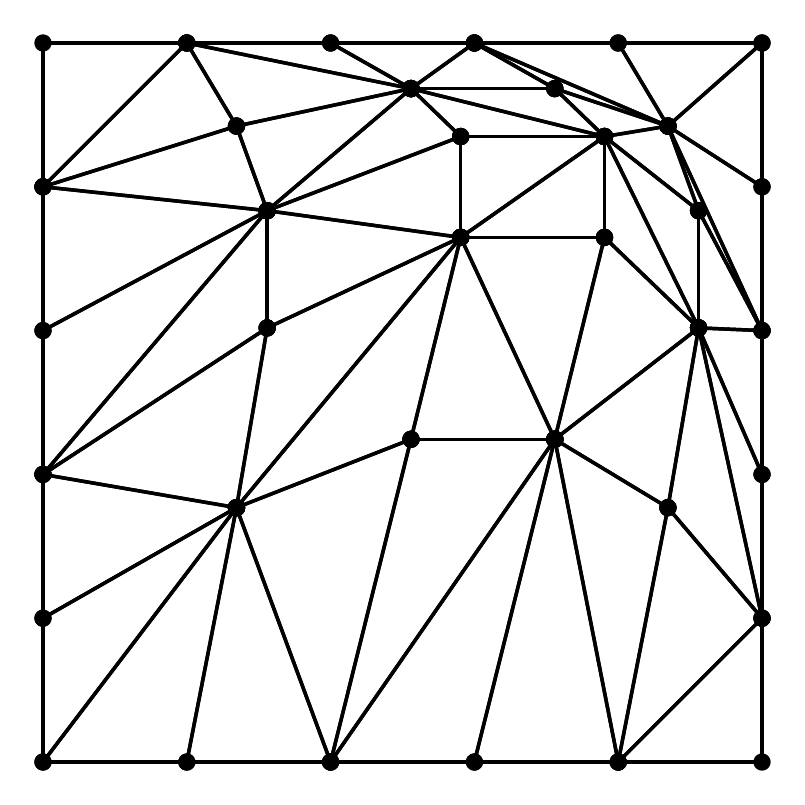}
 		\caption{$p_{\textrm{map}} = 1$}
 	\end{subfigure}~
 	\begin{subfigure}{0.33\linewidth}
 		\centering
 		\includegraphics[width=0.9\textwidth]{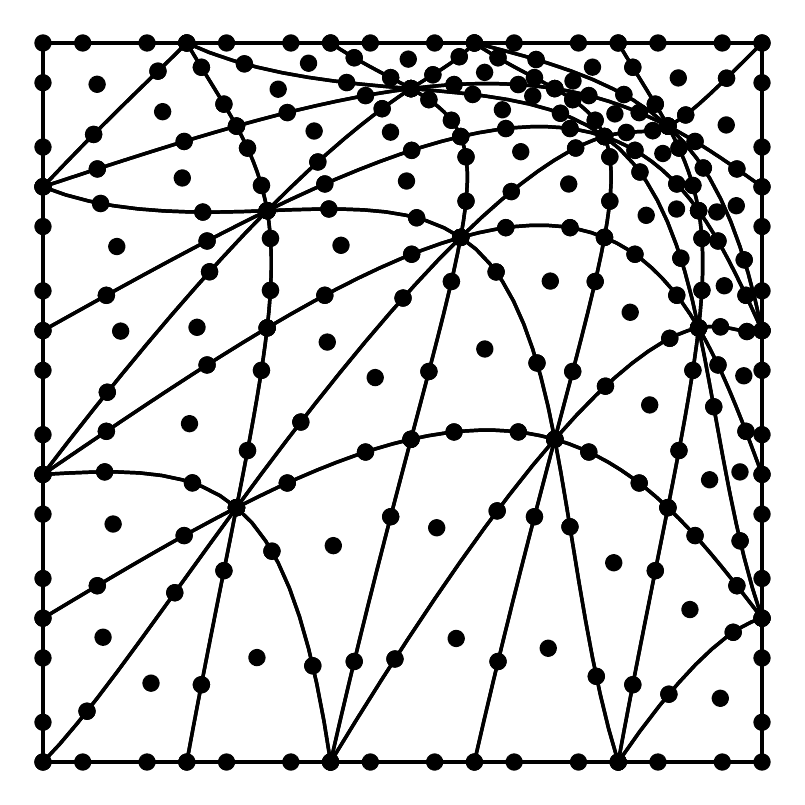}
 		\caption{$p_{\textrm{map}} = 3$}
 	\end{subfigure}
 	\caption{Warping of a split Cartesian mesh through Lagrange interpolation}\label{fig:meshes}
 \end{figure}

\subsubsection*{Temporal discretization}
The standard explicit fourth-order Runge-Kutta method is used to advance the solution in time, where the time step is chosen as $\Delta t \approx C_t h/a$,\footnote{Specifically, we take $\Delta t = T/ \lfloor T/ \Delta t^* \rfloor$, where $\Delta t^* = C_th/a$ is the target step size.} with a characteristic element size given by $h = L/M$ and a characteristic wave speed $a \in \R_+$ taken in \textsection \ref{sec:advection} to be the magnitude of the advection velocity and in \textsection \ref{sec:euler} to be the magnitude of the background velocity for the isentropic vortex. Motivated by the Courant-Friedrichs-Lewy (CFL) condition for the DG method in \cite[\textsection 2.2]{cockburn_shu_jsc_01}, we choose $C_t = \beta/(2p+1)$ with $\beta = 2.5 \times 10^{-3}$ in order to ensure that the error due to the temporal discretization is dominated by the error due to the spatial discretization.

\subsubsection*{Quantities of interest}
In order to provide numerical evidence in support of the theoretical results established in \textsection \ref{sec:analysis}, we define the following quantities of interest.
\begin{enumerate}[font=\itshape]
	\item \emph{Equivalence.} We evaluate the $L^2$ norm of the difference between each component of the strong-form and weak-form numerical solutions at time $T$, which is given in terms of $\fn{u}_e^{\textrm{strong}}(\cdot,T), \fn{u}_e^{\textrm{weak}}(\cdot,T) \in \set{V}_h$ by $\lVert\fn{u}_e^{\textrm{strong}}(\cdot,T) - \fn{u}_e^{\textrm{weak}}(\cdot,T) \rVert_{\ug{\Omega}}= \sqrt{\int_{\ug{\Omega}} (\fn{u}_e^{\textrm{strong}}(\vc{x},T) -  \fn{u}_e^{\textrm{weak}}(\vc{x},T))^2 \, \dd \vc{x}}$. The integral over $\ug{\Omega}$ is computed by separating the contributions from each element $\ug{\Omega}^{(\kappa)} \in \set{T}_h$, applying the corresponding transformation $\fv{X}^{(\kappa)}$ in order to use a quadrature rule of degree $4p$ due to Xiao and Gimbubtas \cite{xiao_gimbutas_quadrature_10}, and summing over all $\kappa \in \{1,\ldots, K\}$. This quantity is expected to be zero for the Quadrature I and Collocation schemes (but not necessarily for the Quadrature II schemes) as a result of Theorems \ref{thm:equiv_dg} and \ref{thm:fr_equiv}.
\item \emph{Conservation.} Considering the functional $\opr{C}$ as defined in Remark \ref{rmk:cons_functional}, we compute $\opr{C}(\fn{u}_e^h(\cdot,T))- \opr{C}(\fn{u}_e^h(\cdot,0))$, corresponding to the net production or destruction of each conserved quantity between $t=0$ and $t=T$, which is expected to be zero as a result of Theorems \ref{thm:conservation} and \ref{thm:conservation_strong}, the conditions of which are satisfied for all discretizations described in this section.
\item \emph{Energy difference.} In the case of the linear advection equation, we also evaluate the energy difference $\lVert \fn{u}^h(\cdot,T) \rVert_{\set{V}_h}^2 - \lVert \fn{u}^h(\cdot,0) \rVert_{\set{V}_h}^2$, measured in terms of the square of the discrete norm defined in Lemma \ref{lem:norm}, which, as discussed in Remark \ref{rmk:stability}, is expected to be nonincreasing with $\lambda = 1$ and invariant with $\lambda = 0$ for the Quadrature I and Collocation schemes (but not necessarily for the Quadrature II schemes) applied the periodic problem considered here. 
\end{enumerate}
\begin{remark}
In the results tabulated in \textsection \ref{sec:advection} and \textsection \ref{sec:euler}, values corresponding to properties established theoretically in \textsection \ref{sec:analysis} that are violated (i.e.\ due to design choices leading to the stated assumptions not being satisfied) are shown in italics. Furthermore, a horizontal line denotes a computation in which the scheme is observed to be unstable.
\end{remark}
\begin{remark}
Due to roundoff error and the influence of the temporal discretization (which is not accounted for in the theoretical analysis), we consider a quantity to be zero if it is on the order of $10^{-12}$ or less. This is many orders of magnitude smaller than the solution error for the grids considered in this work, which were deliberately chosen to be coarse in order to unambiguously demonstrate that the properties proven in \textsection \ref{sec:analysis} hold discretely regardless of the grid resolution, rather than asymptotically. 
\end{remark}
 
\subsection{Linear advection equation}\label{sec:advection}
We solve the constant-coefficient linear advection equation in (\ref{eq:advection}) on a square domain with $L=1$ and periodic boundary conditions, with the initial data given by $
\fn{u}^0(\vc{x}) \df \sin(2\pi x_1/L)\sin(2\pi x_2/L)$ and an advection velocity given by $\vc{a} \df  a[\cos(\theta), \sin(\theta)]^\T$ with $a = \sqrt{2}$ and $\theta = \pi/4$, stopping after one advection period (i.e.\ with $T = 1$). The mesh is constructed as described in \textsection \ref{sec:design_choices} with $M=8$, where we use $p_{\textrm{map}} = 1$ such that Assumption \ref{asm:affine} is satisfied. The results for polynomial degrees 2, 3, and 4 are shown in Tables \ref{tab:advection_p2}, \ref{tab:advection_p3}, and \ref{tab:advection_p4}, respectively. In all cases, the results for the Quadrature I and Collocation schemes are consistent with the theory, satisfying the equivalence, conservation, and stability (i.e.\ energy conservation for $\lambda =0$ and energy dissipation for $\lambda=1$) properties which were proven in \textsection \ref{sec:analysis}. As discussed in \textsection \ref{sec:design_choices}, the Quadrature II schemes violate Assumption \ref{asm:sbp}, and thus the theory in \textsection \ref{sec:equiv} and \textsection \ref{sec:stability} does not apply; accordingly, the strong-form and weak-form numerical solutions are different, and both the strong and weak forms are found to be unstable when a central flux (i.e.\ $\lambda =0$) is used. Although the Quadrature II schemes are stable in practice for an upwind flux (i.e.\ $\lambda=1$), no \emph{a priori} energy estimate exists due to the violation of the SBP property. The use of an upwind flux in such a case can therefore be interpreted as the addition of numerical dissipation to stabilize an unstable baseline discretization, where the required amount of dissipation is dependent on the particular problem and grid resolution. 

\begin{table}
	\caption{Linear advection equation, $p=2$}\label{tab:advection_p2}
\centering
\scriptsize
\begin{tabularx}{\textwidth}{@{\extracolsep{\fill}}*{8}{c}}
	\toprule
	{\multirow{2}{*}{\makecell{Discrete inner\\ products}}} &
	{\multirow{2}{*}{\makecell{$c$}}} &
	{\multirow{2}{*}{\makecell{$\lambda$}}}
	& {\multirow{2}{*}{\makecell{Equivalence}}}
	& \multicolumn{2}{c}{\makecell{Conservation}}
	&  \multicolumn{2}{c}{\makecell{Energy difference}} \\
	\cmidrule(lr){5-6}
	\cmidrule(lr){7-8}
	& & & & Strong & Weak & Strong  & Weak \\
	\midrule 
	Quadrature I & $c_{\mathrm{DG}}$ & 0& 8.558e-15 & -3.121e-16 & -5.757e-16 & -1.443e-15 & -1.887e-15\\ 
	&  & 1& 1.423e-15 & -1.278e-16 & -9.335e-17 & -5.847e-03 & -5.847e-03\\ 
	\cmidrule{2-8} 
	& $c_+$ & 0& 7.232e-15 & -7.081e-16 & -8.887e-16 & 1.665e-16 & 4.441e-16\\ 
	&  & 1& 1.320e-15 & 6.117e-16 & 5.333e-16 & -4.131e-02 & -4.131e-02\\ 
	\midrule 
	Quadrature II & $c_{\mathrm{DG}}$ & 0& --- &  --- & --- & --- & ---\\ 
	&  & 1& \textit{1.741e-02} & -1.382e-15 & 4.901e-16 & -5.308e-03 & -4.728e-03\\ 
	\cmidrule{2-8} 
	& $c_+$ & 0 & --- &  --- & --- & --- & ---\\ 
	&  & 1& \textit{2.410e-02} & 5.849e-16 & 3.646e-16 & -3.374e-02 & -4.306e-02\\ 
	\midrule 
	Collocation & $c_{\mathrm{DG}}$ & 0& 4.156e-15 & -6.516e-17 & -2.480e-16 & -7.577e-15 & -1.019e-14\\ 
	&  & 1& 3.201e-15 & -8.142e-17 & -4.937e-16 & -5.775e-03 & -5.775e-03\\ 
	\cmidrule{2-8} 
	& $c_+$ & 0& 5.068e-15 & -1.312e-16 & -3.735e-16 & -3.331e-16 & -4.441e-15\\ 
	&  & 1& 4.211e-15 & -2.686e-16 & -5.489e-16 & -4.008e-02 & -4.008e-02\\ 
		
\bottomrule	
\end{tabularx}
\end{table}
\begin{table}
	\caption{Linear advection equation, $p=3$}\label{tab:advection_p3}
	\centering
	\scriptsize
	\begin{tabularx}{\textwidth}{@{\extracolsep{\fill}}*{8}{c}}
		\toprule
		{\multirow{2}{*}{\makecell{Discrete inner\\ products}}} &
		{\multirow{2}{*}{\makecell{$c$}}} &
		{\multirow{2}{*}{\makecell{$\lambda$}}}
		& {\multirow{2}{*}{\makecell{Equivalence}}}
		& \multicolumn{2}{c}{\makecell{Conservation}}
		&  \multicolumn{2}{c}{\makecell{Energy difference}}\\
		\cmidrule(lr){5-6}
		\cmidrule(lr){7-8}
		& & & & Strong & Weak & Strong  & Weak  \\
		\midrule
		Quadrature I & $c_{\mathrm{DG}}$ & 0& 2.626e-14 & -3.973e-16 & -2.494e-18 & 1.610e-15 & 4.996e-16\\ 
		&  & 1& 2.436e-15 & -3.943e-16 & -1.884e-16 & -1.942e-04 & -1.942e-04\\ 
		\cmidrule{2-8} 
		& $c_+$ & 0& 1.742e-14 & -2.168e-19 & 3.306e-16 & 1.776e-15 & 2.220e-16\\ 
		&  & 1& 2.350e-15 & -1.997e-16 & -1.561e-17 & -1.860e-03 & -1.860e-03\\ 
		\midrule 
		Quadrature II & $c_{\mathrm{DG}}$ & 0 & --- &  --- & --- & --- & ---\\ 
		&  & 1& \textit{2.736e-03} & -1.476e-16 & -8.695e-17 & -1.831e-04 & -1.710e-04\\ 
		\cmidrule{2-8} 
		& $c_+$ & 0& --- &  --- & --- & --- & ---\\ 
		&  & 1& \textit{2.817e-03} & -3.589e-16 & 2.949e-17 & -1.709e-03 & -1.909e-03\\ 
		\midrule 
		Collocation & $c_{\mathrm{DG}}$ & 0& 6.336e-15 & -9.606e-17 & -1.621e-16 & -6.106e-16 & -1.832e-15\\ 
		&  & 1& 2.118e-15 & -3.253e-16 & -3.778e-16 & -1.970e-04 & -1.970e-04\\ 
		\cmidrule{2-8} 
		& $c_+$ & 0& 4.804e-15 & -1.233e-16 & -3.703e-16 & -2.220e-16 & 1.055e-15\\ 
		&  & 1& 1.754e-15 & -1.138e-16 & -2.719e-16 & -1.847e-03 & -1.847e-03\\
		\bottomrule	
	\end{tabularx}\vspace{2em}
\end{table}
\begin{table}
	\caption{Linear advection equation, $p=4$}\label{tab:advection_p4}
	\centering
	\scriptsize
	\begin{tabularx}{\textwidth}{@{\extracolsep{\fill}}*{8}{c}}
		\toprule
		{\multirow{2}{*}{\makecell{Discrete inner\\ products}}} &
		{\multirow{2}{*}{\makecell{$c$}}} &
		{\multirow{2}{*}{\makecell{$\lambda$}}}
		& {\multirow{2}{*}{\makecell{Equivalence}}}
		& \multicolumn{2}{c}{\makecell{Conservation}}
		&  \multicolumn{2}{c}{\makecell{Energy Difference}}\\
		\cmidrule(lr){5-6}
		\cmidrule(lr){7-8}
		& & & & Strong & Weak & Strong  & Weak \\
		\midrule
	Quadrature I & $c_{\mathrm{DG}}$ & 0& 1.219e-14 & 1.149e-16 & -3.945e-16 & -1.943e-15 & -9.437e-16\\ 
	&  & 1& 2.437e-15 & 4.878e-16 & 2.328e-16 & -4.326e-06 & -4.326e-06\\ 
	\cmidrule{2-8} 
	& $c_+$ & 0& 8.803e-15 & -6.771e-16 & -8.122e-16 & -1.665e-16 & -5.551e-17\\ 
	&  & 1& 2.342e-15 & 1.100e-16 & -1.149e-17 & -4.887e-05 & -4.887e-05\\ 
	\midrule 
	Quadrature II & $c_{\mathrm{DG}}$ & 0 & --- &  --- & --- & --- & ---\\ 
	&  & 1& \textit{4.366e-04} & -4.306e-16 & -1.510e-16 & -4.219e-06 & -3.999e-06\\ 
	\cmidrule{2-8} 
	& $c_+$ & 0 & --- &  --- & --- & --- & ---\\ 
	&  & 1& \textit{3.643e-04} & -1.304e-15 & 1.009e-15 & -4.910e-05 & -4.879e-05\\ 
	\midrule 
	Collocation & $c_{\mathrm{DG}}$ & 0& 1.016e-14 & -2.016e-16 & 1.746e-16 & -2.776e-15 & 1.887e-15\\ 
	&  & 1& 5.314e-15 & -1.081e-16 & -2.819e-17 & -4.360e-06 & -4.360e-06\\ 
	\cmidrule{2-8} 
	& $c_+$ & 0& 1.028e-14 & 5.302e-17 & 4.937e-16 & -2.276e-15 & 4.496e-15\\ 
	&  & 1& 7.288e-15 & -5.237e-17 & 1.223e-16 & -4.826e-05 & -4.826e-05\\ 
			\bottomrule	
	\end{tabularx}
\end{table}

\subsection{Euler equations}\label{sec:euler}
The propagation of an isentropic vortex is commonly used as a test case for numerical methods applied to unsteady nonlinear hyperbolic systems. Numerous versions of this problem have been posed in the literature, with the form considered in this work being a modification of that presented by Shu \cite[\textsection 5.1]{shu_eno_weno_98} following the general formulation described by Spiegel \etal \cite{spiegel_isentropic_euler_vortex_survey_15}. Considering the Euler equations in (\ref{eq:euler}) and denoting the Mach number and direction of the background flow as $\mathrm{Ma}_\infty \in \R_{\geq 0}$ and $\theta \in [0,2\pi)$, respectively, the initial velocity field is given for a vortex of strength $\varepsilon \in \R_+$ centred at $\vc{x}^0 \in \ug{\Omega}$ as
\begin{equation}\label{eq:init_vel}
\fv{v}^0(\vc{x}) \df \mathrm{Ma}_\infty\qty(\mqty[\cos(\theta)\\ \sin(\theta) ] + \varepsilon\exp\qty(1-\lVert \vc{x}-\vc{x}^0\rVert_2^2)\mqty[-(x_2-x_2^0) \\ x_1-x_1^0]), 
\end{equation}
the initial temperature field is prescribed as
\begin{equation}\label{eq:init_temp}
\pi^0(\vc{x}) \df 1 -\frac{(r-1)\varepsilon^2\mathrm{Ma}_\infty^2}{2}\exp(1-\lVert \vc{x}-\vc{x}^0\rVert_2^2),
\end{equation}
and constant entropy is enforced throughout the spatial domain. Using (\ref{eq:init_vel}) and (\ref{eq:init_temp}), as well as the equation of state and isentropic relations for an ideal gas with constant specific heat, the initial condition is then given by
\begin{equation}
\fv{u}^0(\vc{x}) \df  \mqty[(\pi^0(\vc{x}))^{1/(r-1)} \\ (\pi^0(\vc{x}))^{1/(r-1)}\fv{v}^0(\vc{x})  \\ 
\frac{1}{r-1}(\pi^0(\vc{x}))^{r/(r-1)} +\frac{1}{2}(\pi^0(\vc{x}))^{1/(r-1)}\lVert\fv{v}^0(\vc{x})\rVert_2^2].
\end{equation}
Similarly to the advection problem described in \textsection \ref{sec:advection}, such an initial condition results in an exact solution on an infinite domain given by the advection of the vortex at the background velocity. For the numerical experiments in this section, we enforce periodic boundary conditions on a square domain of length $L = 10$, and we set $\mathrm{Ma}_\infty = 0.4$, $r = 1.4$, $\theta = \pi/4$, $\varepsilon = 5\sqrt{2}\exp(1/2)/(4\pi)$, and $\vc{x}^0 = [5,5]^\T$ to recover a subsonic vortex which is otherwise similar to the supersonic problem described by Shu \cite{shu_eno_weno_98}. The solution is evolved in time for one period (i.e.\ with $T =\sqrt{2}L/\mathrm{Ma}_{\infty}$) on a mesh constructed as described in \textsection \ref{sec:design_choices} with $M=16$ and $p_{\textrm{map}} = p$.
\par
The results for polynomial degrees 2, 3, and 4 are shown in Tables \ref{tab:euler_p2}, \ref{tab:euler_p3}, and \ref{tab:euler_p4}, respectively, where we see that the strong and weak forms are equivalent for the Quadrature I and Collocation schemes, and the discrete integrals of all four solution variables are invariant for all three approaches, as expected from the theoretical analysis in \textsection \ref{sec:equiv} and \textsection\ref{sec:conservation}. For the Quadrature II schemes, the strong and weak forms again result in different numerical solutions due to Assumption \ref{asm:sbp} not being satisfied. We do not investigate energy balances for this problem, as the theory in \textsection \ref{sec:stability} is not applicable as a result of the system being nonlinear and the mapping not being affine. Hence, the numerical results presented in this section for the Euler equations, as with those presented in \textsection \ref{sec:advection} for the linear advection equation, are consistent with the theory in \textsection \ref{sec:analysis}.
\begin{table}[t]
	\centering
	\caption{Euler equations, $p=2$}\label{tab:euler_p2}
	\scriptsize
\begin{tabularx}{\textwidth}{@{\extracolsep{\fill}}*{6}{c}}
	\toprule
	{\multirow{2}{*}{\makecell{Discrete inner\\ products}}} 
	& {\multirow{2}{*}{$c$}} & {\multirow{2}{*}{Equation}} 
	& {\multirow{2}{*}{\makecell{Equivalence}}}
	& \multicolumn{2}{c}{\makecell{Conservation}}\\
	\cmidrule(lr){5-6}
	& & & & Strong & Weak \\
	\midrule
	Quadrature I & $c_{\mathrm{DG}}$ & $\rho$ & 5.974e-14 & 2.700e-13 & -8.527e-14\\ 
	& & $\rho V_1$ & 1.452e-13 & 3.197e-13 & -7.709e-13\\ 
	& & $\rho V_2$ & 1.538e-13 & 1.009e-12 & 1.016e-12\\ 
	& & $E$ & 1.478e-13 & 1.052e-12 & 3.126e-13\\ 
	\cmidrule{2-6} 
	& $c_+$ & $\rho$ & 5.124e-14 & 3.979e-13 & 1.137e-13\\ 
	& & $\rho V_1$ & 1.434e-13 & 3.411e-13 & -7.461e-13\\ 
	& & $\rho V_2$ & 1.602e-13 & 1.027e-12 & 1.140e-12\\ 
	& & $E$ & 1.334e-13 & 1.165e-12 & 2.558e-13\\ 
	\midrule 
	Quadrature II & $c_{\mathrm{DG}}$ & $\rho$ & \textit{3.019e-02} & 9.948e-14 & -1.563e-13\\ 
	& & $\rho V_1$ & \textit{5.370e-02} & -5.045e-13 & -1.787e-12\\ 
	& & $\rho V_2$ & \textit{5.870e-02} & 1.116e-12 & 1.631e-12\\ 
	& & $E$ & \textit{6.608e-02} & 3.695e-13 & -5.400e-13\\ 
	\cmidrule{2-6} 
	& $c_+$ & $\rho$ & \textit{2.308e-02} & 1.847e-13 & -4.263e-14\\ 
	& & $\rho V_1$ & \textit{4.605e-02} & -4.725e-13 & -1.901e-12\\ 
	& & $\rho V_2$ & \textit{4.414e-02} & 1.101e-12 & 1.705e-12\\ 
	& & $E$ & \textit{5.742e-02} & 4.832e-13 & -8.527e-14\\ 
	\midrule 
	Collocation & $c_{\mathrm{DG}}$ & $\rho$ & 1.616e-13 & 5.400e-13 & -6.253e-13\\ 
	& & $\rho V_1$ & 3.850e-13 & -7.816e-14 & -2.576e-12\\ 
	& & $\rho V_2$ & 4.296e-13 & 2.515e-12 & 9.663e-13\\ 
	& & $E$ & 4.353e-13 & 1.620e-12 & -1.535e-12\\ 
	\cmidrule{2-6} 
	& $c_+$ & $\rho$ & 9.546e-14 & 5.116e-13 & -2.842e-13\\ 
	& & $\rho V_1$ & 2.136e-13 & -4.263e-14 & -1.172e-12\\ 
	& & $\rho V_2$ & 2.176e-13 & 2.512e-12 & 1.034e-12\\ 
	& & $E$ & 2.586e-13 & 1.705e-12 & -3.411e-13\\ 
	\bottomrule
\end{tabularx}
\end{table}

\begin{table}[t]
	\centering
	\caption{Euler equations, $p=3$}\label{tab:euler_p3}
	\scriptsize
	\begin{tabularx}{\textwidth}{@{\extracolsep{\fill}}*{6}{c}}
	\toprule
	{\multirow{2}{*}{\makecell{Discrete inner\\ products}}} 
	& {\multirow{2}{*}{$c$}} & {\multirow{2}{*}{Equation}} 
	& {\multirow{2}{*}{\makecell{Equivalence}}}
	& \multicolumn{2}{c}{\makecell{Conservation}}\\
	\cmidrule(lr){5-6}
	& & & & Strong & Weak \\
	\midrule
	Quadrature I & $c_{\mathrm{DG}}$ & $\rho$ & 6.875e-14 & 6.821e-13 & 3.411e-13\\ 
	& & $\rho V_1$ & 1.687e-13 & -6.679e-13 & -7.851e-13\\ 
	& & $\rho V_2$ & 1.472e-13 & 2.419e-12 & 1.329e-12\\ 
	& & $E$ & 1.660e-13 & 1.506e-12 & 5.684e-13\\ 
	\cmidrule{2-6} 
	& $c_+$ & $\rho$ & 7.187e-14 & 6.821e-13 & 1.990e-13\\ 
	& & $\rho V_1$ & 1.678e-13 & -6.679e-13 & -9.486e-13\\ 
	& & $\rho V_2$ & 1.542e-13 & 2.377e-12 & 1.300e-12\\ 
	& & $E$ & 1.777e-13 & 1.535e-12 & 3.126e-13\\ 
	\midrule 
	Quadrature II & $c_{\mathrm{DG}}$ & $\rho$ & \textit{1.407e-02} & 4.690e-13 & 1.421e-13\\ 
	& & $\rho V_1$ & \textit{2.250e-02} & -1.002e-12 & -1.297e-12\\ 
	& & $\rho V_2$ & \textit{2.674e-02} & 1.982e-12 & 1.670e-12\\ 
	& & $E$ & \textit{2.179e-02} & 9.663e-13 & 2.558e-13\\ 
	\cmidrule{2-6} 
	& $c_+$ & $\rho$ & \textit{1.073e-02} & 3.126e-13 & 2.132e-13\\ 
	& & $\rho V_1$ & \textit{1.688e-02} & -1.020e-12 & -1.197e-12\\ 
	& & $\rho V_2$ & \textit{1.904e-02} & 1.954e-12 & 1.616e-12\\ 
	& & $E$ & \textit{1.771e-02} & 6.537e-13 & 3.126e-13\\ 
	\midrule 
	Collocation & $c_{\mathrm{DG}}$ & $\rho$ & 5.045e-14 & -1.847e-13 & -2.984e-13\\ 
	& & $\rho V_1$ & 1.031e-13 & -2.128e-12 & -2.370e-12\\ 
	& & $\rho V_2$ & 1.029e-13 & 1.794e-12 & 1.698e-12\\ 
	& & $E$ & 9.036e-14 & -3.979e-13 & -5.969e-13\\ 
	\cmidrule{2-6} 
	& $c_+$ & $\rho$ & 1.390e-13 & -1.279e-13 & 9.948e-14\\ 
	& & $\rho V_1$ & 4.973e-13 & -2.100e-12 & 5.294e-13\\ 
	& & $\rho V_2$ & 5.055e-13 & 1.794e-12 & -3.908e-13\\ 
	& & $E$ & 3.993e-13 & -4.263e-13 & 2.274e-13\\
		\bottomrule
	\end{tabularx}
\end{table}

\begin{table}[t]
	\centering
	\caption{Euler equations, $p=4$}\label{tab:euler_p4}
	\scriptsize
\begin{tabularx}{\textwidth}{@{\extracolsep{\fill}}*{6}{c}}
	\toprule
	{\multirow{2}{*}{\makecell{Discrete inner\\ products}}} 
	& {\multirow{2}{*}{$c$}} & {\multirow{2}{*}{Equation}} 
	& {\multirow{2}{*}{\makecell{Equivalence}}}
	& \multicolumn{2}{c}{\makecell{Conservation}}\\
	\cmidrule(lr){5-6}
	& & & & Strong & Weak \\
	\midrule
	Quadrature I & $c_{\mathrm{DG}}$ & $\rho$ & 5.573e-14 & 4.690e-13 & 3.553e-13\\ 
	& & $\rho V_1$ & 6.885e-14 & -1.528e-13 & -3.126e-13\\ 
	& & $\rho V_2$ & 6.624e-14 & 1.219e-12 & 1.098e-12\\ 
	& & $E$ & 8.680e-14 & 6.821e-13 & 3.695e-13\\ 
	\cmidrule{2-6} 
	& $c_+$ & $\rho$ & 5.004e-14 & 5.116e-13 & 4.121e-13\\ 
	& & $\rho V_1$ & 7.479e-14 & -2.025e-13 & -4.299e-13\\ 
	& & $\rho V_2$ & 7.004e-14 & 1.190e-12 & 1.094e-12\\ 
	& & $E$ & 7.689e-14 & 8.242e-13 & 6.537e-13\\ 
	\midrule 
	Quadrature II & $c_{\mathrm{DG}}$ & $\rho$ & \textit{8.368e-03} & 6.395e-13 & -5.684e-14\\ 
	& & $\rho V_1$ & \textit{7.560e-03} & 1.421e-13 & -1.791e-12\\ 
	& & $\rho V_2$ & \textit{7.803e-03} & 1.773e-12 & 1.911e-12\\ 
	& & $E$ & \textit{5.937e-03} & 1.620e-12 & -2.274e-13\\ 
	\cmidrule{2-6} 
	& $c_+$ & $\rho$ & \textit{4.230e-03} & 7.958e-13 & 4.263e-14\\ 
	& & $\rho V_1$ & \textit{5.130e-03} & 1.421e-13 & -1.986e-12\\ 
	& & $\rho V_2$ & \textit{5.918e-03} & 1.759e-12 & 1.897e-12\\ 
	& & $E$ & \textit{4.764e-03} & 1.705e-12 & 2.842e-14\\ 
	\midrule 
	Collocation & $c_{\mathrm{DG}}$ & $\rho$ & 1.531e-13 & 2.700e-13 & 9.379e-13\\ 
	& & $\rho V_1$ & 4.463e-13 & 3.553e-14 & -4.832e-13\\ 
	& & $\rho V_2$ & 4.288e-13 & 2.878e-13 & 3.276e-12\\ 
	& & $E$ & 4.082e-13 & -1.421e-13 & 1.990e-12\\ 
	\cmidrule{2-6} 
	& $c_+$ & $\rho$ & 1.862e-13 & 1.705e-13 & 1.165e-12\\ 
	& & $\rho V_1$ & 5.534e-13 & 3.553e-14 & -3.908e-14\\ 
	& & $\rho V_2$ & 5.066e-13 & 3.091e-13 & 3.929e-12\\ 
	& & $E$ & 5.137e-13 & 1.137e-13 & 2.956e-12\\
	\bottomrule
\end{tabularx}
\end{table}

\section{Conclusions}\label{sec:conclusions}
We have proposed a unifying framework enabling the algebraic formulation of a broad class of high-order DG and FR methods applied to systems of conservation laws, facilitating the unified theoretical analysis of such schemes through matrix-based techniques. The role of the multidimensional SBP property is examined in detail for methods within the proposed framework, revealing new insights and generalizing existing results pertaining to the equivalence, conservation, and stability properties of DG and FR methods. Potential avenues for future research include the application of similar unifying principles to nonlinearly stable schemes and problems with diffusive terms as well as the comparative evaluation of methods within and outside of the present framework. Specifically, it would be interesting to examine the difference in computational expense between strong-form and weak-form implementations of methods which we have shown in this paper to be mathematically equivalent, and to compare the DG and FR methods considered in this paper with collocated multidimensional SBP schemes without an analytical basis (e.g.\ those described in \cite{hicken_mdsbp_16} and \cite{delrey_mdsbp_sat_18}). Additionally, the theory could be extended to enable proofs of convergence, establishing \emph{a priori} error estimates for linear problems with smooth solutions as consequences of the SBP property, which was discussed by Sv\"ard and Nordstr\"om \cite{svard_nordstrom_convergence_rate_energystable_19} in the context of finite-difference schemes in one space dimension.

\begin{acknowledgements}
The authors are grateful to Gianmarco Mengaldo, Masayuki Yano, and David Del Rey Fern\'andez for their feedback. Portions of the computations were performed on the Niagara supercomputer at the SciNet HPC Consortium \cite{ponce_niagara_19}. The two-dimensional polynomial basis functions, interpolation nodes, and quadrature rules were generated using \texttt{modepy}, developed by Andreas Kl\"ockner. The initial uniform triangular grids and the LGL quadrature rules were generated using \texttt{meshzoo} and \texttt{quadpy}, respectively, both developed by Nico Schl\"omer. All plots appearing in this paper were produced using Matplotlib \cite{hunter_matplotlib_07}.
\end{acknowledgements}

\section*{Statements}

\subsection*{Code availability}
\small
The code developed for this work is publicly available on GitHub under the GNU General Public License at \url{https://github.com/tristanmontoya/GHOST/tree/algebraic_framework_v2}.
\subsection*{Availability of data and material}
The simulation data analyzed in \textsection \ref{sec:numerical} are available on GitHub within the above repository.
\subsection*{Funding}
The first author acknowledges the financial support provided by the University of Toronto, the Natural Sciences and Engineering Research Council of Canada, and the Government of Ontario. SciNet is funded by the Canada Foundation for Innovation, the Government of Ontario, the Ontario Research Fund -- Research Excellence, and the University of Toronto.
\subsection*{Conflicts of interest/Competing interests}
The authors declare that there are no conflicts of interest or competing interests associated with this work.

\appendix
{
\section{Polynomial Bases}\label{app:bases}
\subsection{Modal basis}\label{app:modal_basis}
A basis $\set{B}_0 \df \{\phi_0^{(i)}\}_{i=1}^{N_p}$ for the space $\set{V}$ is orthogonal with respect to the standard $L^2$ inner product if $\il \phi_0^{(i)}, \phi_0^{(j)} \ir_{\hat{\ug{\Omega}}} = 0$ for all $i \neq j$. While such basis functions may be obtained numerically through the Gram-Schmidt orthogonalization of a monomial basis, analytical expressions exist for certain spaces and element types. A two-variable analogue of the Legendre polynomial system was first proposed for triangular elements by Proriol \cite{proriol_polynomials_57} (see also Koornwinder \cite[\textsection 3.3]{koornwinder_orthogonal_polynomials_75} and Dubiner \cite[\textsection 5]{dubiner_spectral_triangle_91}), and is given as a basis for $\set{P}_p(\hat{\ug{\Omega}})$ with $\hat{\ug{\Omega}} \df \{\hat{\vc{x}} \in \R^2 : \hat{x}_1 \geq -1, \ \hat{x}_2 \geq -1, \ \hat{x}_1 + \hat{x}_2 \leq 1 \}$ by
\begin{equation}\label{eq:triangle_basis}
\phi_0^{I(\vc{\alpha})}(\hat{\vc{x}}) \df \fn{P}_{\alpha_1}^{(0,0)}\big(2(1 + \hat{x}_1)/(1-\hat{x}_2) - 1 \big)\qty(\frac{1-\hat{x}_2}{2})^{\alpha_1} \fn{P}_{\alpha_2}^{(2\alpha_1+1,0)}(\hat{x}_2),
\end{equation}
where $\fn{I} : \set{N} \to \{1,\ldots, N_p\}$ reduces the multi-index set $\set{N} \df \{\vc{\alpha} \in \mathbb{N}_0^d : \lvert \vc{\alpha} \rvert \leq p \}$ to a single index, and we define $\fn{P}_k^{(a,b)}$ as the degree $k \in \mathbb{N}_0$ Jacobi polynomial corresponding to the weight function $\fn{W}(\hat{x}) \df (1-\hat{x})^a(1+\hat{x})^b$ for $a,b > 1$. In the case of $\set{Q}_p(\hat{\ug{\Omega}})$ with $\hat{\ug{\Omega}} \df [-1,1]^d$, an orthogonal tensor-product basis is given by $\phi_0^{I(\vc{\alpha})}(\hat{\vc{x}}) \df \fn{L}_{\alpha_1}(\hat{x}_1) \cdots \fn{L}_{\alpha_d}(\hat{x}_d)$, where $\fn{L}_k(\hat{x}) \df \fn{P}_k^{(0,0)}(\hat{x})$ denotes the $k^{\textrm{th}}$ Legendre polynomial. Similar bases may be constructed on tetrahedral, prismatic, and pyramidal elements, where we refer to the textbook of Karniadakis and Sherwin \cite[Ch.\ 3]{karniadakis_sherwin_spectral_hp_element} for further details regarding multivariate orthogonal polynomials and their use as basis functions for high-order methods.
\subsection{Nodal basis}\label{app:nodal_basis}
While it is possible to directly employ a modal basis to represent the numerical solution such that $\set{B} = \set{B}_0$, we may alternatively choose $\set{B}$ to be a nodal (i.e.\ Lagrange) basis associated with $\tilde{\set{S}} \df \{\tilde{\hat{\vc{x}}}^{(i)}\}_{i=1}^{N_p} \subset \hat{\ug{\Omega}}$ such that $\phi^{(i)}(\tilde{\hat{\vc{x}}}^{(j)}) = \delta_{ij}$, which is satisfied for $d=1$ with $
\phi^{(i)}(\hat{x}) \df \prod_{j=1, j\neq i}^{p+1} (\hat{x} - \tilde{\hat{x}}^{(j)})/(\tilde{\hat{x}}^{(i)} - \tilde{\hat{x}}^{(j)} )$ when the nodes in $\tilde{\set{S}}$ are distinct. When $d \geq 2$, such analytical expressions for Lagrange bases do not exist for arbitrary non-tensorial nodal sets, and we may then take a more general approach by constructing a Vandermonde matrix $\tilde{\mat{V}} \in \R^{N_p \times N_p}$ for the modal basis $\set{B}_0$ described in \textsection \ref{app:modal_basis} such that $\tilde{\el{V}}_{ij} = \phi_0^{(j)}(\tilde{\hat{\vc{x}}}^{(i)})$. If the nodal set $\tilde{\set{S}}$ is unisolvent for the space $\set{V}$, the matrix $\tilde{\mat{V}}$ is invertible, and hence a nodal basis may be obtained as $\phi^{(i)}(\hat{\vc{x}}) \df \sum_{i=1}^{N_p}[\tilde{\mat{V}}^{-\T}]_{ij}\phi_0^{(j)}(\hat{\vc{x}})$. Further details regarding the construction of nodal bases may be found in the textbook of Hesthaven and Warburton \cite[\textsection 3.1, \textsection 6.1, \textsection 10.1]{hesthaven08}.

\section{Discrete inner products}\label{app:disc_in_prod}
\subsection{Quadrature-based approximation}\label{app:quadrature}
Considering quadrature rules with abscissae given by $\set{S} \df \{\hat{\vc{x}}^{(i)}\}_{i=1}^{N} \subset \hat{\ug{\Omega}}$ and $\set{S}^{(\gamma)} \df \{\hat{\vc{x}}^{(\gamma,i)}\}_{i=1}^{N_\gamma} \subset \hat{\ug{\Gamma}}^{(\gamma)}$, corresponding to weights $\set{W} \df \{\omega^{(i)}\}_{i=1}^N$ and  $\set{W}^{(\gamma)} \df \{\omega^{(\gamma,i)}\}_{i=1}^{N_\gamma}$, we may approximate the $L^2$ inner products on the reference element and each of its facets using the discrete inner products $\il\fn{u},\fn{v}\ir_{\mat{W}} \df \sum_{i=1}^N\omega^{(i)} \fn{u}(\hat{\vc{x}}^{(i)}) \fn{v}(\hat{\vc{x}}^{(i)})$ and $\il\fn{u},\fn{v}\ir_{\mat{W}^{(\gamma)}} \df \sum_{i=1}^{N_\gamma}\omega^{(\gamma,i)} \fn{u}(\hat{\vc{x}}^{(\gamma,i)}) \fn{v}(\hat{\vc{x}}^{(\gamma,i)})$, corresponding to (\ref{eq:disc_in_prod_vol}) and (\ref{eq:disc_in_prod_facet}), respectively, with $\mat{W} \df \textrm{diag}([\omega^{(1)}, \ldots, \omega^{(N)}]^\T)$ and $\mat{W}^{(\gamma)} \df \textrm{diag}([\omega^{(\gamma,1)}, \ldots, \omega^{(\gamma,N_\gamma)}]^\T)$. The following lemmas provide sufficient conditions on such quadrature rules for Assumption \ref{asm:sbp}, and consequently, the SBP property in (\ref{eq:sbp}), to be satisfied.
\begin{lemma}\label{lem:total_sbp}
Assumption \ref{asm:sbp} is satisfied for any reference element type with $\set{V} = \set{P}_p(\hat{\ug{\Omega}})$ when the discrete inner products computed using quadrature rules on $\set{S}$ and $\set{S}^{(\gamma)}$ of total degree greater than or equal to $2p-1$ and $2p$, respectively.
\end{lemma}
\begin{proof}
For any $\fn{u},\fn{v} \in \set{P}_p(\hat{\ug{\Omega}})$, the volume integrals in (\ref{eq:ibp}) contain functions belonging to the space $\set{P}_{2p-1}(\hat{\ug{\Omega}})$, while the surface integrals contain functions in $\set{P}_{2p}(\hat{\ug{\Gamma}}^{(\gamma)})$. For volume and facet quadrature rules of total degree $2p-1$ or greater and $2p$ or greater, respectively, all terms in (\ref{eq:ibp}) can be computed exactly, and (\ref{eq:ibp_polynomials_discrete}) is therefore satisfied for all $m \in \{1,\ldots, d\}$.
\end{proof}
\begin{lemma}\label{lem:tensor_sbp}
For $\hat{\ug{\Omega}} = [-1,1]^d$ and $\set{V} =\set{Q}_p(\hat{\ug{\Omega}})$, Assumption \ref{asm:sbp} is satisfied when the discrete inner products (on the reference element and each of its facets) are evaluated using tensor products of a one-dimensional quadrature rule of at least degree $2p-1$ on $[-1,1]$.
\end{lemma}
\begin{proof}
The proof is similar to that in Kopriva and Gassner \cite[\textsection 3]{kopriva_nodaldg_choosing_quadrature_weakform_10} and involves expanding $\fn{u},\fn{v} \in \set{Q}_p(\hat{\ug{\Omega}})$ in terms of the tensor product of  $\set{B}_{\textrm{1D}} \df \{\phi_{\textrm{1D}}^{(i)}\}_{i=1}^{p+1}$ for the space $\set{P}_p([-1,1])$ such that $\fn{u}(\hat{\vc{x}}) = \sum_{\vc{\alpha} \in \set{N}}\el{u}^{(\vc{\alpha})}\phi_{\textrm{1D}}^{(\alpha_1)}(\hat{x}_1) \cdots \phi_{\textrm{1D}}^{(\alpha_d)}(\hat{x}_d)$ and $\fn{v}(\hat{\vc{x}}) = \sum_{\vc{\alpha} \in \set{N}}\el{v}^{(\vc{\alpha})}\phi_{\textrm{1D}}^{(\alpha_1)}(\hat{x}_1) \cdots \phi_{\textrm{1D}}^{(\alpha_d)}(\hat{x}_d)$, where we recall the definition of $\set{N}$ for a tensor-product polynomial space in \textsection \ref{sec:approx_space}. Defining the abscissae and weights of the one-dimensional quadrature rule as $\set{S}_{\textrm{1D}} \df \{\hat{x}_{\textrm{1D}}^{(i)}\}_{i=1}^{N_{\textrm{1D}}}$ and $\set{W}_{\textrm{1D}} \df \{\omega_{\textrm{1D}}^{(i)}\}_{i=1}^{N_{\textrm{1D}}}$, respectively, we may factorize the discrete inner product as in \cite[Eq.\ (32)]{kopriva_nodaldg_choosing_quadrature_weakform_10}, and the relation in (\ref{eq:ibp_polynomials_discrete}) therefore follows from the fact that
\begin{equation}\label{eq:ibp_quad_1d}
\begin{multlined}
\sum_{k=1}^{N_{\textrm{1D}}}\omega_{\textrm{1D}}^{(k)}\phi_{\textrm{1D}}^{(i)}(\hat{x}_{\textrm{1D}}^{(k)}\normalsize)\pdv{\phi_{\textrm{1D}}^{(j)}}{\hat{x}}\normalsize(\hat{x}_{\textrm{1D}}^{(k)}\normalsize)+\sum_{k=1}^{N_{\textrm{1D}}}\omega_{\textrm{1D}}^{(k)}\pdv{\phi_{\textrm{1D}}^{(i)}}{\hat{x}}\normalsize(\hat{x}_{\textrm{1D}}^{(k)}\normalsize) \phi_{\textrm{1D}}^{(j)}(\hat{x}_{\textrm{1D}}^{(k)}\normalsize) \\ = \phi_{\textrm{1D}}^{(i)}(1)\phi_{\textrm{1D}}^{(j)}(1)  -\phi_{\textrm{1D}}^{(i)}(-1)\phi_{\textrm{1D}}^{(j)}(-1)
\end{multlined}
\end{equation}
is satisfied for all $i,j \in \{1,\ldots, p+1\}$ when the quadrature is of degree $2p-1$ or higher.
\end{proof}
\begin{remark}
Unlike Lemma \ref{lem:total_sbp}, Lemma \ref{lem:tensor_sbp} does not presuppose any accuracy conditions on the facet quadrature, and therefore ensures that Assumption \ref{asm:sbp} is satisfied, for example, in the case of tensor-product LGL quadrature rules with $p+1$ nodes in each direction (which were considered for the DGSEM-LGL in \cite{kopriva_nodaldg_choosing_quadrature_weakform_10}), for which the corresponding facet quadrature is only of degree $2p-1$.
\end{remark}
\begin{remark}\label{rmk:spd}
The matrices $\mat{W}$ and $\mat{W}^{(\gamma)}$ are SPD if and only if the quadrature weights are strictly positive, which is easily checked. However, Assumption \ref{asm:inner_prod} also requires $\mat{V}$ to be of rank $N_p$, which is difficult to prove when $d \geq 2$. In practice, the rank of $\mat{V}$ is verified numerically, and we note that Assumption \ref{asm:inner_prod} was confirmed to hold for all schemes described in \textsection \ref{sec:design_choices}.
\end{remark}

\subsection{Collocation-based approximation}\label{app:collocation}
Assuming, as in \textsection \ref{sec:review_fr}, that $\set{S}$ is unisolvent for the space $\set{V}$, we may take $\set{B}$ to be a nodal basis (as described in \textsection \ref{app:nodal_basis}) with $\tilde{\set{S}} = \set{S}$, which results in $\mat{V} = \mat{I}$. The corresponding interpolation operator satisfying $(\opr{I}\fn{v})(\hat{\vc{x}}^{(i)}) = \fn{v}(\hat{\vc{x}}^{(i)})$ for all $i \in \{1,\ldots, N\}$ is then given explicitly as $(\opr{I}\fn{v})(\hat{\vc{x}}) \df \sum_{i=1}^{N_p}\fn{v}(\hat{\vc{x}}^{(i)})\phi^{(i)}(\hat{\vc{x}})$. Similarly, if $\set{S}^{(\gamma)}$ is unisolvent for the trace space $\set{V}^{(\gamma)} \df \{\hat{\ug{\Gamma}}^{(\gamma)} \ni \hat{\vc{x}} \mapsto \fn{v}(\hat{\vc{x}}) \in \R: \exists \fn{u} \in \set{V} : \fn{v} = \fn{u}\lvert_{\hat{\ug{\Gamma}}^{(\gamma)}} \}$ of dimension $N_p^{(\gamma)}$, a nodal basis $\set{B}^{(\gamma)} \df \{\phi^{(\gamma,i)}\}_{i=1}^{N_p^{(\gamma)}}$ may be defined for each facet of the reference element such that $\phi^{(\gamma,i)}(\hat{\vc{x}}^{(\gamma,j)}) = \delta_{ij}$, with a corresponding interpolation operator $(\opr{I}^{(\gamma)}\fn{v})(\hat{\vc{x}}) \df \sum_{i=1}^{N_p^{(\gamma)}}\fn{v}(\hat{\vc{x}}^{(\gamma,i)})\phi^{(\gamma,i)}(\hat{\vc{x}})$. The associated discrete inner products are thus obtained as $\il\fn{u},\fn{v}\ir_{\mat{W}} \df \il\opr{I}\fn{u},\opr{I}\fn{v}\ir_{\hat{\Omega}}$ and $\il\fn{u},\fn{v}\ir_{\mat{W}^{(\gamma)}} \df \il\opr{I}^{(\gamma)}\fn{u},\opr{I}^{(\gamma)}\fn{v}\ir_{\hat{\ug{\Gamma}}^{(\gamma)}}$, which take the forms given in (\ref{eq:disc_in_prod_vol}) and (\ref{eq:disc_in_prod_facet}), respectively, with $\el{W}_{ij} = \el{M}_{ij} = \il\phi^{(i)},\phi^{(j)}\ir_{\hat{\ug{\Omega}}}$ and $\smash{\el{W}_{ij}^{(\gamma)}} = \il\phi^{(i)},\phi^{(j)}\ir_{\hat{\ug{\Gamma}}^{(\gamma)}}$, where the integrals may be evaluated in a preprocessing stage, either analytically or using quadrature rules of sufficient degree. Since functions belonging to $\set{V}$ and $\set{V}^{(\gamma)}$ are invariant under the interpolation operators $\opr{I}$ and $\opr{I}^{(\gamma)}$, respectively, the discrete inner products are exactly equal to the corresponding $L^2$ inner products for such functions. Moreover, the matrices $\mat{W}$ and $\mat{W}^{(\gamma)}$ are generally dense, and are ensured to be SPD due to the linear independence of the functions in $\set{B}$ and $\set{B}^{(\gamma)}$, respectively. Assumptions \ref{asm:sbp} and \ref{asm:inner_prod} are therefore both satisfied by construction, enabling the formulation of linearly stable discretizations on unisolvent nodal sets for which the quadrature accuracy requirements of Lemmas \ref{lem:total_sbp} and \ref{lem:tensor_sbp} may not be satisfied.
\par

\begin{remark}
If the interpolatory quadrature rule with abscissae $\set{S}$ and corresponding weights given by $\omega^{(i)} =  \sum_{j=1}^{N_p} \el{W}_{ij} =  \int_{\hat{\ug{\Omega}}}\phi^{(i)}(\hat{\vc{x}})\, \dd \vc{x}$ is exact for all products of two functions in $\set{V}$, it can be shown that $\mat{W}$ reduces to a diagonal matrix of quadrature weights, recovering an approximation identical to that described in \textsection \ref{app:quadrature}. An analogous equivalence holds for $\mat{W}^{(\gamma)}$.  
\end{remark}
\begin{remark}
Collocation-based discrete inner products may also be used with modal bases and over-integrated formulations (i.e.\ those with $N > N_p$ or $N_\gamma > N_p^{(\gamma)})$ by defining $\opr{I}$ using an auxiliary nodal basis $\tilde{\set{B}}$ for the space $\tilde{\set{V}}$, distinct from the solution basis $\set{B}$, and similarly defining the operators $\set{I}^{(\gamma)}$ using nodal bases for spaces $\tilde{\set{V}}^{(\gamma)}$ which may differ from $\set{V}^{(\gamma)}$. Assumption \ref{asm:sbp} is then satisfied for such schemes when $\tilde{\set{V}} \supseteq \set{V}$ and $\tilde{\set{V}}^{(\gamma)} \supseteq \set{V}^{(\gamma)}$ for all $\gamma \in \{1,\ldots, N_f\}$. In the case of $\tilde{\set{V}} \neq \set{V}$, we may express the orthogonality condition in Lemma \ref{lem:proj} as $\il \fn{v}, \Pi\fn{u}- \opr{I}\fn{u}  \ir_{\hat{\ug{\Omega}}} = 0$ for all $\fn{v} \in \set{V}$, and it therefore follows (see, for example, \cite[\textsection 5]{dubiner_spectral_triangle_91}) that the discrete projection $\Pi$ corresponds an interpolation $\opr{I}$ onto the space $\tilde{\set{V}}$, followed by an exact $L^2$ projection onto $\set{V}$.

\end{remark}

\section{Multi-indices and coefficients for VCJH schemes}\label{app:k_mat}
The following lemma establishes suitable choices of the multi-index set $\set{M}$ in order for the first part of Assumption \ref{asm:k_mat} to be satisfied.
\begin{lemma}\label{lem:k_mat}
For the total-degree polynomial space $\set{P}_p(\hat{\ug{\Omega}})$ on any element type, the choice of $\set{M} \df \{\vc{\alpha} \in \mathbb{N}_0^d : \lvert \vc{\alpha} \rvert = p\}$ ensures that $\mat{K}\mat{D}^{(m)} = \zeromat$ for all $m \in \{1,\ldots, d\}$, irrespective of $\fn{C}(\vc{\alpha})$ and the reference element type. Such a property is ensured for the tensor-product polynomial space $\set{Q}_p(\hat{\ug{\Omega}})$ on $\hat{\ug{\Omega}} = [-1,1]^d$ with the choice of $\set{M} \df \{(p,\ldots,p)\}$.
\end{lemma}
\begin{proof}
Since $\mat{K}\mat{D}^{(m)} = \zeromat$ if $\mat{D}^{(\vc{\alpha})}\mat{D}^{(m)} = \zeromat$ for all $\vc{\alpha} \in \set{M}$, we choose $\set{M}$ such that the image of any function $\fn{v} \in \set{V}$ under the operator $\partial^{\lvert\vc{\alpha}\rvert}/\partial \hat{x}_1^{\alpha_1} \cdots \partial \hat{x}_d^{\alpha_d}$ represented by the matrix $\mat{D}^{(\vc{\alpha})}$ is a constant function. Since differentiating any function $\fn{v} \in \set{P}_p(\hat{\ug{\Omega}})$ a total of $p$ times yields a constant function, and the action of the operator $\partial^{\lvert\vc{\alpha}\rvert}/\partial \hat{x}_1^{\alpha_1} \cdots \partial \hat{x}_d^{\alpha_d}$ consists of applying a total of $\lvert \vc{\alpha} \rvert$ partial derivatives, constraining $\set{M}$ to satisfy $\lvert\vc{\alpha}\rvert = p$ leads to the desired result for the total-degree polynomial space $\set{P}_p(\hat{\ug{\Omega}})$. In the tensor-product case, we may proceed similarly to the proof of Lemma \ref{lem:tensor_sbp} and note that any function $\fn{v} \in \set{Q}_p(\hat{\ug{\Omega}})$ may expanded as $\fn{v}(\hat{\vc{x}}) = \sum_{\vc{\alpha} \in \set{N}}\el{v}^{(\vc{\alpha})}\phi_{\textrm{1D}}^{(\alpha_1)}(\hat{x}_1) \cdots \phi_{\textrm{1D}}^{(\alpha_d)}(\hat{x}_d)$, where it is clear that applying the operator $\partial^p/\partial \hat{x}_m^p$  eliminates the variation of $\fn{v}(\hat{\vc{x}})$ with respect to the $m^{\textrm{th}}$ coordinate. It then follows that differentiating $p$ times with respect to each coordinate yields a constant function, and we can therefore take $\set{M}$ to contain the single multi-index $\vc{\alpha} = (p,\ldots, p)$ in order to obtain $\mat{K}\mat{D}^{(m)} = \zeromat$ for all $m \in \{1,\ldots, d\}$.
\end{proof}
Following Castonguay \etal \cite[\textsection 5.3]{castonguay_vincent_jameson_triangular_fr_11} and Williams \etal \cite[\textsection 4.1]{williams_esfr_adv_diff_tetrahedra_13}, the set of multi-indices $\set{M}$ associated with the space $\set{P}_p(\hat{\ug{\Omega}})$ may be parametrized in terms of $d-1$ scalar indices to obtain
\begin{equation}\label{eq:k_param}
\mat{K}^{\textrm{1D}} \df \frac{c}{\lvert\hat{\ug{\Omega}}\rvert}(\mat{D}^p)^\T\mat{M}\mat{D}^p, \quad
\quad \mat{K}^{\textrm{2D}} \df \frac{c}{\lvert\hat{\ug{\Omega}}\rvert}\sum_{q=0}^p\mqty(p \\ q) (\mat{D}^{(p-q,q)})^\T \mat{M}\mat{D}^{(p-q,q)},
\end{equation}
and 
\begin{equation}
\mat{K}^{\textrm{3D}} \df \frac{c}{\lvert\hat{\ug{\Omega}}\rvert}\sum_{q_1=0}^p\sum_{q_2=0}^{q_1}\mqty(p \\ q_1)\mqty(q_1 \\ q_2) (\mat{D}^{(p-q_1,q_1-q_2,q_2)})^\T \mat{M}\mat{D}^{(p-q_1,q_1-q_2,q_2)},
\end{equation}
defining one-parameter families of symmetric correction fields, where the VCJH parameter $c \in \R$ determines the properties of the resulting schemes. With respect to the formalism of Theorem \ref{thm:vcjh}, we may then take $\fn{C}(\alpha) = c$ for $d=1$, $C(\vc{\alpha}) = c\ \binom{p}{\alpha_2}$ for $d=2$, and $\fn{C}(\vc{\alpha}) = c\ \binom{p}{p-\alpha_1}\binom{p-\alpha_1}{\alpha_3}$ for $d=3$. These choices lead to $\mat{K}$ being symmetric positive-semidefinite for all $c \geq 0$, which is a sufficient condition for $\mat{M}+\mat{K}$ to be SPD, and hence we see from Lemma \ref{lem:k_mat} that Assumption \ref{asm:k_mat} is satisfied under such conditions. While $c \geq 0$ is typically assumed for simplicial elements with $d \geq 2$, it has been shown in the case of $d=1$ (see, for example, Vincent \etal \cite{vincent_esfr_10}, Allaneau and Jameson \cite{allaneau_jameson_dg_fr_11}, and Ranocha \etal \cite{ranocha_sbp_cpr_16}) that there exists a constant $c_- < 0$ depending only on $p$ such that $\mat{M}+\mat{K}$ is SPD for $ c_- < c < \infty$. In the tensor-product case, the choice of $\set{M}$ given in Lemma \ref{lem:k_mat} leads to a sum over a single multi-index $\vc{\alpha} \df (p,\ldots,p)$, corresponding to
\begin{equation}
\mat{K}^{\textrm{TP}} \df \frac{c}{\lvert \hat{\ug{\Omega}} \rvert} \Big((\mat{D}^{(1)})^p \cdots (\mat{D}^{(d)})^p\Big)^\T\mat{M}\Big((\mat{D}^{(1)})^p \cdots (\mat{D}^{(d)})^p\Big).
\end{equation}
It therefore follows from (\ref{eq:sobolev_norm}) that when $\il\fn{u},\fn{v}\ir_{\mat{W}} = \il\fn{u},\fn{v}\ir_{\hat{\ug{\Omega}}}$ for all $\fn{u},\fn{v} \in \set{V}$, the norm in Lemma \ref{lem:norm} may be expressed as
\begin{equation}
\lVert \fn{v} \rVert_{\set{V}_h}^2 = \sum_{\kappa=1}^K \int_{\hat{\ug{\Omega}}}\Bigg(\big(\fn{v}(\fv{X}^{(\kappa)}(\hat{\vc{x}}))\big)^2 + \frac{c}{\lvert\hat{\ug{\Omega}}\rvert}\bigg(\frac{\partial^{d\cdot p}\fn{v}(\fv{X}^{(\kappa)}(\hat{\vc{x}}))}{\partial \hat{x}_1^{p} \cdots \partial \hat{x}_d^{p}}\bigg)^2 \Bigg)\mathcal{J}^{(\kappa)}(\hat{\vc{x}})\, \dd \hat{\vc{x}},
\end{equation}
which for $d=3$ recovers the norm used by Cicchino and Nadarajah \cite{ciccino_tensor_product_new_norm_fr_20} to prove stability for VCJH schemes on affine hexahedral elements.
}
\bibliographystyle{ieeetr}

\end{document}